\documentclass[11pt,a4paper]{article}

\usepackage{amsthm,amsmath,multirow,graphicx,makecell,booktabs,url,mathtools,wrapfig,mathrsfs,bm,relsize,color, enumitem, fullpage,float,bbm,xcolor, verbatim, siunitx, subcaption}
\usepackage[normalem]{ulem}
\usepackage{cancel}
\usepackage{algorithm}
\usepackage{algorithmic}
\usepackage{xb_preamble} 
\RequirePackage[numbers]{natbib} 
\usepackage[colorlinks=true,
linkcolor=blue,
urlcolor=blue,
citecolor=blue]{hyperref}

\usepackage[capitalise]{cleveref}
\numberwithin{equation}{section}
\newtheorem{conj}{Conjecture}
\newtheorem{theorem}[conj]{Theorem}
\newtheorem{cor}[conj]{Corollary}
\newtheorem{prop}[conj]{Proposition}
\newtheorem{lemma}[conj]{Lemma}

\newtheorem{fact}{Fact}

\newtheorem{innercustomgeneric}{\customgenericname}
\providecommand{\customgenericname}{}
\newcommand{\newcustomtheorem}[2]{%
	\newenvironment{#1}[1] 
	{%
		\renewcommand\customgenericname{#2}%
		\renewcommand\theinnercustomgeneric{##1}%
		\innercustomgeneric
	}
	{\endinnercustomgeneric}
}

\newcustomtheorem{customAss}{Assumption}

\theoremstyle{remark}\newtheorem{remark}{Remark}

\def\T{\top}
\def\i{\infty}   
\def\sw{\Sigma}

\def\sT{\Sigma_T}
\def\whsT{\wh\Sigma_T}

\def\whsw{\wh \Sigma} 
\def\Dtmax{\Delta_{\max}}
\def\Dtmin{\Delta_{\min}}
\def\cNeps{\cN_{\epsilon_1,\epsilon_2,\epsilon_3}}
\def\rd{{\rm d}}
 
\def\KL{{\rm KL}}
\def\rI{\textrm{I}}
\def\rII{\textrm{II}}
\def\rIII{\textrm{III}}

\def\pmin{\pi_{\min}^*}
\def\pmax{\pi_{\max}^*}
\def\pminsq{\pi_{\min}^{*2}}
\def\Dt{\Delta}
\def\bpi{\bm \pi}
\def\bmu{\bm \mu}
\def\bsw{\bm \sw}
\def\bh{\bm h}
\def\bwhsw{\bm \whsw}

\title{Convergence and Optimality of the EM Algorithm Under Multi-Component Gaussian Mixture Models}

\author{
    Xin Bing\thanks{Department of 
  Statistical Sciences, University of Toronto. E-mail: \texttt{xin.bing@utoronto.ca}}
  ~~~~~Dehan Kong\thanks{Department of 
  Statistical Sciences, University of Toronto. E-mail: \texttt{dehan.kong@utoronto.ca}.}
  ~~~~~ Bingqing Li\thanks{Department of 
  Statistical Sciences, University of Toronto. E-mail: \texttt{bbingqing.li@mail.utoronto.ca}}
}

\begin{document}

\maketitle


 \begin{abstract}
 	 Gaussian mixture models (GMMs) are fundamental statistical tools for modeling heterogeneous data. Due to the nonconcavity of the likelihood function, the Expectation-Maximization (EM) algorithm is widely used for parameter estimation of each Gaussian component. Existing analyses of the EM algorithm's convergence to the true parameter focus on either the two-component case or multi-component settings with known mixing probabilities and isotropic covariance matrices.
 	 
 	 In this work, we study the convergence of the EM algorithm for multi-component GMMs in full generality. The population-level EM is shown to converge to the true parameter when the smallest separation among all pairs of Gaussian components exceeds a logarithmic factor of the largest separation and the reciprocal of the minimal mixing probabilities. At the sample level, the EM algorithm is shown to be   minimax rate-optimal, up to a logarithmic factor. We develop two distinct novel analytical approaches, each tailored to a different regime of separation, reflecting two complementary perspectives on the use of EM. As a byproduct of our analysis, we show that the EM algorithm, when used for community detection, also achieves the minimax optimal rate of misclustering error under milder separation conditions than spectral clustering and Lloyd's algorithm, an interesting result in its own right. Our analysis allows the number of components, the minimal mixing probabilities, the separation between Gaussian components and the  dimension to grow with the sample size. Simulation studies corroborate our theoretical findings.
 \end{abstract}
{\em Keywords:} Community detection; Clustering; EM algorithm;
Gaussian mixture models; Minimax optimal estimation.

\section{Introduction}

Finite mixture models are fundamental statistical tools for modeling heterogeneous data \citep{lindsay1995mixture,mclachlan1988mixture} and have found wide-ranging applications in fields such as biology, economics, image processing, and natural language processing, to name just a few. These models assume that the observed data arise from a mixture of several latent sub-populations, each governed by its own distribution \citep{mclachlan2000finite, titterington1985statistical}. Among them, the Gaussian mixture model stands out as one of the most extensively studied and applied, due to its mathematical tractability and empirical success in sub-population inference, clustering, density estimation, and outlier detection \citep{dempster1977maximum, redner1984mixture}.

In this paper, we focus on the $d$-dimensional Gaussian location mixture model with $L \ge 2$ components, hereafter referred to as the $L$-GMM. Specifically, let $X_1,\ldots, X_n$ be i.i.d. samples of a random vector $X\in\RR^d$ that follows, for some {\em latent} label $Y \in [L] := \{1, \ldots, L\}$,
\begin{equation}\label{model}
	X \mid Y = \ell ~ \sim ~  \cN_d\left(\mu_\ell, \sw \right),
\end{equation}  
with $\PP(Y = \ell)=\pi_\ell$ for each $\ell \in [L]$.
The Gaussian components have distinct mean vectors $\mu_1, \ldots, \mu_L \in \RR^d$, and share a common covariance matrix $\sw \in \RR^{d \times d}$, which is assumed to be strictly positive definite. The mixing probabilities $\pi = (\pi_1, \ldots, \pi_L)^\T$ are all positive and sum up equal to one. By collecting the mean vectors into the $d\times L$ matrix $M := (\mu_1, \ldots, \mu_L),$ our main interest lies in estimating the parameter 
$\theta := (\pi, M, \sw)$  based on  $X_1,\ldots, X_n$.

In the long history of Gaussian mixture models, the identifiability of $\theta$, up to label switching, under model \eqref{model} is well-known \cite{teicher1963identifiability,yakowitz1968identifiability}. See also \cite{titterington1985statistical} and \cite{lindsay1995mixture}, as well as the recent \cite{allman2009identifiability} for general finite mixture models.  Regarding estimation, a significant line of work analyzes the maximum likelihood estimation (MLE) for density estimation \citep{ghosal2001entropies} and for estimating the mixing measure,
$
\rho = \sum_{\ell=1}^L \pi_\ell   \delta_{\{\mu_\ell, \sw\}},
$ 
with $\delta$ being the Dirac measure
under the exactly specified (in terms of the number of fitted mixture components) and  over-specified settings \citep{chen1995optimal,ho2016convergence,ho2016strong}. However, since the likelihood function under GMMs is non-concave, computing the MLE is challenging, especially in multivariate settings. To overcome this challenge, three main types of alternative approaches have been developed and studied.

\begin{itemize}
	\item  {\em Method of Moments (MoM)-based approaches.} MoM estimation of $\theta$ has been developed in \cite{LindsayBasak,hsu2013learning, belkin2010polynomial, hardt2015tight, moitra2010settling, kalai2010efficiently}. Although these methods require minimal separation between Gaussian components, they typically yield slow convergence rates and suboptimal sample complexity with respect to both the dimension $d$ and the number of components $L$. Recently, \cite{WuYang2020,doss2023optimal}  studied MoM-based estimation of the mixing measure and showed that it achieves the minimax optimal rate under the Wasserstein distance for bounded $L$. While their rates for estimating the mixing measure could be translated into rates for estimating   $\theta$ with improved sample complexity, the resulting convergence is still exponentially slow in $L$, even in the univariate case, and corresponding results in the multivariate setting remain largely unknown.
	
	\item {\em Clustering-based approaches.} Another substantial body of literature focuses on clustering $X_1,\ldots, X_n$ by predicting their labels $Y_1,\ldots, Y_n$ in model \eqref{model}, a task also known as the {\em community detection} problem   \citep{dasgupta1999learning,sanjeev2001learning,vempala2002spectral,kannan2008spectral,awasthi2012improved,achlioptas2005spectral,lu2016statistical,chen2024achieving}. Once the labels are known, the parameter $\theta$ can be readily estimated from the labeled data. However, as we detail in \cref{sec_theory_clustering}, such methods require strong separation conditions between Gaussian components to achieve (nearly) perfect clustering.
	
	\item {\em The Expectation Maximization (EM) algorithm.}  The EM algorithm is designed in a broader context  to approximate the MLE by iteratively maximizing a lower bound of the likelihood function through alternating between the expectation step (E-step) and the maximization step (M-step) \citep{dempster1977maximum}. Under the $L$-GMM, both steps have closed-form expressions, leading to fast computation. In contrast to its empirical success,  EM is only known to converge to a local stationary point in general, which may be far from the true parameter \citep{wu1983convergence, redner1984mixture}.  Over the past two decades, significant progress has been made in the theoretical understanding of EM's convergence to the true parameter under the $L$-GMM.
\end{itemize}

The general positioning of the EM algorithm lies between the first two approaches: compared to MoM-based methods, EM estimation, when suitably initialized, typically requires more separation  but yields faster and, oftentimes minimax optimal, rates of convergence; whereas compared to clustering-based methods, EM requires weaker separation condition to achieve optimal convergence rates. The primary goal of this work is to conduct a thorough theoretical analysis of the EM algorithm to formally demonstrate its adaptivity under the $L$-GMM in the presence of mild separation conditions. As explained below, existing analyses of EM are largely limited either to the case $L = 2$ or to settings where the mixing probabilities $\pi$ are known and the covariance matrix $\sw$ are isotropic (see the summary in \cref{tab_EM_compare}).  Analyzing EM for mixtures with $L \ge 3$ presents significant challenges, as the likelihood landscape deteriorates with increasing $L$. Furthermore, allowing $\pi$ and $\sw$ to be unknown requires not only estimating these quantities but also adjusting the estimation of the mean vectors, as the E-steps (see \cref{sec_EM}) depend on their estimates as well. This work is the first to establish both the convergence of EM to the true parameter and its minimax optimality under the $L$-GMM with $L \ge 2$ and both $\pi$ and $\sw$ unknown.

\subsection{Related literature}

In this section, we review existing theoretical results on the convergence of the EM algorithm under Gaussian mixture models. The early work by \cite{xu1996convergence} establishes a connection between EM iterates and the gradient of the likelihood function under general GMMs. Later, \cite{dasgupta2007probabilistic} analyzed a two-step EM procedure equipped with a specially designed initialization scheme, under GMMs with $L \ge 2$ components and isotropic covariance matrices. Under a strong separation condition between  Gaussian components, the two-step EM estimator is proved to converge to within a small neighborhood of the true parameter, although the size of this neighborhood may not shrink to zero as the sample size increases.

Recently, there has been renewed progress in understanding the convergence of EM to the true parameter under the $L$-GMM. In \cite{EM2017}, the authors develop a general framework for analyzing EM and apply it to the symmetric, isotropic 2-GMM with $\pi_1 = \pi_2 = 1/2$, $\mu_1 = -\mu_2 = \mu$, and $\sw = \bI_d$. They prove that the EM iterates $\wh \mu^{(t)}$ converge to the true parameter $\mu$ in $\ell_2$-norm at the rate $\|\mu\|_2^3 \sqrt{d/n}$, provided that the initialization $\wh\mu^{(0)}$ satisfies $\|\wh\mu^{(0)} - \mu\|_2 \le \|\mu\|_2/4$, and the separation between Gaussian components, $\|\mu\|_2$, exceeds a sufficiently large constant. Further improvements in this symmetric, isotropic $2$-GMM are made in \cite{Xu_Hsu_Maleki} and \cite{daskalakis2017ten}, where the population level EM is shown to converge to the true parameter as long as the separation $\|\mu\|_2 > 0$ and the initialization $\wh \mu^{(0)}$ is not orthogonal to $\mu$. In other words, the population level likelihood function in this special case has no spurious local maxima, so that any random initialization suffices for the sample level EM algorithm to converge to the ground truth as $n \to \infty$. More recently, still under the setting $\mu_1 = -\mu_2 = \mu$ and $\sw = \bI_d$, it has been shown that the EM iterates with simple random initialization converge to the MLE and achieve adaptive minimax optimality for $\|\mu\|_2 = \cO(1)$, with equal mixing weights $\pi_1 = \pi_2 = 1/2$ \citep{wu2021randomly} and for unequal but known weights \citep{weinberger2022algorithm}. Similar guarantees are also obtained for the case where $L = 2$ is specified but the true model consists of a single Gaussian component \citep{dwivedi2020singularity}. These strong global convergence results, however, rely on the assumptions of known $\pi$ and $\sw = \bI_d$. When these conditions are not met, EM must also estimate both $\pi$ and $\sw$, and even for $L=2$, it remains unclear whether global convergence guarantees are possible. For the 2-GMM with both unknown $\pi$ and unknown $\sw$, \cite{cai2019chime} established convergence of the EM algorithm to the true parameter, provided that a suitable initialization condition is met and the separation between Gaussian components is sufficiently large but remains bounded.

Despite of these works on the $2$-GMM, they can not extend to the  $L$-GMM with $L\ge 3$ from several aspects. First, the landscape of the likelihood function for $L \ge 3$ differs drastically from the two-component case. As shown in \cite{jin2016local}, the population-level likelihood function under the $L$-GMM with $L \ge 3$ can exhibit poor local maxima even in the simplest case of equally weighted mixtures of well-separated, spherical Gaussians. Moreover, the EM algorithm, when initialized randomly from the data, converges with high probability to such bad local maxima. It is therefore crucial to establish, even at the population level, conditions under which EM converges to the true parameter. Specifically, one must characterize: (1) the region in which EM contracts; (2) the type of initialization that ensures the iterates remain within this region; and (3) how the required separation between Gaussian components depends on $L$, the minimal mixing proportion $\pi_{\min} = \min_\ell \pi_{\ell}$ as well as  the dimension $d$. Second, these issues must also be addressed at the sample level.  Third, once we allow the aforementioned quantities to diverge as $n\to \i$, both the minimax optimal rate for estimating $\theta$, and whether EM provably achieves this rate, remain open problems.

For such reasons, much fewer results exist on the convergence of the EM algorithm for $L \ge 3$. When $\pi$ is known  and $\sw = \bI_d$, \citet{yan2017convergence} derive the convergence rate for a gradient-based EM algorithm, while \citet{Zhao2020} analyze the standard EM. Their results require that $\min_{k \ne \ell}\|\mu_k - \mu_\ell\|_2$ satisfy a certain separation condition and that the initialization lies within a neighborhood of the true parameter. This separation condition is subsequently relaxed in \cite{kwon2020algorithm,regev2017learning}, and the rate analyses are refined by \cite{segol2021improved}. However, none of these works apply when both $\pi$ and $\sw$ are unknown. Moreover, the convergence rates established in the aforementioned studies are not minimax-optimal, even under known $\pi$ and $\sw$, and can be improved in several respects. Table \ref{tab_EM_compare} summarizes a technical comparison with existing results for $L \ge 2$, while a full table including the $L=2$ case is deferred to Table \ref{tab_EM_compare_full} in Appendix \ref{app_sec_tab}.

\begin{table}[ht]
	\centering
	\caption{Comparison with existing theoretical results of the EM algorithm and its variant under the $L$-GMM with $L \ge 2$. The rates of convergence are in probability and the notation $\lessapprox$ hides multiplicative logarithmic factors.}
	\label{tab_EM_compare}
	\renewcommand{\arraystretch}{1.5}
	\resizebox{\textwidth}{!}{
		\begin{tabular}{l|c|c|c|c|c}
			\toprule
			& Settings & Separation $\Dtmin \gtrsim $ & Initialization & Convergence Rates & Optimality\\
			\midrule
			\cite{dasgupta2007probabilistic} &  $\sw=\sigma^2\bI_d$     & $\displaystyle  \sqrt{d \log(\pi_{\min}^{-1})} + \log(\pi_{\min}^{-1}) $          &  A particular initialization  & $ d(\wh\pi,\pi)+ d(\wh M, M)  \lesssim  \sqrt{d /(n\pi_{\min})} + \exp(-c\sqrt{d})/ \pi_{\min}   $ & $\times$\\
            \cite{kwon2020algorithm}   &  $\sw=\sigma^2\bI_d$                            & $\log(1/\pi_{\min})$         & $d(\pi^{(0)},\pi) < 1,~ d(M^{(0)}, M) \le c \sqrt{\Dtmin}   $ 
            & A sample-splitting version of the EM & $\checkmark$\\
			\cite{yan2017convergence}   &  $\pi$ known, $\sw=\bI_d$                              & $ (d \wedge L)\log(\Dtmax / \pi_{\min}) $         & $ d(M^{(0)}, M) \le c \sqrt{\Dtmin}$ 
            & $\displaystyle d(\wh M, M) \lessapprox (\pi_{\max} / \pi_{\min})\sqrt{\Dtmax}(\sqrt{d} + L^3\Dtmax)\sqrt{d / n}$ & $\times$\\
			\cite{Zhao2020}   &   $\pi$ known, $\sw=\bI_d$                              & $ (d \wedge L)\log(\Dtmax / \pi_{\min})$         & $ d(M^{(0)}, M) \le c \sqrt{\Dtmin}$ 
            & $\displaystyle d(\wh M, M) \lessapprox (\sqrt{\Dtmax} / \pi_{\min})\sqrt{dL / n}  $ & $\times$\\
            \cite{segol2021improved}   &   $\pi$ known, $\sw=\bI_d$                              & $ \log(1/\pi_{\min}) $         & $ d(M^{(0)}, M) \le c \sqrt{\Dtmin}   $ 
            & $\displaystyle d(\wh M, M) \lessapprox  (1 / \pi_{\min}) \sqrt{dL/ n}  $ & $\times$\\
			\midrule 
			\multirow{2}{*}{This work}                   &  \multirow{2}{*}{Full generality}                            &  \multirow{2}{*}{$\log(\Dtmax /\pi_{\min})$}       &      $d(\pi^{(0)},\pi) < 1,~d(M^{(0)}, M)\le c\sqrt{\Dtmin}$   & $\displaystyle d(\wh\pi,\pi) + d(\wh M, M) \lessapprox \sqrt{d / (n\pi_{\min})} $  & \multirow{2}{*}{$\checkmark$} \\
			& & &  $d(\sw^{(0)}, \sw)\le  c $,   or via preliminary labels  & $\displaystyle d(\whsw, \sw) \lessapprox \sqrt{d / n} $ & \\
			\bottomrule
		\end{tabular}
    } 
\end{table} 

\subsection{Our contributions}

We now summarize our main results which include analyses of the EM algorithm at both the population level and the sample level. For the former, we establish a contraction region for EM while for the latter, our analysis is based on two complementary approaches, reflecting different perspectives on EM: parameter estimation and clustering. The first approach is more suitable when the separation is small to moderate, while the second is particularly effective when the separation is moderate to large.   To distinguish between the true model parameters and their iterative estimates, we denote the true parameter by 
$\theta^* = (\pi^*, M^*, \sw^*)$.

\subsubsection{Convergence of the population level EM algorithm}

We review  the population level  EM algorithm under model \eqref{model} in \cref{sec_EM_popu}, and  establish its convergence  to $\theta^*$  in \cref{thm_EM_popu_unknown} (for unknown $\sw^*$)  and in  \cref{cor_EM_popu_known} (for known $\sw^*$) of \cref{sec_theory_popu}.  Denote by $\Dtmin$ (resp. $\Dtmax$) the smallest (resp. largest) squared Mahalanobis distance among all pairs of Gaussian components. When 
\begin{equation}\label{cond_sep_intro}
    \Dtmin ~ \gtrsim~  \log(\Dtmax / \Dtmin) + \log(1/\pmin),
\end{equation}
the population level EM iterates $\theta^{(t)} = (\pi^{(t)},M^{(t)},\sw^{(t)})$ with $t\ge 0$, once initialized within a neighborhood of $\theta^*$, converges:  for all $t\ge 1$ and some contraction rate $\kappa\in(0,1)$,
\begin{equation}\label{convg_EM_intro}
    d(\pi^{(t)}, \pi^*) + d(M^{(t)},M^*) + d(\sw^{(t)},\sw^*) \le \kappa^t \left\{d(\pi^{(0)}, \pi^*) + d(M^{(0)},M^*)+d(\sw^{(0)},\sw^*)\right\}.
\end{equation}
The distances between the parameters in $\theta^{(t)}$ and $\theta^*$ are given in \cref{def_dist_pi_M,dist_sw}. 
The convergence in \eqref{convg_EM_intro} is at least linear as $t\to \i$ and becomes superlinear once $\Dtmin$ increases, since our results show that $\kappa$ decays exponentially in $\Dtmin$ (see, \eqref{def_kappa}). To the best of our knowledge, this is the first convergence result of the population level EM algorithm under the $L$-GMM with $L \ge 3$, unknown $\pi^*$  and unknown $\sw^*$.  Compared to the required separation when $\sw^* = \bI_d$, we only pay the price of a logarithmic factor in $\Dtmax$ for estimating $\sw^*$.  Regarding the initialization,   \cref{thm_EM_popu_unknown} shows that  \eqref{convg_EM_intro} requires 
\begin{equation}\label{init_intro}
    d(\pi^{(0)},\pi^*) < 1,\quad d(M^{(0)},M^*) \le c_\mu\sqrt{\Dtmin},\quad d(\sw^{(0)},\sw^*) \le c_{\sw}
\end{equation}
for some absolute constants $0\le 2c_\mu + \sqrt{2} c_{\sw} < \sqrt{2}-1$.  The existing initialization requirements on $M^{(0)}$  are similar to ours in (\ref{init_intro}), although derived under known $\pi^*$ and $\sw^* = \bI_d$. 
For known $\sw^*$, we have $\sw^{(t)} = \sw^*$ for all $t\ge 0$ in \eqref{convg_EM_intro}  and $c_{\sw}=0$ in \eqref{init_intro}. Moreover, our \cref{cor_EM_popu_known} states that the contraction rate $\kappa$ for known $\sw^*$ becomes faster by a factor of $\Dtmax$.   For unknown $\sw^*$, we note in \cref{sec_EM_popu} that the EM iterates $\sw^{(t)}$ depend on $\EE_{\theta^*}[XX^\T]$, $\pi^{(t)}$ and $M^{(t)}$, which  suggests a simple way to choose $\sw^{(0)}$. For such a choice, \cref{cor_EM_popu_unknown} of \cref{sec_theory_popu} shows that $d(\sw^{(0)}, \sw^*) \le c_{\sw}$ in \eqref{init_intro} is satisfied, provided that the requirements on $d(\pi^{(0)}, \pi^*)$ and $d(M^{(0)}, M^*)$ in \eqref{init_intro} are tightened by a multiplicative factor of $1/\Dtmax$ on the right-hand side.

To the best of our knowledge, this is the first convergence result of the population level EM algorithm under the $L$-GMM with $L \ge 3$, unknown $\pi^*$, and either known or unknown $\sw^*$. As mentioned earlier, for  known $\pi^*$ and $\sw ^*= \bI_d$,  \cite{yan2017convergence,Zhao2020} analyze the population level  EM and its variant under  a  separation condition that is strictly stronger than ours in \eqref{cond_sep_intro}, often substantially so when $L$ and $d$ are moderate to large. Their initialization requirement on $M^{(0)}$  is similar to ours in (\ref{init_intro}), although derived under the stronger separation condition in addition to  known $\pi^*$ and $\sw^* = \bI_d$.

At the technical level, we prove \eqref{convg_EM_intro} by viewing the  M-step of population level EM  as an operator: for all $t\ge 0$,
$$
	\theta^{(t)} ~ \mapsto~  (\bpi(\theta^{(t)}), \bM(\theta^{(t)}),\bsw(\theta^{(t)})) =: 	(\pi^{(t+1)}, M^{(t+1)},\sw^{(t+1)})  = \theta^{(t+1)}.
$$ 
Throughout this paper, we use $\bpi$, $\bM$, and $\bsw$ to denote operators, in order to distinguish them from their vector or matrix counterparts $\pi$, $M$, and $\sw$.
By further using the self-consistency of EM in \cref{app_sec_self_consistency}, that is, $\theta^* = (\bpi(\theta^*), \bM(\theta^*),\bsw(\theta^*))$, bounding $d(\pi^{(t+1)}, \pi^*)$ boils down to controlling  $d(\bpi(\theta^{(t)}), \bpi(\theta^*))$. This, in turn, amounts to establishing the Lipschitz continuity of the operator $\theta \mapsto \bpi(\theta)$, and similarly for  $\theta \mapsto \bM(\theta)$ and $\theta\mapsto \bsw(\theta)$, which constitutes the main challenge in our analysis. Our proof first re-parametrizes $\theta$ by $\omega(\theta) = (\pi, M, J(\theta))$ with $J(\theta) := \sw^{-1}M$ as such reparametrization  alleviates the difficulty of direct manipulation of $\sw^{-1}$. Then  in \cref{thm_EM_samp_omega} of \cref{app_sec_proof_contraction_samp} we establish the Lipschitz continuity of $\omega(\theta) \mapsto \bpi(\omega(\theta))$ and $\omega(\theta) \mapsto \bM(\omega(\theta))$, by quantifying the partial derivatives of $\bpi(\cdot)$ and $\bM(\cdot)$ with respect to $\pi$, $M$, and $J(\theta)$ uniformly for all $\theta$ satisfying \eqref{init_intro} in lieu of  $\theta^{(0)}$.  Finally, combined with the Lipschitz property of $\theta \mapsto J(\theta)$ established in \cref{lem_lip_J_basic}, the conclusion in \eqref{convg_EM_intro} follows. Allowing $L \ge 3$, unknown $\pi^*$ and $\sw^*$, and avoiding strong separation conditions significantly complicates the analysis. We refer to the end of \cref{sec_theory_popu} for a more detailed technical discussion.

\subsubsection{Convergence of the sample level EM algorithm}

The sample level EM algorithm is reviewed in \cref{sec_EM_samp}. We  establish its rate of convergence in \cref{sec_theory_samp}, and show that it achieves the minimax optimal rate under model \eqref{model} and the separation condition \eqref{cond_sep_intro}. 
Due to the dependence of $\Dtmin$ on $\Dtmax$ and $\pmin$ in \eqref{cond_sep_intro}, and hence implicitly on $L$ and $d$, the analysis becomes significantly more challenging when aiming to derive the optimal rate of convergence with respect to all these quantities. To overcome this challenge, we adopt two complementary analytical approaches in \cref{sec_theory_samp_small,sec_theory_samp_large}, each tailored to different regimes depending on the magnitude of  $\Dtmin$ and $\Dtmax$.\\

\noindent{\bf Optimal rate of convergence under small to moderate separation.}  The first approach can be viewed as the sample-analogue of proving the population level result in \eqref{convg_EM_intro} where we directly prove the convergence  to $\theta^*$, of the sample-level EM iterates $\wh \theta^{(t)} = (\wh \pi^{(t)}, \wh M^{(t)},\whsw^{(t)})$, with $t\ge 0$.  For each parameter $\wh h\in \{\wh\pi, \wh M, \whsw\}$, with its sample level operator $\wh\bh$ and population level operator $\bh$, we decompose 
\begin{equation}\label{decomp_intro}
	d(\wh h^{(t+1)}, h^*) ~ = ~ d(\wh \bh (\wh \theta^{(t)}),  \bh(\theta^*) )  ~ \le ~  d(\wh 
	\bh (\wh \theta^{(t)}),  \wh \bh(\theta^*) ) + d(\wh \bh ( \theta^*),  \bh(\theta^*) ).
\end{equation}
We emphasize that this is a different decomposition from the one used in the literature, such as \cite{EM2017,yan2017convergence,cai2019chime}, and is crucial to obtain minimax optimal rates. We refer to the beginning of \cref{sec_theory_samp_small} for detailed explanations. 

In \cref{prop_EM_samp_contra} of \cref{sec_theory_samp_small}, we control the first term on the right-hand side by establishing the Lipschitz properties of the sample level operators $\theta \mapsto \wh\bpi(\theta)$, $\theta \mapsto \wh\bM(\theta)$, and $\theta \mapsto \bwhsw(\theta)$, which readily yield: with probability $1-\cO(n^{-1})$, for all $\theta$ satisfying \eqref{init_intro} in lieu of $\theta^{(0)}$,
\begin{align*}
	d(\wh\bpi(\theta), \wh\bpi(\theta^*)) + 	d(\wh \bM(\theta), \wh \bM(\theta^*)) + d(\bwhsw(\theta),\bwhsw(\theta^*))   \le \kappa_n  \left\{
	d(\pi,\pi^*) + d(M,M^*) +  d(\sw, \sw^*)
	\right\}.
\end{align*} 
Its proof follows the road-map of proving \eqref{convg_EM_intro}. By the reparametrization $\omega(\theta)$, the key is to derive uniform convergence of certain empirical quantities of the form
$ \partial \wh \bpi(\omega(\theta))$,
$\partial \wh \bM(\omega(\theta))$ and
$\partial \bwhsw(\omega(\theta))$,
where the partial derivatives  are taken with respect to $\pi$, $M$, and $J(\theta)$ separately.  The established uniform convergence rate along with the population level contraction factor $\kappa$ in \eqref{convg_EM_intro} determines   $\kappa_n$. To ensure   $\kappa_n < 1$, 
we need 
\begin{equation}\label{Dtmin_intro}
	 \log{\Dtmax \over \Dtmin}+\log{1\over \pmin} ~ \lesssim ~ \Dtmin ~ \lesssim ~   {\Dtmin \over \Dtmax}{n\pmin \over dL^{3/2}\log^2(n)  }
\end{equation}
It is for this reason, this analysis is only suitable for small to moderate separation.  We pause here to clarify that although our analysis significantly relaxes the upper bound restriction on $\Dtmax$ compared to existing results \cite{yan2017convergence,Zhao2020}, the condition imposed in \eqref{Dtmin_intro} remains inevitable, as it is rooted in directly proving the convergence of $\wh \theta^{(t)}$ to $\theta^*$. Nevertheless, as we explain below, this dependence can be avoided by adopting a different analytical approach.

Regarding the second term $d(\wh \bh(\theta^*), \bh(\theta^*))$ in \eqref{decomp_intro}, we derive in \cref{thm_concent} of \cref{sec_theory_samp_small}  concentration inequalities between the  operators $\wh\bpi$, $\wh\bM$, and $\bwhsw$ and their population-level counterparts, evaluated at $\theta^*$. 
By combining the bounds for both terms in \eqref{decomp_intro}, we establish in \cref{thm_EM_samp} of \cref{sec_theory_samp_small} that the EM iterates $\wh\theta^{(t)}$, initialized with $\wh\theta^{(0)}$ satisfying \eqref{init_intro} in lieu of $\theta^{(0)}$, satisfy: with high probability,  for all $t\ge C\log n$,
\begin{equation}\label{rate_intro}
    d(\wh\pi^{(t)}, \pi^*)  + d(\wh M^{(t)}, M^*) 
    \lesssim  \sqrt{d \log n\over n\pmin},\qquad  
    d(\whsw^{(t)},\sw^*)   
    \lesssim    \sqrt{d\log n\over n}.
\end{equation}
For known $\sw^*$,  \cref{thm_EM_samp_known} states that the  first  bound holds under a weaker condition  on $\Dtmax$. 

In \cref{thm_lower_bound} of \cref{sec_theory_optimality}, we establish new minimax lower bounds for estimating $\theta^*$ under model \eqref{model} with the separation condition \eqref{cond_sep_intro}, thereby confirming that the bounds in \eqref{rate_intro} are minimax optimal up to logarithmic factors.\\

\noindent{\bf Optimal rate of convergence under moderate to large separation.} 
Our second approach to analyzing the EM algorithm is more aligned with the clustering perspective, and is tailored to settings with moderate to large separation. Intuitively, when there is substantial separation between Gaussian components, it is easy to cluster data points according to their respective labels, and then estimate the model parameter  from the clustered data.  However, since EM is fundamentally designed to approximately maximize the likelihood for estimating $\theta^*$, its performance for predicting the labels $Y_1,\ldots, Y_n$ remains unknown, even for the case $L=2$.

The main idea of our proof consists of two steps: in the first step we introduce a surrogate loss related to the community detection problem under model \eqref{model} and establish its convergence; in the second step we translate this convergence into a convergence rate for estimating $\theta^*$. The surrogate loss, $\phi(\theta)$ given by \eqref{def_phi_theta}, characterizes the discrepancy between the posterior probabilities $\PP_{\theta}(Y_i = \ell \mid X_i)$ in the E-step at parameter value $\theta$, and the true labels $1\{Y_i = \ell\}$ for all $i \in [n]$ and $\ell \in [L]$. It also captures the degree of separation between the Gaussian components, thereby reflecting the benefit of increased separation. In \cref{thm_conv_phi,thm_conv_phi_known} of \cref{sec_theory_samp_large}, we establish the convergence of $\phi(\wh \theta^{(t)})$ along the EM iterates $\wh \theta^{(t)}$: with probability at least $1-4n^{-1}-\exp(-c\Dtmin)$ for some absolute constant $c>0$,
 \begin{equation}\label{conv_phi_intro}
 	\phi(\wh\theta^{(t)}) ~  \le~   n\exp(-c\Dtmin) + {1\over 2}	\phi(\wh\theta^{(t-1)}),\qquad\forall ~ t\ge 1
 \end{equation}
for both known and unknown $\sw^*$. In addition to the initialization condition in \eqref{init_intro}, the above result admits an alternative requirement on $\wh \theta^{(0)}$ in terms of $\phi(\wh \theta^{(0)})$, which can be linked to preliminary estimates of the labels. On the other hand, unlike the contraction-based argument in our first approach, the convergence result in \eqref{conv_phi_intro} requires only a lower bound on $\Dtmin$ and, in fact, improves with larger values of $\Dtmin$; see, details in \cref{thm_conv_phi,thm_conv_phi_known}.

For the estimation of $\theta^*$, the convergence in \eqref{conv_phi_intro} further implies that the EM iterates $\wh \theta^{(t)}$ themselves converge to $\wt \theta$, which is the MLE of $\theta^*$ under model \eqref{model} when the true labels are given. This, combined with a simple analysis of $\wt \theta$, leads to \cref{thm_param_phi} in \cref{sec_theory_samp_large}, which states that whenever $\Dtmin \gtrsim \log(n)$, with probability at least $1 - \cO(n^{-1})$, the bounds in \eqref{rate_intro} hold for all $\wh \theta^{(t)}$ with $t \gtrsim \log n$. The required sample complexity is
$n \gtrsim d (L^2 / (\pmin \Dtmin)^2 + L^2 + 1/\pmin)$, up to a logarithmic factor. Deriving such dependence on $d$ is nontrivial and relies on a leave-one-out analysis. We refer the reader to the discussion following \cref{thm_conv_phi} for further explanation of the proof technique.

Finally, by combining the results from both approaches, we have established  that the EM iterates satisfy the bounds in \eqref{rate_intro} after $\cO(\log n)$ iterations, and is also minimax optimal up to logarithmic factors, for all $\Dtmin$ satisfying \eqref{cond_sep_intro}. Although our results are established for $L \ge 3$, it is worth noting that they also improve upon existing results in the case of $L = 2$ with unknown $\pi^*$ and $\sw^*$; see, \cref{cor_EM_2GMM} and its subsequent discussion.

\subsubsection{Optimality of the EM for community detection}

Building on our convergence result in \eqref{conv_phi_intro}, we also establish the convergence rate of the EM algorithm's misclustering error in the context of community detection. Specifically, let $\ell(\wh \theta^{(t)})$ denote the proportion of misclustered labels among $Y_1, \ldots, Y_n$ when using the parameter estimate $\wh \theta^{(t)}$ (see \eqref{def_ell_theta} for details). Then, \cref{thm_clustering} in \cref{sec_theory_clustering} shows that, when $\Dtmin \gtrsim \log (1/\pmin) + \log(\Dtmax \wedge \log n)$ and initialized as above, we have, with probability at least $1 - n^{-1} - \exp(-c\Dtmin)$, for all $ t\ge 2\log n$,
\[
\ell(\wh\theta^{(t)} )  ~ \le~    \exp(-c\Dtmin).
\]
This rate is minimax optimal under model \eqref{model}, up to the constant $c$. To the best of our knowledge, this is the first work to establish the rate optimality of the EM algorithm for community detection under the $L$-GMM, including the case $L=2$.

As mentioned earlier, a substantial body of literature has studied the community detection problem under model \eqref{model}, with spectral clustering and Lloyd's algorithm being the most well-studied methods. 
Compared to both approaches, the EM algorithm is rate-optimal while requiring weaker separation conditions. These differences stem from two key distinctions:
(1)  Spectral clustering and Lloyd's algorithm estimate only the labels and the mean vectors $\mu_\ell$'s of the Gaussian components, without directly estimating the covariance matrix $\sw$; 
(2) They rely on {\em hard assignments}, assigning each data point to a single label.
In contrast, EM performs {\em soft assignments} in its E-steps, assigning probabilities to all labels, which are then used to update the estimates of $\mu_\ell$'s and $\sw$ in the subsequent M-step. Our analysis shows that, while estimating $\sw$ requires a stronger initialization condition, it enables EM to achieve the optimal misclustering rate. Moreover, the soft-assignment nature of EM allows it to attain this optimal rate under milder separation conditions than those required by Lloyd's algorithm. We refer the reader to \cref{sec_theory_clustering} for a detailed discussion.\\

This paper is organized as follows: In \cref{sec_EM} we review both the population level and sample level EM algorithms. Convergence of the population level EM is established in \cref{sec_theory_popu}. We analyze the convergence of the sample level EM in \cref{sec_theory_samp_small} for small to moderate separation, and in \cref{sec_theory_samp_large} for moderate to large separation. Minimax lower bounds for parameter estimation are presented in \cref{sec_theory_optimality}, while the application of our EM analysis to quantifying its performance in community detection is discussed in \cref{sec_theory_clustering}. We give detailed discussion on practical choices of the initialization  in \cref{sec_init} and conduct numerical studies to corroborate our theoretical findings in \cref{sec_sims}. All the proofs are deferred to the Appendix.

\paragraph{Notation.}   For any vector $x \in \RR^p$, denote its $\ell_q$ norm as $\|x\|_q$.  For any $p\times p$ symmetric and invertible matrix $\sw$, we write 
$\|x\|_{\sw}^2 = x^\T \sw^{-1} x$. 
For any matrix $A$, we write $\|A\|_{\op} $ and $\|A\|_F$ as  its operator norm and Frobenius norm. Further write $\|A\|_{2,\infty} = \max_{j} \|A_{j\cdot}\|_2$.
For any symmetric, positive semi-definite matrix $Q \in \mathbb{R}^{p \times p}$, we use $\lambda_1(Q) \geq \lambda_2(Q) \geq \dots \geq \lambda_p(Q)$ to denote its eigenvalues. We write $c \prec Q \prec C$ if $c\le\lambda_p(Q) \le \lambda_1(Q) \le C$. For any two sequences $a_n$ and $b_n$, we write $a_n \lesssim b_n$ if there exists some constant $C$ such that $a_n \leq Cb_n$. The notation $a_n \asymp b_n$ stands for $a_n \lesssim b_n$ and $b_n \lesssim a_n$. For any $a,b\in \RR$, we write $a \wedge b = \min\{a, b\}$ and $a \vee b = \max\{a, b\}$. We use $\bI_d$ to denote the $d \times d$ identity matrix. 
Throughout this paper, we use $c,c', C,C'$ to denote positive and finite absolute constants that, unless otherwise indicated, can vary from line to line.

\section{The EM algorithm under the $L$-GMM}\label{sec_EM}

In this section we review both the population level and the sample level EM algorithms under  model \eqref{model}. Recall that $\theta^*=(\pi^*, M^*,\sw^*)$ denotes the true parameter.

\subsection{The population level EM algorithm}\label{sec_EM_popu}

Under model \eqref{model}, the log-likelihood of $\theta$ at any $X=x\in \RR^d$ is
$$
     \log f(\theta \mid x) \\
    \propto  -\frac{1}{2} \log|\sw| + \log{\left\{\sum_{\ell=1}^L \pi_\ell \exp{\left[ -\frac{1}{2}(x-\mu_\ell)^\T \sw^{-1} (x-\mu_\ell)\right]}\right\}}
$$
which is non-concave in $\theta$. 
The EM algorithm iteratively maximizes a lower bound of the log-likelihood function, which is known as the surrogate 
$Q$-function.

To state its definition, we first recall model \eqref{model}.
When $Y$ were observable, the {\em complete} log-likelihood of $\theta$ at $X= x$ and $Y = y$  is 
$$
     \log f(\theta \mid x,y) 
    \propto - \sum_{\ell=1}^L 1\{y=\ell\}\left\{(x-\mu_\ell )^\T \sw^{-1} (x-\mu_\ell) - 2\log(\pi_\ell) +   \log|\sw| \right\},
$$
which is a concave function in $\theta$. Given  any  parameter $\theta$, the surrogate $Q$-function simply replaces the indicator $ 1\{y=\ell\}$ in the above display with 
\begin{align}\label{def_gamma} 
    \gamma_\ell(x; \theta)   &:=~  \PP_{\theta}(Y=\ell \mid X = x) ~ =~  \frac{\pi_\ell}{ \sum_{k=1}^L \pi_k \exp\left\{(x-{1\over 2}( \mu_\ell+\mu_k))^\T \sw^{-1} (\mu_k-\mu_\ell)\right\}}.
\end{align} 
Concretely, given any parameter value $\theta'$, the {\em population level} $Q$-function  is 
\begin{equation}\label{pop_Q}
   Q(\theta  \mid \theta^\prime)= - \EE_{\theta^*} \left\{\sum_{\ell=1}^{L} \gamma_\ell(X;\theta^\prime)\left[ (X-\mu_\ell)^\T \sw^{-1}(X-\mu_\ell) - 2 \log( \pi_\ell) +  \log|\sw| \right] \right\}.
\end{equation} 
For any $t\ge 0$ with the current parameter $\theta^{(t)}$, the E-step is to evaluate  $\gamma_\ell(x;\theta^{(t)})$ in  \eqref{def_gamma} while the M-step is to maximize $Q(\theta \mid \theta^{(t)})$ in \eqref{pop_Q} with respect to its first argument $\theta$. The latter turns out to have a closed-form solution:  for all $\ell \in [L]$,
\begin{align}\label{iter_pi}
    &\pi_\ell^{(t+1)} := \bpi_\ell(\theta^{(t)})~ =   ~ \EE_{\theta^*}[\gamma_\ell(X; \theta^{(t)})],\\ \label{iter_mu}
    &\mu_\ell^{(t+1)} := \bmu_\ell(\theta^{(t)}) ~  =  ~  \frac{\EE_{\theta^*} [\gamma_\ell(X; \theta^{(t)}) X]}{\EE_{\theta^*}[\gamma_\ell(X; \theta^{(t)})]}  
\end{align}
and, by collecting  $\pi^{(t+1)} = (\pi_1^{(t+1)},\ldots,\pi_L^{(t+1)})^\T$ and $M^{(t+1)} = (\mu_1^{(t+1)},\ldots,\mu_L^{(t+1)})$, 
\begin{align}\label{iter_sw}
     \sw^{(t+1)}  :=   \bsw(\theta^{(t)}) &=   \EE_{\theta^*}\left[\sum_{\ell=1}^L \gamma_\ell(X;\theta^{(t)})\bigl(X-\mu_\ell^{(t+1)}\bigr)\bigl(X-\mu_\ell^{(t+1)}\bigr)^\T \right]\\\label{iter_sw_alter}
     &= \sT^* - M^{(t+1)}\diag(\pi^{(t+1)}) M^{(t+1)\T}.
\end{align}
Here we write $\sT^*:=\EE_{\theta^*}[XX^\T]$ and $\diag(\pi^{(t+1)})$ denotes the diagonal matrix with diagonal elements equal to $\pi^{(t+1)}$.
In \cref{iter_pi,iter_mu}, we use bold fonts $\bpi_\ell(\cdot)$, $\bmu_\ell(\cdot)$, and $\bsw(\cdot)$ to denote the operators used in the population level EM algorithm.
The updated parameter  $\theta^{(t+1)} = (\pi^{(t+1)}, M^{(t+1)}, \sw^{(t+1)})$ is subsequently used to compute the E-step in the next iteration. 
The population level EM algorithm  keeps iterating between the E-step in \eqref{def_gamma} and the M-step in \cref{iter_pi,iter_mu,iter_sw} until convergence.

\subsection{The  sample level EM algorithm}\label{sec_EM_samp}
When we have access to $X_1, X_2, \ldots, X_n$ from \eqref{model},
the sample level EM maximizes the following  {\em sample level} $Q$-function: for any given $\theta'$,
\begin{equation}\label{samp_Q}
   Q_n(\theta  \mid \theta^\prime)= - \EE_n \left\{\sum_{\ell=1}^{L} \gamma_\ell(X;\theta^\prime)\Bigl[(X-\mu_\ell)^\T \sw^{-1}(X-\mu_\ell) - 2  \log( \pi_\ell) + \log|\sw|  \Bigr] \right\}
\end{equation} 
where the expectation $\EE_{\theta^*}$  in \eqref{pop_Q} is replaced by the empirical average $\EE_n$ over $X_1,\ldots, X_n$. Correspondingly,  for any $t\ge 0$ and given its current estimate $\wh\theta^{(t)}$,  the E-step at the $(t+1)$th iteration  computes $\gamma_\ell(X_i; \wh\theta^{(t)})$ as in \eqref{def_gamma}, for each $i\in [n]$, while the M-step maximizes $Q_n(\theta \mid \wh\theta^{(t)})$ in \eqref{samp_Q} over $\theta$, yielding that: for all $\ell \in [L]$,
\begin{align}\label{iter_pi_hat}
    \wh\pi^{(t+1)}_\ell &:= ~ \wh\bpi_\ell(\wh \theta^{(t)}) ~ :=~  \EE_n[\gamma_\ell(X; \wh\theta^{(t)})],\\ \label{iter_mu_hat}
    \wh\mu^{(t+1)}_\ell &:= ~ \wh\bmu_\ell(\wh \theta^{(t)}) ~ :=~  \frac{\EE_n[\gamma_\ell(X; \wh\theta^{(t)})   X]}{\EE_n[\gamma_\ell(X; \wh\theta^{(t)})]}
\end{align}
and, with $\whsT = n^{-1}\sum_{i=1}^n X_iX_i^\T$,
\begin{align}\label{iter_sw_hat}
    \whsw^{(t+1)} :=\bwhsw(\theta^{(t)})&=  \EE_n\left[
        \sum_{\ell = 1}^L \gamma_\ell(X; \wh\theta^{(t)}) (X - \wh\mu^{(t+1)}_\ell)(X - \wh\mu^{(t+1)}_\ell)^\T
    \right]\\\label{iter_sw_hat_alter}
    &= \whsT  - \wh M^{(t+1)} \diag(\wh\pi^{(t+1)}) \wh M^{(t+1)\T}.
\end{align}
Analogous to $\bpi_\ell(\cdot)$, $\bmu_\ell(\cdot)$ and $\bsw(\cdot)$ in \cref{iter_pi,iter_mu,iter_sw},  $\wh \bpi_\ell(\cdot)$, $\wh\bmu_\ell(\cdot)$ and $\bwhsw(\cdot)$ denote their sample analogue. 
The updated parameter $\wh \theta^{(t+1)} = (\wh\pi^{(t+1)}, \wh M^{(t+1)}, \whsw^{(t+1)}$) is subsequently used in the next iteration to compute the E-step in \eqref{def_gamma} with $\theta$ replaced by $\wh\theta^{(t+1)}$.
The sample level EM algorithm keeps operating  between the E-step and M-step given above until convergence.

One practical consideration is that computing $\gamma_\ell(X_i; \wh\theta^{(t)})$ in \eqref{def_gamma} requires to invert $\whsw^{(t)}$ for each iteration $t\ge 0$. Using the expression in \eqref{iter_sw_hat_alter}, the following lemma shows that  $\whsw^{(t)}$, for all $t\ge 0$, are invertible as long as $\whsT$ is so.
Its proof is stated in \cref{app_sec_proof_lem_inv_sw_hat}.

\begin{lemma}\label{lem_inv_sw_hat}
    Provided that $\whsT$ is invertible, the matrices $\whsw^{(t+1)}$ in \eqref{iter_sw_hat} for all $t\ge 0$ are invertible. 
\end{lemma}

We remark that the singularity of $\whsT$ can be easily checked, and is guaranteed almost surely  when $n\ge d$.

\section{Convergence of the population level EM algorithm}\label{sec_theory_popu}

In this section, we analyze the population level EM and establish its convergence to the true parameter.   
Denote the smallest and largest  squared Mahalanobis distance between any two Gaussian components by 
\begin{equation}\label{Delta_max_min} 
        \Dtmin :=  \min_{k\ne \ell}  \left\|\mu_k^*-\mu_\ell^*\right\|_{\sw^*}^2,\qquad  \Dtmax  :=  \max_{k, \ell}  \left\|\mu_k^*-\mu_\ell^*\right\|_{\sw^*}^2. 
\end{equation}
The smallest separation, $\Dtmin$, is known as a key quantity of learning Gaussian location mixtures. Regarding $\Dtmax$, we have $\Dtmax = \Dtmin$ when $L = 2$. For $L \ge 3$, \cite{jin2016local} finds that local search procedures may fail to find the global optimum of the likelihood function when one mixture component is significantly farther from the others. This suggests that $\Dtmax$, or more precisely the ratio $\Dtmax/\Dtmin$, influences the difficulty of learning Gaussian mixtures. Another well-known quantity that affects the learning is 
 the smallest mixing probabilities $\pmin = \min_{\ell\in [L]} \pi_\ell^*.$ Our results below characterize the effect of $\Dtmin$, $\Dtmax /\Dtmin$ and $1/\pmin$ on the convergence of the EM algorithm. Finally, since we can always subtract the marginal mean of $X$, we assume $\EE[X] =\sum_{\ell=1}^L \pi_\ell^* \mu_\ell^* = 0$ throughout our analysis.


The following theorem states the convergence rate of the population level EM iterates, $\theta^{(t)} = (\pi^{(t)}, M^{(t)}, \sw^{(t)})$, given in \cref{iter_pi,iter_mu,iter_sw}. For any  $\theta = (\pi, M, \sw)$ and $\theta' = (\pi', M', \sw')$, we define the distances between their mixing probabilities and mean vectors as 
\begin{align}\label{def_dist_pi_M}
     d(\pi, \pi') = \max_{\ell\in [L]} ~ {1\over \pi_\ell^*}|\pi_\ell - \pi'_\ell|,\qquad 
     d(M, M') =  \max_{\ell\in[L]}\left\|\mu_\ell - \mu_\ell'\right\|_{\sw^*}
\end{align} 
and the distance between their covariance matrices as  
\begin{equation}\label{dist_sw}
d(\sw, \sw') = \|\sw^{*-1/2}(\sw - \sw')\sw^{*-1/2}\|_\op.
\end{equation}

\begin{theorem}\label{thm_EM_popu_unknown}
    Assume there exists some large constant $C>0$ such that
    \begin{equation}\label{cond_min_Delta} 
           \Dtmin \ge C \left(\log{\Dtmax \over \Dtmin}+\log{1\over \pmin}\right).
    \end{equation}
    Further assume the initialization $\theta^{(0)} = (\pi^{(0)}, M^{(0)}, \sw^{(0)})$ satisfies 
    \begin{align}\label{cond_init_unknown}
        &  d(\pi^{(0)}, \pi^*)  \le  {1\over 2},\quad  d(M^{(0)},M^*)     \le  c_\mu \sqrt{\Dtmin},\quad d(\sw^{(0)},\sw^*)     \le   c_\sw
    \end{align}
    for  some absolute positive constants $0\le 2c_{\mu} + \sqrt{2} c_\sw < \sqrt{2}-1$.  There exists  some $\kappa  \in (0,1)$ such that $\theta^{(t)} = (\pi^{(t)}, M^{(t)}, \sw^{(t)})$ for any $t\ge 0$, given by \cref{iter_pi,iter_mu,iter_sw}, satisfy
    \begin{align*} 
       d(\pi^{(t)}, \pi^* ) + d(M^{(t)},  M^*) +  d(\sw^{(t)},\sw^*) ~ &\le  ~  \kappa^t \left\{ d(\pi^{(0)},\pi^*) +    d(M^{(0)},M^*) +  d(\sw^{(0)},\sw^*)\right\}.
    \end{align*}    
\end{theorem}
\begin{proof}
    The proof is given in \cref{app_sec_proof_thm_EM_population}.
\end{proof}

\cref{thm_EM_popu_unknown} guarantees linear convergence of the population level EM iterates. In fact, our proof reveals that the contraction rate $\kappa$ satisfies
\begin{equation}\label{def_kappa}
     \kappa = {\Dtmax^2 \over \pmin} \exp(-c\Dtmin),
\end{equation}
for some absolute constant $c>0$ depending only on $c_\mu$ and $c_\sw$ in \eqref{cond_init_unknown}. This  implies the convergence rate of the EM algorithm can be superlinear, since $\kappa$ decays exponentially when $\Dtmin \equiv \Dtmin(d) \to \infty$ as $d = d(n) \to \infty$.  

As highlighted in the Introduction, this is the first convergence result of the population level EM iterates for the $L$-GMM with $L\ge 2$ under full generality. In particular, it covers the case when $\sw^*$ is known, in which case  the EM algorithm only updates $\pi^{(t)}$ and $M^{(t)}$ in the M-step as in \cref{iter_pi,iter_mu} while computes the E-step in \eqref{def_gamma} using  $\theta^{(t)} = (\pi^{(t)}, M^{(t)}, \sw^*)$. 
The following corollary states a faster convergence rate of such EM iterates. 
\begin{cor}[Known $\sw^*$]\label{cor_EM_popu_known}
    Grant condition \eqref{cond_min_Delta}. Assume the initialization $\pi^{(0)}$ and $M^{(0)}$ satisfy \eqref{cond_init_unknown} with $c_\sw = 0$. There exists some $\kappa \in (0,1)$ given by \eqref{def_kappa} such that  $\theta^{(t)} = (\pi^{(t)},M^{(t)}, \sw^*)$ for all $t\ge 0$, computed in  \cref{iter_pi,iter_mu}, satisfy
    \begin{align*} 
       d( \pi^{(t)}, \pi^* )  + d(M^{(t)},  M^*)  &~ \le  ~  (\kappa/\Dtmax)^t    \left\{ d(\pi^{(0)},\pi^*) +    d(M^{(0)},M^*)\right\}.
    \end{align*} 
\end{cor}

\begin{remark}[Separation requirement]
    Condition \eqref{cond_min_Delta} states the required separation condition between Gaussian components for the population level EM iterates to converge. It characterizes the effect of both the smallest mixing probability $\pmin$ and the ratio between the largest and smallest separations among Gaussian components, $\Dtmax / \Dtmin$. We emphasize that such dependence is particularly informative for $L \ge 3$. In the case of $L = 2$, we have $\Dtmax = \Dtmin$, and it is standard to assume that $\pi_1, \pi_2 \in (c, 1-c)$ for some $c\in (0,1/2)$, under which condition \eqref{cond_min_Delta} reduces to $\Dtmin \ge C$. This is assumed in \cite{EM2017} for symmetric, isotropic setting, and  in \cite{cai2019chime} when $\pi^*$ and $\sw^*$ are unknown. For $L \ge 3$, even when $\Dtmax \asymp \Dtmin$ and $\pmin \asymp 1/L$, condition \eqref{cond_min_Delta} simplifies to $\Dtmin \ge C \log(L)$, still reflecting the effect of  $L$ on the required separation. Such logarithmic dependence on $L$ is known to be necessary for achieving polynomial sample complexity guarantees in the isotropic setting \citep{regev2017learning}. Our condition in \eqref{cond_min_Delta} further shows that when $\sw^*$ is unknown, the additional price to pay is only a logarithmic factor in $\Dtmax$. We refer to \cref{sec_theory_clustering} for comparison with the separation conditions of  clustering-based algorithms.
\end{remark}
\begin{remark}[Initialization]
    The requirement in \eqref{cond_init_unknown}  on $M^{(0)}$ becomes milder as the separation $\Dtmin$ increases.  This is consistent with existing results \cite{EM2017} for  $L = 2$, and \cite{yan2017convergence,Zhao2020,segol2021improved,kwon2020algorithm} for $L\ge 3$. Compared to the latter, their initialization requirement is comparable to ours, although derived under  known / unknown $\pi^*$ and $\sw^*=\bI_d$.
    Since we allow both $\pi^*$ and $\sw^*$ unknown, our initialization requirement in \eqref{cond_init_unknown} also depends on $\pi^{(0)}$ and $\sw^{(0)}$. The number $1/2$ in $d(\pi^{(0)},\pi)$ can be replaced by any nonnegative value less than   $1$. 
    Later in \cref{sec_theory_samp_large} we introduce an alternative initialization scheme that also enjoys provable guarantees. Detailed discussion on  practical choices of initialization is given in \cref{sec_init}.
\end{remark}

In the following we mention one simple choice of $\sw^{(0)}$ when $\sw^*$ is unknown. Recall that the EM iterates in \cref{iter_pi,iter_mu,iter_sw} update $\sw^{(t)}$ by using $\sT^*$,  $\pi^{(t)}$ and $M^{(t)}$ via \eqref{iter_sw_alter}. This suggests using
\begin{equation}\label{def_init_sw}
\sw^{(0)} = \sT^* - M^{(0)} \diag(\pi^{(0)}) M^{(0)\T}
\end{equation}
as the initialization for $\sw^*$, once $\pi^{(0)}$ and $M^{(0)}$ are available. 
The following corollary states the convergence of the  EM algorithm for such choice.  

\begin{cor}[A special initialization of   $\sw^*$]\label{cor_EM_popu_unknown}
    Grant condition \eqref{cond_min_Delta}. Assume that the initialization $\theta^{(0)} = (\pi^{(0)},M^{(0)}, \sw^{(0)})$, with $\sw^{(0)}$ given in \eqref{def_init_sw}, satisfies 
    \begin{align}\label{cond_init_unknown_special}
        &  d(\pi^{(0)}, \pi^*) ~ \le ~ {c_0\over \Dtmax},\qquad   d(M^{(0)},M^*)  ~  \le  ~  {c_0\over \sqrt{\Dtmax}},
    \end{align}
    for some small, absolute constant $c_0>0$, There exists some $\kappa \in (0,1)$ given by \eqref{def_kappa} such that  $\theta^{(t)} = (\pi^{(t)},M^{(t)}, \sw^{(t)})$ for all $t\ge 0$, computed in  \cref{iter_pi,iter_mu,iter_sw}, satisfy
    \begin{align*} 
       d( \pi^{(t)}, \pi^* )  + d(M^{(t)},  M^*) + d(\sw^{(t)},\sw^*) &~ \le  ~  \kappa^t    \left\{ d(\pi^{(0)},\pi^*) +    d(M^{(0)},M^*)\right\}.
    \end{align*} 
\end{cor}
\begin{proof}
   In \eqref{lip_sw} of \cref{app_sec_proof_EM_popu} we prove that for any $\sw = \sT^* - M \diag(\pi) M^\T$ with $d(M,M^*)\le \sqrt{\Dtmax}$,  one has
    \begin{equation}\label{eq_lip_sw}
         d(\sw, \sw^*)~ \le ~   2\Dtmax~  d(\pi,\pi^*)+3\sqrt{\Dtmax}~  d(M,M^*).
    \end{equation} 
    \cref{cor_EM_popu_unknown} follows by invoking \eqref{eq_lip_sw} for $\sw^{(0)}$ and noting that \eqref{cond_init_unknown} holds under \eqref{cond_init_unknown_special}.
\end{proof}

 In order to ensure the choice of $\sw^{(0)}$ in \eqref{def_init_sw} satisfying \eqref{cond_init_unknown}, the requirement on $\pi^{(0)}$ and $M^{(0)}$  gets more restrictive than \eqref{cond_init_unknown} by a multiplicative factor of $1/\Dtmax$.

\paragraph{Proof sketch of \cref{thm_EM_popu_unknown}.}
    The full proof can be found in \cref{app_sec_proof_thm_EM_population} and we only offer its sketch  here. For any $\theta^{(t)}=(\pi^{(t)}, M^{(t)}, \sw^{(t)})$ and any $\ell \in [L]$, with $\pi_\ell^{(t+1)}$ and $\mu_\ell^{(t+1)}$ defined in  \cref{iter_pi,iter_mu},  we need to analyze the differences  
    \begin{align*}
        \pi_\ell^{(t+1)}-\pi^*_\ell &= \bpi_\ell(\theta^{(t)})-\bpi_\ell(\theta^*)= \EE_{\theta^*}[\gamma_\ell(X; \theta^{(t)})] - \EE_{\theta^*}[\gamma_\ell(X; \theta^*)]\\
        \mu_\ell^{(t+1)}-\mu^*_\ell &= \bmu_\ell(\theta^{(t)})-\bmu_\ell(\theta^*)= {\EE_{\theta^*}[\gamma_\ell(X; \theta^{(t)}) X]\over \EE_{\theta^*}[\gamma_\ell(X; \theta^{(t)})]} - {\EE_{\theta^*}[\gamma_\ell(X; \theta^*)X]\over \EE_{\theta^*}[\gamma_\ell(X; \theta^*)]}
    \end{align*} 
    where the first equalities in both lines use  the so-called {\em self-consistency} of the $Q$ function in \eqref{pop_Q}, that is, $\theta^* = \argmax_\theta Q(\theta \mid \theta^*)$ such that $\bpi_\ell(\theta^*) = \pi_\ell^*$ and $\bmu_\ell(\theta^*) = \mu_\ell^*$. See, details in \cref{app_sec_self_consistency}. 
    The first step of our analysis is to re-parametrize the function $\gamma_\ell(x; \theta)$ in \eqref{def_gamma} by using $\omega(\theta) = (\pi, M, J(\theta))$ with $J(\theta) := \sw^{-1}M$ so that
    \[
        \gamma_\ell(x; \theta) = \frac{\pi_\ell}{ \sum_{k=1}^L \pi_k \exp\left\{(x-{1\over 2}( \mu_\ell+\mu_k))^\T J(\theta)(\be_k-\be_\ell)\right\}} = \gamma_\ell(x; \omega(\theta))
    \]
    as well as $\bpi_\ell(\theta) = \bpi_\ell(\omega(\theta))$ and $\bmu_\ell(\theta) = \bmu_\ell(\omega(\theta))$.
    This reparametrization avoids direct manipulation of $\sw^{-1}$ which alleviates the difficulties of subsequent analysis.
    In \cref{thm_EM_population_omega} of \cref{app_sec_proof_thm_EM_population} we establish the following Lipschitz continuity of the map $\omega(\theta) \mapsto \bpi_\ell(\omega(\theta))$ with the Lipschitz constant $\kappa_w \asymp  (1/\pmin) \exp{(-c\Dtmin)}$:
    \begin{align}\label{eq_contraction} 
       {1\over \pi_\ell^*} \left|\bpi_\ell(\omega(\theta))-\bpi_\ell(\omega(\theta^*))\right|
        & ~ \le~ \kappa_w  \sqrt{\Dtmax} \bigl\{d(\pi,\pi^*) +  d(M,M^*) +      d(J(\theta),J(\theta^*))\bigr\}
    \end{align}
    for all $\omega(\theta)$ within a neighborhood of  $\omega(\theta^*)$,  
    in particular, for $\omega(\theta^{(t)})$. We refer to \eqref{def_dist_J} the definition of $d(J(\theta), J(\theta^*))$.
    The same bound in \eqref{eq_contraction}, up to a multiplicative factor $\sqrt{\Dtmax}$, is also shown to hold for $\|\bmu_\ell(\theta) -\bmu_\ell(\theta^*)\|_{\sw^*}$. 
    In \cref{lem_lip_J_basic} of \cref{app_sec_proof_thm_EM_population} we further derive the following Lipschitz property of $\theta\mapsto J(\theta)$ for any $\theta$ satisfying \eqref{cond_init_unknown}:
    \begin{equation}\label{eq_lip_J}
        d(J(\theta), J(\theta^*)) ~ \le ~ 2d(M,M^*) + \sqrt{\Dtmax}~ d(\sw,\sw^*),
    \end{equation}
    which in conjunction with \eqref{eq_contraction} gives  
    $$
        d(\pi^{(t+1)},\pi^*) + d(M^{(t+1)},M^*) \le  \kappa_w \Dtmax \left\{
            d(\pi^{(t)},\pi^*) + 3d(M^{(t)},M^*) +   \sqrt{\Dtmax}~ d(\sw^{(t)},\sw^*)
        \right\}.
    $$
    Combining with \eqref{eq_lip_sw}, \cref{thm_EM_popu_unknown} thus follows.  

   From the technical perspective, establishing the Lipschitz continuity in \eqref{eq_contraction} for both $\bpi_\ell(\omega(\theta)) $ and $\bmu_\ell(\omega(\theta))$ is the most challenging part. It requires quantifying the partial derivatives of $\gamma_\ell(x; \omega(\theta))$ with respect to $\pi$, $M$, and $J(\theta)$ uniformly over a neighborhood of $\omega(\theta^*)$. 
   The proofs of \eqref{eq_contraction} are provided in \cref{app_sec_proof_rate_contract_pi_w_J} and \cref{app_sec_proof_rate_contract_mu_w_J}, along with the supporting technical lemmas in \cref{lem_delta_J_order,lem_I_fN,lem_var_gamma_X} of \cref{app_sec_tech_lemmas_EM_popu}.

\section{Theoretical analysis of the sample level EM algorithm}\label{sec_theory_samp}

We present our theoretical guarantees of the sample level EM algorithm in this section. The analysis is based on two different approaches of analyzing the EM depending on the magnitude of  $\Dtmin$. When $\Dtmin$ is small to moderate, as considered in \cref{sec_theory_samp_small}, we directly characterize the convergence of the EM iterates in \cref{iter_pi_hat,iter_mu_hat,iter_sw_hat}, to the true parameter. On the other hand, when $\Dtmin$ is moderate to large, our analysis in \cref{sec_theory_samp_large} first establishes the convergence of a surrogate loss related to community detection under model \eqref{model}, and then translates this into a convergence between the EM iterates and $\theta^*$.

\subsection{Convergence of the EM algorithm under small to moderate separation}\label{sec_theory_samp_small}

In this section we directly analyze the convergence to $\theta^*$ of the EM iterates, $\wh\theta^{(t)} = (\wh \pi^{(t)}, \wh M^{(t)}, \whsw^{(t)})$, for $t\ge 0$, given by \cref{iter_pi_hat,iter_mu_hat,iter_sw_hat}. For each $h\in \{\pi, M, \sw\}$, we decompose 
\begin{align}\label{decomp_h_diff}
	\wh h^{(t+1)} - h^*  ~ =~  \wh \bh(\wh\theta^{(t)})- \bh(\theta^*)~ =~  
	\left( \wh \bh(\wh\theta^{(t)})- \wh \bh(\theta^*)\right)+\left(\wh \bh(\theta^*) - \bh(\theta^*)\right).
\end{align}
The first equality uses the self-consistency of population level EM. 
We analyze  $ \wh \bh(\wh\theta^{(t)})- \wh \bh(\theta^*)$ in \cref{sec_theory_samp_conv} by  establishing the Lipschitz continuity of the operator $\wh \bh \in \{\wh \bpi, \wh \bM, \bwhsw\}$. The second term $\wh \bh(\theta^*) - \bh(\theta^*)$ is analyzed in \cref{sec_theory_samp_rate} where we derive concentration inequalities between the sample level and population level M-steps evaluated at   $\theta^*$.
 
It is worth clarifying that the decomposition in \eqref{decomp_h_diff} differs from that introduced by \cite{EM2017} and subsequently followed in \cite{cai2019chime,yan2017convergence,Zhao2020,segol2021improved}. Specifically, their approach decomposes
\[
	\wh h^{(t+1)} - h^*  ~ =~   
	\left( \wh \bh(\wh\theta^{(t)})-   \bh(\wh\theta^{(t)})\right)+\left( \bh(\wh\theta^{(t)}) - \bh(\theta^*)\right)
\]
where the second term is relatively easy to handle using the population level contraction property of $\bh$ in \cref{sec_theory_popu}. The first term, on the other hand, characterizes the statistical error, which involves analyzing the empirical process
$
	\sup_{\theta} \|\wh \bh(\theta)-   \bh(\theta)\|.
$
This is a challenging task, especially when aiming for convergence rates with optimal dependence on the parameters $n$, $d$, $L$, $\pmin$, $\Dtmin$, and $\Dtmax$. In contrast, our decomposition in (\ref{decomp_h_diff}) only requires establishing that $\sup_{\theta} \|\wh \bh(\theta)-   \wh \bh(\theta^*)\| < 1$, which is a relatively easier task to accomplish. Moreover, under (\ref{decomp_h_diff}), the statistical error is governed by the pointwise convergence $\wh\bh(\theta^*) - \bh(\theta^*)$, for which the optimal convergence rate is comparatively easier to derive, as demonstrated in the next section.


\subsubsection{Concentration inequalities of the   M-steps at the true parameter}\label{sec_theory_samp_rate}

We derive concentration inequalities between the sample-level M-steps $\wh \bpi_\ell(\theta^*)$, $\wh \bmu_{\ell}(\theta^*)$, and $\bwhsw(\theta^*)$ and their population-level counterparts. This requires to bound
\begin{align*}
	     \wh\bpi_\ell(\theta^*)- \bpi_\ell(\theta^*)  &=   \EE_n[\gamma_\ell(X; \theta^*)] - \EE_{\theta^*}[\gamma_\ell(X; \theta^*)] ,\\
	    \wh\bmu_\ell(\theta^*)- \bmu_\ell(\theta^*)   &= { \EE_n[\gamma_\ell(X; \theta^*) X]\over \EE_n[\gamma_\ell(X; \theta^*)]} - {\EE_{\theta^*}[\gamma_\ell(X; \theta^*) X]\over \EE_{\theta^*}[\gamma_\ell(X; \theta^*)]}
	\end{align*}
	as well as  $\bwhsw(\theta^*)-\bsw(\theta^*) $ which equals to
	\begin{align*}
		\EE_n[\gamma_\ell(X; \theta^*) (X-\wh \bmu_{\ell}(\theta^*))(X-\wh \bmu_{\ell}(\theta^*))^\T]  -\EE[\gamma_\ell(X; \theta^*) (X- \mu_{\ell}^*)(X- \mu_{\ell}^*)^\T] .
	\end{align*} 
	
	The following theorem states the rate of convergence of the above three terms.  Recall the distance $d(\sw,\sw^*)$ from \eqref{dist_sw}.

 \begin{theorem}\label{thm_concent}
    Grant condition \eqref{cond_min_Delta}.  Assume $n\ge d\log n$ and
    $
        n\pmin  \ge C \log(n)(1 \vee  \log(n)/d).
    $
    Then with probability at least $1- 4n^{-1}$, one has that for all $\ell \in [L]$, 
\begin{align} \label{concent_rate_pi} 
     {1\over \pi_{\ell}^*} | \wh \bpi_{\ell}(\theta^*) - \bpi_{\ell}(\theta^*) | &~ \lesssim ~ \sqrt{ \log n \over n \pi_{\ell}^*},\\\label{concent_rate_M} 
      \|\wh \bmu_\ell(\theta^*) - \bmu_{\ell}(\theta^*) \|_{\sw^*}  & ~ \lesssim ~  \sqrt{  d\log n \over n\pi_{\ell}^*}.
\end{align}
If, additionally, $n\ge d L^2\log n$ and $nd \ge L^2\log^3 n$ hold, then with the same probability,
\begin{align} \label{concent_rate_sw}
	d(\bwhsw(\theta^*), \bsw(\theta^*) ) &\lesssim \sqrt{d\log n \over n}.  
\end{align} 
\end{theorem} 
\begin{proof}
	The proof can be found in \cref{app_sec_proof_concentration}.  
\end{proof}

\subsubsection{Contraction of the sample level EM iterates}\label{sec_theory_samp_conv}

From the decomposition in \eqref{decomp_h_diff}, another key ingredient towards proving the convergence of the EM algorithm is to establish the contraction property of $\wh \bh(\wh\theta^{(t)})- \wh \bh(\theta^*)$  
for all three operators $\wh\bh \in \{\wh\bpi, \wh\bM, \wh\bsw\}$.  The following proposition states a one-step, local contraction result for any $\theta$ satisfying the initialization requirement in \eqref{cond_init_unknown}.
\begin{prop}\label{prop_EM_samp_contra}
		Under the separation condition \eqref{cond_min_Delta}, assume 
		\begin{equation}\label{rate_cond_param_known}
			dL\log^2(n)  {\sqrt{\Dtmax(L+\Dtmax)}}~ \le c ~ n\pmin
		\end{equation}  
		for some small constant $c>0$.
		There exists some $\kappa_n \in (0,1)$ such that with probability at least $1-5n^{-1}$, the following holds uniformly for any  $\theta = (\pi, M,\sw)$ satisfying \eqref{cond_init_unknown}:  
		\begin{align*} 
			d(\wh\bpi(\theta), \wh\bpi(\theta^*)) + 	d(\wh \bM(\theta), \wh \bM(\theta^*)) & ~ \le ~ \kappa_n  \left\{
			d(\pi,\pi^*) + d(M,M^*) + \sqrt{\Dtmax}~d(\sw, \sw^*)
			\right\}\\
			d(\bwhsw(\theta),\bwhsw(\theta^*)) &~\le ~ \kappa_n  \sqrt{\Dtmax} \left\{
			d(\pi,\pi^*) + d(M,M^*) + \sqrt{\Dtmax}~d(\sw, \sw^*)
			\right\}.
		\end{align*} 
\end{prop}
 
	 \cref{prop_EM_samp_contra} is proved in \cref{app_sec_proof_contraction_samp}.
	 Its proof follows the road-map of proving the population-level result in  \cref{thm_EM_popu_unknown}. By reparametrizing $\gamma_\ell(x; \theta)$ as $\gamma_{\ell}(x; \omega(\theta))$, where $\omega(\theta) = (\pi, M, J(\theta))$, the key technical difficulty lies in establishing the Lipschitz continuity of the mappings $\theta \mapsto \wh \bpi_\ell(\omega(\theta))$ and $\theta \mapsto \wh \bmu_{\ell}(\omega(\theta))$ for all $\theta$ satisfying \eqref{cond_init_unknown}. This reduces to deriving uniform convergence of certain empirical quantities of the form
	 $\EE_n[\partial \gamma_{\ell}(X; \omega(\theta))]$,
	 $\EE_n[\partial \gamma_{\ell}(X; \omega(\theta)) X]$, and
	 $\EE_n[\partial \gamma_{\ell}(X; \omega(\theta)) XX^\T]$,
	 where the partial derivatives of $\gamma_{\ell}(X; \omega(\theta))$ are taken with respect to $\pi$, $M$, and $J(\theta)$ separately. The uniform convergence established in \cref{lem_I_fN_samp}, together with  the population-level contraction rate $\kappa$ defined in \eqref{def_kappa}, determines the contraction rate $\kappa_n$. Under conditions \eqref{cond_min_Delta} and \eqref{rate_cond_param_known}, it is ensured that $\kappa_n < 1$.
	 
	 From \cref{prop_EM_samp_contra}, we also observe that estimating $\sw^*$ results in a slower algorithmic rate of convergence for the EM algorithm; specifically, it is slower by a factor of $\Dtmax$ compared to the case when $\sw^*$ is known.\\

 \begin{remark}[Condition on $\Dtmin$ and $\Dtmax$]\label{rem_separation_small}
	Condition \eqref{rate_cond_param_known} puts a restriction on the largest separation $\Dtmax$. Combined with condition \eqref{cond_min_Delta}, \cref{prop_EM_samp_contra} requires
	\[
	C \left(\log{\Dtmax \over \Dtmin}+\log{1\over \pmin}\right) \le \Dtmin \le  {c ~ n\pmin \over dL (\Dtmax /\Dtmin)  \log^2n }\left(
	1 + {n\pmin \over dL^2\log^2 n}
	\right).
	\]
	Although this already allows for a wide range of $\Dtmin$, it still imposes a restriction that prevents $\Dtmin$ from growing too quickly. Since a large $\Dtmin$ is expected to aid parameter estimation, this requirement is somewhat counterintuitive. Our analysis in the next section reveals that this constraint on $\Dtmin$ arises from  directly analyzing the contraction of the M-steps, and that it can be removed by adopting a different analytical approach. It is for this reason that the analysis presented in this section is only suitable for small to moderate separation.\\
\end{remark}

Combining \cref{thm_concent} with \cref{prop_EM_samp_contra} yields the following guarantees of the sample level EM iterates.  We state our results for known $\sw^*$ and unknown $\sw^*$ separately. 
  
\begin{theorem}[Known $\sw^*$]\label{thm_EM_samp_known}
	Grant conditions \eqref{cond_min_Delta} and \eqref{rate_cond_param_known}.  Assume $nd\ge L^2\log^3 n$.
	Whenever the initialization $\wh \pi^{(0)}$ and $\wh M^{(0)}$ satisfy  \eqref{cond_init_unknown} with $c_{\sw}=0$, 
	with probability at least $1-5n^{-1}$, the EM iterates $\wh \theta^{(t)} = (\wh \pi^{(t)},\wh M^{(t)},\sw^*)$ in \cref{iter_pi_hat,iter_mu_hat} satisfy
	\begin{align} \label{rate_EM_final_known}
		d(\wh\pi^{(t)}, \pi^*) +d(\wh M^{(t)}, M^*) 
		&~ \lesssim ~   \sqrt{d\log n\over n\pmin},\qquad \forall ~ t\ge C \log n.
	\end{align}  
\end{theorem} 
\begin{proof}
	The proof is given in \cref{app_sec_proof_thm_EM_samp}.
\end{proof}

\cref{thm_EM_samp_known} states that, with a suitable initialization, the EM iterates achieve a statistical precision of order $\sqrt{d \log n / (n\pmin)}$ after only $\cO(\log n)$ iterations. The next theorem extends such guarantees to the case when $\sw^*$ is unknown and needs to be estimated.

\begin{theorem}[Unknown $\sw^*$]\label{thm_EM_samp}
     Grant condition \eqref{cond_min_Delta}, $nd\ge L^2\log^3 n$  and
	\begin{equation}\label{rate_cond_param_unknown}
		dL\log^2(n)  \Dtmax {\sqrt{(\Dtmax+1/\pmin )(L+\Dtmax)}} ~ \le c~  n\pmin. 
	\end{equation} 
	Whenever $\wh \theta^{(0)} = (\wh \pi^{(0)},\wh M^{(0)},\whsw^{(0)})$ satisfies  \eqref{cond_init_unknown}, 
	with probability $1-5n^{-1}$, the EM iterates $\wh \theta^{(t)} = (\wh \pi^{(t)},\wh M^{(t)},\whsw^{(t)})$ in \cref{iter_pi_hat,iter_mu_hat,iter_sw_hat} satisfy: for all $t\ge C\log n$,
	\begin{align}\label{rate_EM_final_unknown}
		d(\wh\pi^{(t)}, \pi^*)  + d(\wh M^{(t)}, M^*) 
		  \lesssim  \sqrt{d \log n\over n\pmin},\qquad  
		d(\whsw^{(t)},\sw^*)   
		   \lesssim    \sqrt{d\log n\over n} .
	\end{align}
\end{theorem}
\begin{proof}
	The proof is given in \cref{app_sec_proof_thm_EM_samp}.
\end{proof}

 Comparing to \eqref{rate_cond_param_known}, the required condition in \eqref{rate_cond_param_unknown} becomes stronger due to the slower rate of convergence for estimating $\sw^*$, as seen in \cref{prop_EM_samp_contra}.

\begin{remark}[A special initialization of $\sw^*$]\label{rem_sT_samp}
	Recall the expression of $\whsw^{(t)}$ in  \eqref{iter_sw_hat_alter}.
	If we opt to initialize  
	$\whsw^{(0)}  = \whsT - \wh M^{(0)}\diag(\wh \pi^{(0)}) \wh M^{(0)\T}$ with $\whsT$ being the marginal sample covariance matrix,
	then similar as \cref{cor_EM_popu_unknown}, under condition  \eqref{rate_cond_param_unknown}   and provided that   $\wh\pi^{(0)}$ and $\wh M^{(0)}$ satisfy   \eqref{cond_init_unknown_special}, the conclusion in \cref{thm_EM_samp} remains. Its proof can be found in \cref{app_sec_proof_rem_sT_samp}.
\end{remark}

\subsection{Convergence of the EM algorithm under moderate to large separation}\label{sec_theory_samp_large}

In this section, we establish the convergence result of the EM algorithm when $\Dtmin$ is moderate to large. Our analysis is based on a surrogate loss function defined below. To provide some intuition, we note that for any data point $i \in [n]$ and label $\ell \in [L]$, when $\Dtmin$ increases, one should expect  $\PP_{\theta^*}(Y_i = \ell \mid X_i) = \gamma_{\ell}(X_i; \theta^*)$ in \eqref{def_gamma} to get closer to one if $Y_i = \ell$, and  closer to zero otherwise. This is also expected  for any  $\theta$ that is  close to $\theta^*$. Therefore, a useful metric to evaluate any given $\theta$ in this case should both characterize the discrepancy between $\gamma_{\ell}(X_i; \theta)$ and the true labels $1\{Y_i = \ell\}$ for all $i \in [n]$ and $\ell \in [L]$, and reflect the separation magnitude.

In light of this,  for any $\theta$ with $\gamma_\ell(\cdot;\theta)$ given in \eqref{def_gamma}, we define the loss function
\begin{equation}\label{def_phi_theta}
	\phi(\theta) :=  \sum_{\ell = 1}^L \sum_{a\in[L]\setminus\{\ell\}}\sum_{i\in \wh G_a} \gamma_{\ell}(X_i;\theta) \|\mu_{\ell}^* - \mu_a^*\|_{\sw^*}^2.
\end{equation}
Here $\wh G_1,\ldots, \wh G_L$ form a partition of  $[n]$, defined as $\wh G_a = \{i \in [n]: Y_i = a\}$ for each $a\in [L]$. It is easy to see that $\phi(\theta)$ equals zero if $\gamma_{\ell}(X_i; \theta) = 1\{Y_i = \ell\}$ for all $i \in [n]$. Therefore, establishing the convergence of $\phi(\wh \theta^{(t)})$ to zero is closely related to the community detection problem. Under this perspective, oracle estimator to which $\wh \theta^{(t)}$ converges is the MLE of $\theta^*$ assuming the true labels are observed, given by $\wt \theta := (\wt \pi, \wt M, \wt \sw)$, where for $\ell \in [L]$, 
\begin{equation}\label{def_theta_td}
	\wt \pi_\ell := {n_\ell\over n},\qquad \wt \mu_{\ell} = {1\over n_\ell} \sum_{i\in \wh G_\ell} X_i,\qquad \wt \sw = {1\over n}\sum_{\ell = 1}^L \sum_{i \in \wh{G}_\ell} (X_i - \wt \mu_{\ell})(X_i - \wt \mu_{\ell})^\T,
\end{equation}
where $n_\ell =  \sum_{i=1}^n 1\{Y_i=\ell\}$. Although loss functions analogous to (\ref{def_phi_theta}) have been employed in the analysis of clustering methods based on hard assignments \cite{chen2024achieving, gao2022iterative}, (\ref{def_phi_theta}) introduces a novel loss function tailored to the setting of soft assignments.

The following theorem states the convergence of $\phi(\wh \theta^{(t)})$, for $t\ge 0$,  with $\wh \theta^{(t)}$ being the EM iterates in \cref{iter_pi_hat,iter_mu_hat,iter_sw_hat}.
We start with the case of known $\sw^*$. 
 
	\begin{theorem}[Known $\sw^*$]\label{thm_conv_phi_known}
		Assume $\Dtmin \ge C\log (1/\pmin)$ and 
		\[
			n \ge  C d \left\{
			{\log n\over \pmin},~ L^2\log n, ~ {L^2\log(n\Dtmax) \over \pminsq \Dtmin^2} 
			\right\}
		\]
		for some large constant $C>0$. 
			For any $\wh \theta^{(0)} = (\wh \pi^{(0)},\wh M^{(0)},\sw^*)$ satisfying  
			\begin{itemize}[itemsep = 0mm]
				\item[(a)] either condition \eqref{cond_init_unknown} with   $
					c_\mu \le  \min  \{(\sqrt{2}-1)/2,   ~  c  ~  \pmin \Dtmin/L\},$
				\item[(b)] or $
				(1 / \pmin\Dtmin) (1+ {L /  \pmin \Dtmin }) \phi(\wh \theta^{(0)})   \le c~ n,$
			\end{itemize} 
			for some small constant $c>0$, 
		 	 with probability at least $1-4n^{-1}-\exp(-c'\Dtmin)$,  
			\[
			\phi(\wh\theta^{(t)})  ~ \le ~   n\exp(-c'\Dtmin) + {1\over 2}	\phi(\wh\theta^{(t-1)}),\qquad \forall ~ t\ge 1.
			\]
	\end{theorem}
	\begin{proof}
		The proof can be found in \cref{app_sec_proof_thm_conv_phi_known}.
	\end{proof}

	\cref{thm_conv_phi_known} states that $\phi(\wh \theta^{(t)})$ converges to $n\exp(-c'\Dtmin)$ linearly as $t\to \infty$. When the latter is small, one should expect $\wh \theta^{(t)}$ converges to $\wt \theta$ given in \eqref{def_theta_td} which further converges to $\theta^*$ in the optimal rate. We state such convergence of $\wh\theta^{(t)}$ to $\theta^*$ in \cref{thm_param_phi}.
	
	The key difference from the convergence of $\wh \theta^{(t)}$ established in \cref{prop_EM_samp_contra,thm_EM_samp_known} is that the convergence of $\phi(\wh \theta^{(t)})$ in \cref{thm_conv_phi_known} only requires a lower bound on $\Dtmin$, and  benefits from larger values of $\Dtmin$. This aligns with the intuition that recovering the labels $Y_1, \ldots, Y_n$ becomes easier as $\Dtmin$ increases.
	
	Regarding initialization requirements, condition (a) is stronger than \eqref{cond_init_unknown} in \cref{thm_EM_popu_unknown} only when $\Dtmin \pmin \le L/c$. This extra requirement is needed to have the factor $1/2$ in front of $\phi(\wh \theta^{(t-1)})$. On the other hand, the convergence of $\phi(\wh \theta^{(t)})$ allows for an alternative initialization scheme as given in (b) of \cref{thm_conv_phi_known}. 
	We refer to \cref{sec_init} for detailed discussion on initialization.

	The following proposition extends \cref{thm_conv_phi_known} to unknown $\sw^*$ in which the required conditions become slightly stronger.

\begin{theorem}[Unknown $\sw^*$]\label{thm_conv_phi}
	For some large constant $C>0$, assume
	$n \pmin \ge C d\log n$, $\Dtmin \ge C \log(1/\pmin) + C \log \log n$
	and 
	\begin{equation}\label{cond_n_phi}
	 n \ge C	 dL^2 \log(n\Dtmax) \left(
		{1\over \pminsq \Dtmin^2} +  {\log n \over \Dtmin} + 1
		\right).
	\end{equation} 
	For any $\wh \theta^{(0)} = (\wh \pi^{(0)},\wh M^{(0)},\whsw^{(0)})$ satisfying  
	\begin{itemize}[itemsep = 0mm]
		\item[(a)] either 	\eqref{cond_init_unknown} with  $
			c_\mu + c_{\sw} \le   \min \{(\sqrt{2}-1)/2,   c(1+\pmin \Dtmin)/L\},$
		\item[(b)] or $
			({\Dtmax / \Dtmin} + L+ {L /  \pminsq \Dtmin^2 })\phi(\wh \theta^{(0)})  \le c ~ n$
	\end{itemize} 
	for some small constant $c>0$, with probability $1-4n^{-1}-\exp(-c'\Dtmin)$, 
	\[
  \phi(\wh\theta^{(t)}) ~  \le~   n\exp(-c'\Dtmin) + {1\over 2}	\phi(\wh\theta^{(t-1)}),\qquad\forall ~ t\ge 1.
	\]
\end{theorem}
\begin{proof}
	Its proof is deferred to \cref{app_sec_proof_thm_conv_phi}.
\end{proof}

Comparing to \cref{thm_conv_phi_known}, both the condition in \eqref{cond_n_phi} and the initialization requirements in (a) and (b) become stronger due to estimating $\sw^*$. Although the required lower bound on $\Dtmin$ has one additional $\log\log n$ term, the results in \cref{thm_conv_phi}  still benefit from larger values of $\Dtmin$, in contrast to \cref{thm_EM_samp}.

\paragraph{Proof sketch of  \cref{thm_conv_phi_known,thm_conv_phi}.}  
The proof mainly consists of four main steps.  In the first step, for any $t\ge 1$, by using   the reparametrization $\omega(\wh \theta^{(t)}) = (\wh \pi^{(t)}, \wh M^{(t)}, J(\wh \theta^{(t)}))$ with $J(\wh \theta^{(t)}) = \whsw^{(t)-1}\wh M^{(t)}$ and adding and subtracting the oracle estimator $\wt \theta = (\wt \pi, \wt M, \wt \sw)$ in \eqref{def_theta_td},  
we bound $\phi(\wh\theta^{(t)}) \lesssim  \rI +\rII +\rIII$ by the following three terms
\begin{align*}
	\rI &=   \sum_{\ell = 1}^L \sum_{a\in[L]\setminus\{\ell\}}\sum_{i\in \wh G_a}   \gamma_{\ell}(X_i; \theta^*) \|\mu_{\ell}^* - \mu_a^*\|_{\sw^*}^2,\\
\rII &= \sum_{\ell = 1}^L \sum_{a\in[L]\setminus\{\ell\}}\sum_{i\in \wh G_a} 1\left\{
\left| N_i^\T\sw^{*1/2} (\wt J - J^*)(\be_a - \be_\ell)  \right| >  c' \|\mu_{\ell}^* - \mu_a^*\|_{\sw^*}^2
\right\} \|\mu_{\ell}^* - \mu_a^*\|_{\sw^*}^2,\\
\rIII&=  \sum_{\ell = 1}^L \sum_{a\in[L]\setminus\{\ell\}}  {n_a\over  \|\mu_{\ell}^* - \mu_a^*\|_{\sw^*}^2} \left(  \|\wh \mu_{\ell}^{(t)}-\wt \mu_{\ell}\|_{\sw^*}^2 + 	  \|\wh \mu_{a}^{(t)}-\wt \mu_{a}\|_{\sw^*} ^2+  \|(\whsw^{(t)}-\wt\sw)\wt J(\be_a - \be_{\ell})\|_{\sw^*}^2\right).
\end{align*}
Here $N_1,\ldots, N_n$ are i.i.d. from $\cN_d(0,\bI_d)$ and $\wt J := \wt\sw^{-1}\wt M$. The remaining steps analyze each term separately. The first term is relatively easy to analyze and the second step of our proof shows  that $\rI \lesssim n \exp(-c'\Dtmin)$ with probability at least $1-\exp(-c'\Dtmin)$.  

To bound $\rII$,  one could apply the inequality
$$
	| N_i^\T\sw^{*1/2} (\wt J - J^*)(\be_a - \be_\ell)  |  \le \|N_i\|_2 \|(\wt J - J^*)(\be_a -\be_{\ell})\|_{\sw^*}
$$ 
along with   $\max_{i \in [n]}\|N_i\|_2 = \cO_\PP(\sqrt{d + \log n})$ and 
$\|(\wt J - J^*)(\be_a -\be_{\ell})\|_{\sw^*} = \cO_\PP(\sqrt{(d+\log n) / n\pmin})$, which follow from the analysis of $\wt M$ and $\wt \sw$ in \eqref{def_theta_td}. However, this approach leads to the suboptimal requirement $n \gg d^2$ in order to ensure that $\rII$ is negligible. Instead, our analysis in \cref{lem_sigma_diff_td} of \cref{app_sec_tech_lemmas_thm_conv_phi} employs a leave-one-out technique to control this term more carefully, thereby improving the requirement to $n \gg d$.

Finally, the fourth step is to bound $\rIII$ in \cref{lem_phi_params} of \cref{app_sec_tech_lemmas_thm_conv_phi}  where we relate  the distance between  $\wh \theta^{(t)} = (\wh \bpi(\wh \theta^{(t-1)}),\wh \bM(\wh \theta^{(t-1)}), \bwhsw(\wh \theta^{(t-1)}))$ and $\wt \theta = (\wt \pi,\wt M, \wt \sw)$ to  $\phi (\wh \theta^{(t-1)})$. Together with condition \eqref{cond_n_phi}, we obtain $\rIII \le \phi(\wh \theta^{(t-1)}) / 2$, which completes the proof. \\

Using \cref{lem_phi_params} in the fourth step above, the convergence of $\phi(\wh \theta^{(t)})$ can be translated into the convergence of $\wh \theta^{(t)}$ toward $\wt \theta$. Combined with the analysis of the estimation error of $\wt \theta$, this immediately yields the following convergence rate of the EM algorithm for estimating $\theta^*$.

\begin{theorem}\label{thm_param_phi}
	Under conditions in \cref{thm_conv_phi}, with probability at least $1-n^{-1}-\exp(-c'\Dtmin)$, the following holds for all $t\ge C\log n$,  
	\begin{align*}
			d(\wh \pi^{(t)},\pi^*) + d(\wh M^{(t)}, M^*)  &~ \lesssim ~ \sqrt{d\log n\over n\pmin} + \exp(-c'\Dtmin),\\
			d(\whsw^{(t)},\sw^*)   
			& ~ \lesssim ~  \sqrt{d\log n\over n} +  \exp(-c'\Dtmin).
	\end{align*}
	As a result, when $\Dtmin \ge \log(n) / (2c')$, with probability at least $1-\cO(n^{-1} )$, for all $t\ge C\log n$,  
	\begin{equation*}
			d(\wh \pi^{(t)},\pi^*) + d(\wh M^{(t)}, M^*)  ~ \lesssim ~  \sqrt{d\log n\over n\pmin} ,\qquad 
		d(\whsw^{(t)},\sw^*)    ~ \lesssim ~  \sqrt{d\log n\over n}.
	\end{equation*}
\end{theorem}
\begin{proof}
	See \cref{app_sec_proof_thm_param_phi}.
\end{proof}

Compared to \cref{thm_EM_samp}, \cref{thm_param_phi} yields the same rates of convergence of the EM algorithm in the regime $\Dtmin \gtrsim \log n$. In the next section, we summarize the theoretical guarantees of using both analytical approaches and demonstrate that the EM algorithm achieves minimax optimality under the separation condition \eqref{cond_min_Delta}.

\subsection{Minimax optimality of the EM algorithm}\label{sec_theory_optimality}

In this section we establish the minimax optimality of the EM algorithm under the $L$-GMM. For obtaining the upper bounds, combining the results in \cref{sec_theory_samp_small,sec_theory_samp_large}  gives two scenarios in  \cref{tab_rates} under either of which the EM estimator achieves the rates in \eqref{rate_EM_final_known} for known $\sw^*$ and \eqref{rate_EM_final_unknown} for unknown $\sw^*$. 


\begin{table}[ht]
	\centering 
	\caption{Settings in which the EM estimation achieves the rates in  \eqref{rate_EM_final_known} and \eqref{rate_EM_final_unknown}}
	\label{tab_rates}
	\renewcommand{\arraystretch}{1.5}
	\resizebox{\textwidth}{!}{
		\begin{tabular}{l|l|l}
			\toprule
			&   Known $\sw^*$ & Unknown $\sw^*$ \\ 
			\midrule 
			\multirow{2}{*}{Analysis 1} & $\displaystyle\log{\Dtmax \over \pmin} \lesssim  \Dtmin \lesssim  {\Dtmin \over \Dtmax}{n\pmin \over dL^{3/2}\log^2(n)  }$ &  $\displaystyle\log{\Dtmax \over \pmin} \lesssim  \Dtmin \lesssim {\Dtmin \over \Dtmax} \sqrt{n\pi_{\min}^{*3/2} \over dL^{3/2}\log^2(n) }$ \\   
			& $\wh \theta^{(0)}$ satisfies   \eqref{cond_init_unknown} or (b) of \cref{thm_conv_phi_known} & $\wh \theta^{(0)}$ satisfies   \eqref{cond_init_unknown} or (b) of \cref{thm_conv_phi} \\
			\midrule 
			\multirow{3}{*}{Analysis 2} &   $\displaystyle \Dtmin  \gtrsim  \log(n) \vee \sqrt{dL^2 \log (n\Dtmax) \over n \pminsq}$ &  $ \displaystyle\Dtmin  \gtrsim  \log(n) \vee \sqrt{dL^2 \log (n\Dtmax) \over n \pminsq}$ \\
			& $d L^2 \log n \lesssim n, \quad d \log n \lesssim n\pmin $ &  $d L^2 \log (n\Dtmax)\lesssim n, \quad d \log n \lesssim n\pmin $ \\
			& $\wh \theta^{(0)}$ satisfies  (a) or (b) of \cref{thm_conv_phi_known}  &   $\wh \theta^{(0)}$ satisfies  (a) or (b) of \cref{thm_conv_phi} \\
			\bottomrule
		\end{tabular}
	}
\end{table}

As illustrated in \cref{fig_regime}, the two analytical approaches require different sets of conditions on $\Dtmin$, which overlap under mild assumptions on  $\Dtmax/\Dtmin$. In this regime, EM achieves the rates in \cref{sec_theory_samp_small,sec_theory_samp_large} which are shown below to be minimax optimal.

\begin{figure}[htbp]
  \centering 
    \begin{minipage}[t]{0.48\textwidth}\centering\includegraphics[width=\textwidth]{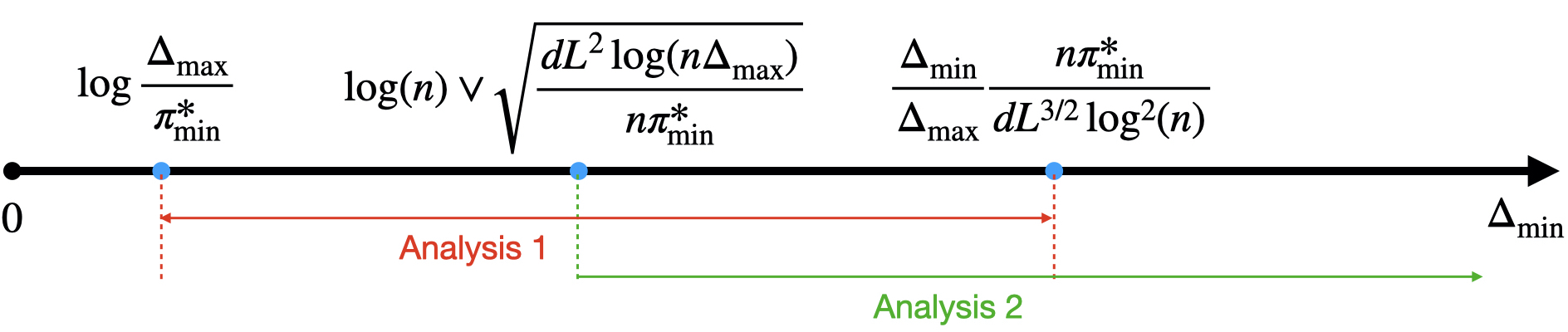}
  \end{minipage}
  \begin{minipage}[t]{0.48\textwidth}\centering\includegraphics[width=\textwidth]{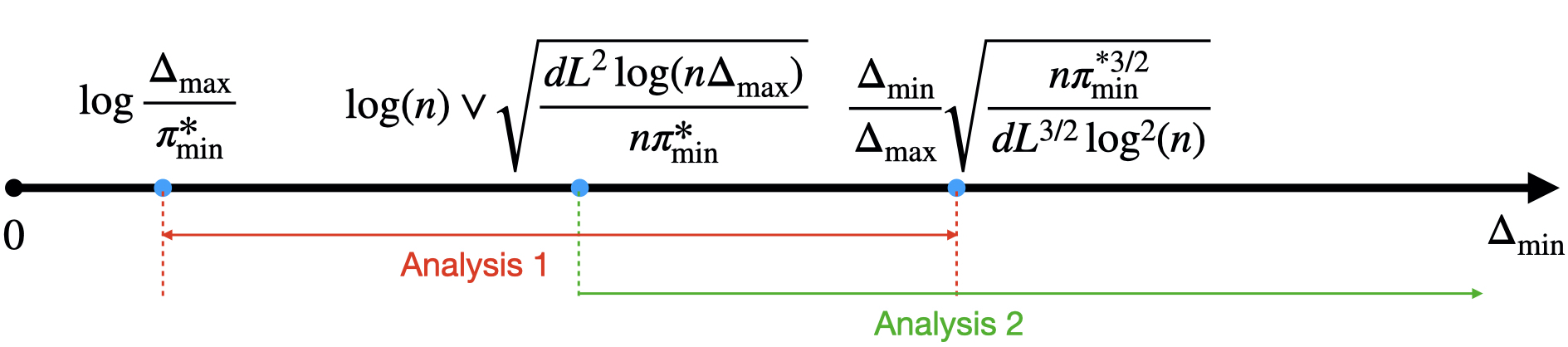}
  \end{minipage} 
  \caption{$\Dtmin$ requirement in which the EM estimation achieves the rates in  \eqref{rate_EM_final_known} and \eqref{rate_EM_final_unknown}. Left: known $\sw^*$; Right: unknown $\sw^*$.}
  \label{fig_regime}
\end{figure}

To further facilitate understanding, we state a simplified  corollary below which invokes the results in \cref{sec_theory_samp_small} for $\Dtmin = \cO(\log n)$ and states the guarantees of the EM algorithm for both known and unknown $\sw^*$.

\begin{cor}\label{cor_EM_final}
	Assume \eqref{cond_min_Delta} and $dL^2 \log(n) \le c n  \pminsq \Dtmin^2$  for some small constant $c>0$. 
	\begin{enumerate}[itemsep=0mm, topsep=1mm,leftmargin=8mm]
		\item[(i)]For known $\sw^*$,  if  
		$
		(\Dtmax  / \Dtmin) dL^{3/2} \log^{3}(n)    \le c n \pmin
		$
		holds, then with probability $1-\cO(n^{-1})$, the bound in \eqref{rate_EM_final_known} holds for the EM iterates $\wh \theta^{(t)} = (\wh \pi^{(t)}, \wh M^{(t)}, \sw^*)$ with $\wh \theta^{(0)}$ satisfying either (a) or (b) of \cref{thm_conv_phi_known},  after  $\cO(\log n)$ iterations. 
		\item [(ii)] For unknown $\sw^*$, if 
		$
		(\Dtmax / \Dtmin)^2   dL^{3/2}\log^4(n)   \le c n \pi_{\min}^{*3/2}
		$
		holds, then with probability $1-\cO(n^{-1})$, the bound in \eqref{rate_EM_final_known} holds for the EM iterates $\wh \theta^{(t)} = (\wh \pi^{(t)}, \wh M^{(t)}, \whsw^{(t)})$ with $\wh \theta^{(0)}$ satisfying either (a) or (b) of \cref{thm_conv_phi},  after  $\cO(\log n)$ iterations. 
	\end{enumerate}  
\end{cor}

To benchmark the rates in \cref{rate_EM_final_known,rate_EM_final_unknown}, we  establish the minimax lower bounds of estimating $\theta^*$ under model \eqref{model}. We consider the following parameter space: for any $\alpha>0,\delta>0$ and  $A \ge 1$,
\begin{equation*}
	\Theta(\alpha, \delta, A)= \left\{
	(\pi, M, \sw):  \min_{\ell\in [L]} \pi_\ell \ge \alpha,  ~ \delta \le   \|\mu_k- \mu_{\ell}\|_{\sw}^2 \le  A \delta, ~ \forall~ k\ne \ell, ~~    \sw \succ 0
	\right\}.
\end{equation*}
Note that for any $\theta^*\in \Theta(\alpha, \delta, A)$, one has $\pmin \ge \alpha$, $\Dtmin \ge \delta$ and $\Dtmax/\Dtmin \le A$. The following theorem states the minimax lower bounds of estimating $\theta^*$ under model \eqref{model} for fixed $\alpha$, $\delta$ and $A$. The magnitude of $A$ also depends on the relative size of $d$ and $L$. 
Recall $d(\wh M,M)$ and $d(\whsw,\sw)$ from \eqref{def_dist_pi_M} and \eqref{dist_sw}.

\begin{theorem}\label{thm_lower_bound}
	Under model \eqref{model}, fix any $\alpha$ and $\delta$ such that  $d \le n \alpha \delta$. Fix  any $A\ge 4$ if $d \ge L-1$, and $A\ge C L^{2/d}$ with some absolute constant $C\ge 1$ if $d < L-1$. Then there exists some absolute constants $c_0\in (0,1)$ and $c_1 > 0$ such that  
	\[
	\inf_{\wh \theta} \sup_{\theta \in \Theta(\alpha, \delta, A)} \PP_{\theta}\left\{ 
	d(\wh M, M) \ge c_1 \sqrt{d \over n \alpha},~ d(\wh \sw, \sw) \ge c_1\sqrt{d \over n}
	\right\} \ge c_0,
	\]
	where the infimum $\wh \theta$ is taken over all estimators. 
\end{theorem}
\begin{proof}
	The proof can be found in \cref{app_sec_proof_lower_bound}.
\end{proof}

\cref{thm_lower_bound} implies that under the conditions of \cref{cor_EM_final}, the EM algorithm  is minimax optimal, up to a multiplicative logarithmic factor. The restriction on $\Dtmax / \Dtmin$ depends on the size $d$ and $L$. Specifically, for $d\ge L-1$,  the most favorable case is $\Dtmax / \Dtmin \asymp 1$   whence \cref{cor_EM_final} requires $dL^{3/2} \log^4(n) = \cO(n \pmin )$ for known $\sw^*$ and  $dL^{3/2} \log^4(n) = \cO(n \pi_{\min}^{*3/2})$ for unknown $\sw^*$. While for $d < L-1$,  the most ideal case is $\Dtmax / \Dtmin \asymp L^{2/d}$ whence  \cref{cor_EM_final} requires $n\pmin \gg  d L^{2/d} L^{3/2}\log^3(n)$ for known $\sw^*$ and  $n\pi_{\min}^{*3/2} \gg  d L^{4/d} L^{3/2}\log^4(n)$  for unknown $\sw^*$.  Such requirement becomes the strongest in the univariate case $d=1$ whence  one needs $n\pmin  \gg   L^{7/2} \log^3(n)$ and $n\pi_{\min}^{*3/2} \gg   L^{11/2} \log^4(n)$ for known and unknown $\sw^*$, respectively.  It is worth mentioning that in such univariate setting with $\sw^*$ known, even for fixed $L$ and $\Dtmin$, the denoising method-of-moment approach  in \cite{WuYang2020} requires at least $(\Dtmax/\Dtmin) L^{4} \le \pminsq n^{1/(2L-1)}$  to achieve the $\sqrt{n}$-rate, which is stronger than what \cref{cor_EM_final} requires.  For either $L=L(n)$ and $\Dtmin = \Dtmin(n)$ satisfying \eqref{cond_min_Delta}, or unknown $\sw^*$, the rate obtained in \citet[Theorem 2]{WuYang2020} is not optimal.\\

Finally, to conclude this section, we state a simplified   result for the $2$-GMM with both unknown $\pi^*$ and $\sw^*$, and compared it with the existing result. 
\begin{cor}[$L=2$]\label{cor_EM_2GMM}
	Assume $\Dtmin \ge C\log(1/\pmin)$, $C d \log(n) \le  n  \pminsq \Dtmin^2$  and 
		$
		    C d \log^4(n)   \le n \pi_{\min}^{*3/2}
		$
    for some large constant $C>0$. Further assume $\wh \theta^{(0)}$ satisfies either \eqref{cond_init_unknown}, or  $ \phi(\wh \theta^{(0)})/n  \le c(1\wedge \pminsq \Dtmin^2)$. Then with probability $1-\cO(n^{-1})$, 
    \eqref{rate_EM_final_known} holds for all $t\ge C\log n$.
\end{cor}

For the $2$-GMM with both $\pi^*$ and $\sw^*$ unknown, \cite{cai2019chime} establishes convergence of the EM algorithm under
\begin{equation}\label{cond_cai}
    C \le \Dtmin \le C',\quad C\le \lambda_d(\sw^*) \le \lambda_1(\sw^*) \le C',\quad c < \pi_1^*, \pi_2^* \le 1-c.
\end{equation}
Their initialization requirement is specified on $\pi^{(0)}, M^{(0)}$ and $\beta^{(0)}$, with the last quantity being the initialization of $\beta^* := \sw^{*-1}(\mu_2^*-\mu_1^*)$. Concretely, their condition (C1) requires at least 
\begin{equation}\label{cond_init_cai}
    d(\pi^{(0)},\pi^*) + d(M^{(0)}, M^*) + \|\beta^{(0)}-\beta^*\|_2 = \cO( \sqrt{\Dtmin / d} ).
\end{equation}
Since our initialization of $\beta^*$ is $\beta^{(0)} = \sw^{(0)-1}(\mu_2^{(0)} -\mu_1^{(0)})$ with $\sw^{(0)}$ given in \cref{thm_EM_popu_unknown}, inspecting our proof reveals that condition \eqref{cond_init_unknown} can be replaced by 
\[  
     d(\pi^{(0)}, \pi^*)   < 1,\quad  \|\beta^{(0)}-\beta^*\|_2 +  d(M^{(0)},M^*)  ~  \le  ~ c_0 \sqrt{\Dtmin},
\]
a much milder condition than \eqref{cond_init_cai} for large $d$. On the other hand, although our optimal rates in (\ref{rate_EM_final_known}) coincide with those derived in \cite{cai2019chime} under their condition (\ref{cond_cai}), our analysis allows $\Dtmin$ to grow and accommodates highly unbalanced mixing weights with $\pmin = o(1)$ as $n\to \i$. In this more general setting, an inspection of the proof in \cite{cai2019chime} reveals that their analysis yields
$
    d(\wh M, M^*) \lesssim (\pmax / \pminsq)\Dtmax^{3/2} \sqrt{{d\log (d) / n} }
$ 
whereas  our analysis gives the optimal rate in \eqref{rate_EM_final_unknown}. 

\subsection{Application to community detection}\label{sec_theory_clustering}

 As noted earlier, our second analytical approach in \cref{sec_theory_samp_large} is closely connected to the community detection problem, which concerns predicting the labels $Y_1,\ldots,Y_n$. In this section, we discuss its implications and show that both of our analytical approaches to the EM algorithm can in fact be used to quantify its performance in community detection.

For any parameter $\theta$, the Bayes rule of predicting  the label of any $x$ from \eqref{model} is  
\begin{equation}\label{def_g}
	g(x; \theta) =  \argmax_{\ell \in [L]} ~ \PP_{\theta}(Y = \ell \mid X= x) =  \argmax_{\ell \in [L]} ~   \gamma_{\ell} (x; \theta).
\end{equation}
The misclustering error between the true labels $Y_1,\ldots, Y_n$ and $g(X_1;\theta),\ldots, g(X_n;\theta)$  is commonly measured by  the Hamming distance, given by
\begin{equation}\label{def_ell_theta}
	\ell(\theta):= {1\over n}\sum_{i =1}^n 1\{Y_i \ne g(X_i; \theta)\}.
\end{equation} 
In \cref{app_sec_proof_thm_clustering}, we prove
$
\ell(\theta) \le 2\phi(\theta) /(n \Dtmin),
$
so that invoking \cref{thm_conv_phi} readily yields the convergence rate of the misclustering error of EM when $\Dtmin \gtrsim \log (1/\pmin) +  \log \log n$. Combined with \cref{thm_EM_samp} of using the first analysis in \cref{sec_theory_samp_small}, we are able to extend this to $\Dtmin \gtrsim \log (1/\pmin) +   \log(\Dtmax \wedge \log n)$.

\begin{theorem}\label{thm_clustering}
	Grant $n \pmin \ge C d\log n$  and $\Dtmin \ge C \log (1/\pmin) + C \log(\Dtmax \wedge  \log n)$ for some large constant $C>0$. 
	\begin{itemize}[itemsep = 0mm]
		\item  If $\Dtmax \le \log n$, assume \eqref{rate_cond_param_unknown} with $\wh \theta^{(0)}$ satisfying either \eqref{cond_init_unknown} or (b) of \cref{thm_conv_phi};
		
		\item 	If $\Dtmax \ge \log n$, assume   \eqref{cond_n_phi}  
		with $\wh \theta^{(0)}$ satisfying either (a) or (b) of \cref{thm_conv_phi};
	\end{itemize} 
	then with probability at least $1-n^{-1}-\exp(-c\Dtmin)$ for some absolute constant $0<c<1/8$, 
	\begin{equation}\label{rate_misclustering_rate}
	\ell(\wh\theta^{(t)} )  ~ \le~    \exp(-c\Dtmin),\qquad \forall ~ t\ge 2\log n.
	\end{equation}
	Furthermore, whenever $\Dtmin > \log(n)/c$, with probability at least   $1-\cO(n^{-1})$, we have the exact recovery, that is, $\ell (\wh \theta^{(t)}) = 0$ for all $t\ge 2\log n$.
\end{theorem}
\begin{proof}
	See \cref{app_sec_proof_thm_clustering}.
\end{proof} 

It is worth comparing   \cref{thm_clustering} with the existing literature on the community detection problem under GMMs. A large body of work focuses on exact label recovery, with spectral clustering being the primary driving approach \citep{dasgupta1999learning,sanjeev2001learning,vempala2002spectral,kannan2008spectral,awasthi2012improved,achlioptas2005spectral}. Along this direction, the best known separation condition for exact recovery \citep{achlioptas2005spectral} requires
\begin{equation}\label{cond_Achli}
	\min_{k\ne \ell} {\|\mu_k^*-\mu_{\ell}^*\|_2^2 \over  \lambda_1(\sw^*)}  ~    \gtrsim  ~  {1 \over \pmin }+ L^2 +  {L\log(nL)} 
\end{equation}
provided that $n \pmin \gtrsim  L (d +\log L)\log n$. By contrast,  \cref{thm_clustering}  for exact recovery requires  the separation condition $\Dtmin \gtrsim \log(1/\pmin)  + \log n$. 
Since  
 $\Dtmin \ge  	\min_{k\ne \ell}\|\mu_k^*-\mu_{\ell}^*\|_2^2 / \lambda_1(\sw^*)$, ours is weaker than \eqref{cond_Achli}, especially when $\pmin$ is small or $L$ large.   As shown in \cite{Ndaoud2022}, the requirement  $\Dtmin\gtrsim \log n$ is necessary for exact recovery even for the symmetric, isotropic 2-GMM. Regarding the sample complexity, under condition \eqref{cond_Achli}, our condition \eqref{cond_n_phi} simplifies to $dL^2 \log (n\Dtmax) \lesssim n$ which is   milder when $\log(\Dtmax /\Dtmin) \ll 1/\pmin$.


The first rate of convergence for the misclustering error was established in \citet{lu2016statistical}, where the authors showed that Lloyd's algorithm, when initialized by spectral clustering and run for $\lceil 4 \log n \rceil$ iterations, yields estimated labels $\wh Y_1, \ldots, \wh Y_n$ that satisfy
\begin{equation}\label{rate_Lloyd}
	{1\over n}\sum_{i=1}^n 1\{Y_i \ne \wh Y_i\} \le \exp\left\{ 
	-(1+o(1))\min_{k\ne \ell} {\|\mu_k^* -\mu_{\ell}^*\|_2^2  \over 8\lambda_1(\sw^*)}
	\right\}
\end{equation}
with high probability.
Their result requires the separation condition
\begin{equation}\label{cond_Lu2016}
	\min_{k \ne \ell} \frac{\|\mu_k^* - \mu_{\ell}^*\|_2^2}{\lambda_1(\sw^*)} ~ \gg ~ \frac{L}{\pmin} + \frac{dL^2}{n \pmin},
\end{equation}
under the assumption that $L \log n \le n \pminsq$. In comparison, \cref{thm_clustering} requires a weaker separation when $\min\{\log(\Dtmax/\Dtmin), \log\log n\} \ll L / \pmin$. This is due to the fact that the EM algorithm uses soft assignments in the E-steps, which influence the updates of the mean vectors in the M-steps. On the other hand, under condition \eqref{cond_Lu2016}, \cref{thm_clustering} still requires  $dL \log^4(n) \le n \pmin$, which is a stronger requirement on the sample size $n$ than that in \cite{lu2016statistical}. This is attributable to the fact that the EM algorithm also estimates the covariance matrix $\sw^*$.

Precisely because Lloyd's algorithm does not explicitly estimate $\sw^*$, the rates in \eqref{rate_Lloyd} are only optimal in the special case where $\sw^* \propto \bI_d$.  \citet{chen2024achieving} recently established the minimax lower bound of the misclustering error for a fixed, general $\sw^*$  under the $L$-GMM: 
\[
	\inf_{\wh Y_1,\ldots, \wh Y_n} \sup_{Y_1,\ldots, Y_n} \EE_{X \mid Y}\left[
		{1\over n}\sum_{i=1}^n 1\{Y_i \ne \wh Y_i\} 
	\right] \ge \exp\left\{ 
	-(1 + o(1)){ \Dtmin \over 8}
	\right\}
\]
provided that $\Dtmin /\log L \to \i$. Here $\EE_{X\mid Y}$ denotes the  conditional expectation of $X_1,\ldots, X_n$ under model \eqref{model} when $Y_1,\ldots, Y_n$ are treated as deterministic quantities. In \cite{chen2024achieving}, they also show that a modified Lloyd's algorithm that uses the hard assignments but also estimates $\sw^*$ provably achieves the above lower bound under conditions including $L = \cO(1)$, $1/\pmin=\cO(1)$, $d^2= \cO(n)$, $\Dtmin \to \infty$ and $c\le \lambda_d(\sw^*)\le \lambda_1(\sw^*) \le C$, as well as an initialization such that $ \sum_{i =1}^n 1\{Y_i \ne \wh Y^{(0)}\} = o(n)$.  By contrast, our result in \cref{thm_clustering} shows that the EM algorithm also achieves the optimal rate of convergence under milder conditions, owing to the use of soft assignments.


\section{Initialization and numerical studies}

We discuss the practical choice of initializations in \cref{sec_init}. Numerical studies are conducted in \cref{sec_sims} to corroborate our theoretical findings.

\subsection{Choice of the initialization}\label{sec_init}

Our analysis in \cref{sec_theory_popu,sec_theory_samp} has focused on the convergent behavior of EM under the assumption that a suitable initialization is provided. In this section, we discuss practical procedures for choosing such initializations. Recall that our theoretical results accommodate two types of initialization schemes: (a) Initializing with $\theta^{(0)}$ that is close to  $\theta^*$ in the sense of \eqref{cond_init_unknown}, or
(b) Initializing the posterior probabilities $\gamma_\ell(X_i; \theta^{(0)})$ for $i \in [n]$ and $\ell \in [L]$ such that $\phi(\theta^{(0)})$ in \eqref{def_phi_theta} is small, as required by (b) of \cref{thm_conv_phi}.

For scheme (b),  if some preliminary estimates of the labels can be obtained, say $\wh Y_1, \ldots, \wh Y_n$, then one can initialize by setting $\gamma_\ell(X_i;  \theta^{(0)}) := 1\{\wh Y_i = \ell\}$ so that 
\[
\phi(\theta^{(0)}) = \sum_{\ell = 1}^L \sum_{a \in [L]\setminus\{\ell\}} \sum_{i \in \wh{G}_a} 1\{\wh Y_i = \ell\} \|\mu_a^*-\mu_{\ell}^*\|_{\sw^*}^2 \le {\Dtmax} ~  {1\over n}\sum_{i=1}^n 1\{Y_i \ne \wh Y_i\}.
\]
This leads to sufficient conditions on the proportion of misclassified data points under which (b) of \cref{thm_conv_phi_known} and (b) of \cref{thm_conv_phi} hold. For example, when the separation condition \eqref{cond_Lu2016} is satisfied, one can use the Lloyd's algorithm to obtain preliminary labels. Indeed, the rates of Lloyd's algorithm given in \eqref{rate_Lloyd} ensure that both (b) of \cref{thm_conv_phi_known} and (b) of \cref{thm_conv_phi} hold, provided the following balancing condition is met:
\[
\min_{k\ne \ell} {\|\mu_k^* -\mu_{\ell}^*\|_2^2  \over \lambda_1(\sw^*)} 
~ \gg ~   \log  
{\max_{k,\ell} \|\mu_k^* -\mu_{\ell}^*\|_2^2  \over \min_{k\ne \ell} \|\mu_k^* -\mu_{\ell}^*\|_2^2} + \log  
{\lambda_1(\sw^*)\over  \lambda_p(\sw^*)}.
\]  
This condition ensures that neither the cluster centers nor the covariance structure are too unbalanced relative to the separation.

On the other hand, when the separation between components is insufficient, scheme (a) can still be employed, provided that a consistent estimator of $\theta^*$ is available. In the context of \cref{rem_sT_samp}, it can be further simplified to find a consistent estimator of $\pi^*$ and $M^*$. A consistent estimator of $\theta^*$ can be obtained, for example, via the Method of Moments (MoM) approaches mentioned in the Introduction. When $d = 1$, the denoising MoM estimator proposed in \cite{WuYang2020} achieves a convergence rate of $(1/\pmin + 1/\Dtmin)   n^{-1/\cO(L)}$. This exponential dependence on $L$ is unavoidable in the absence of any separation condition, and can be improved as the separation increases. For the multivariate $L$-GMM, a fast and consistent MoM estimation is proposed in \cite{LindsayBasak}, though no explicit convergence rates are provided. Recently, for higher dimensions ($d \ge 2$) and $L = 2$, \citet{kalai2010efficiently} proposed a MoM procedure that provably recovers the mixture parameters under the minimal separation. This framework was later extended to $L \ge 3$ by \citet{moitra2010settling}. While their estimation errors exhibit only polynomial dependence on $d$, $1/\pmin$, and $1/\Dtmin$, the degree of such polynomial is high, and the required sample complexity still grows exponentially in $L$. More recently, \citet{doss2023optimal} developed an estimator for the mixing measure that is minimax optimal in the multivariate isotropic $L$-GMM setting. In principle, this estimator could be used to construct a consistent estimator of $\theta^*$ with lower-order dependence on $d$, but this direction is not pursued in their work. For general anisotropic $L$-GMMs with $L \ge 3$, designing a provably consistent and computationally efficient MoM estimator with low sample complexity remains an open and important direction for future research.

In practice, one could initialize EM with different preliminary label estimates obtained from other clustering algorithms, and then select the final estimator that yields the highest likelihood. In our simulation studies below, we adopt initialization scheme (a) for simplicity.

\subsection{Simulation studies}\label{sec_sims}

In this section we first examine the decreasing behavior per iteration of the EM algorithm by tracking its optimization error and statistical error. Next, we evaluate how the convergence rate of the EM algorithm depends on the minimal mixing proportion $\pmin$, the smallest separation among all Gaussian components $\Dtmin$, and whether the covariance matrix $\sw^*$ is known. Third, we track the empirical convergence rate by varying the sample size $n$, and then compare it with our theoretical minimax optimal rate. 
In all cases, the data are generated from model \eqref{model}. 

We initialize the EM algorithm as $\wh{\pi}^{(0)} = 0.7 \pi^* + 0.3 \,\text{Dir}(5)$,
$\wh{\mu}_\ell^{(0)} = \mu_\ell^* + \nu_\ell$, for $\ell \in [L]$ and $
\wh{\sw}^{(0)} = \sw^* + (0.2 \cdot {0.4^2}/d) \,AA^\top$
where  $\text{Dir}(5)$ is the symmetric Dirichlet distribution,
$\nu_\ell$ is drawn uniformly from the sphere of radius $r$ in $\mathbb{R}^d$,  
and $A \in \mathbb{R}^{d \times d}$ has i.i.d. entries from $\mathcal{N}(0,1)$. We set $r = 0.4$ for the first two settings below while $r=0.2$ for the last setting.


\paragraph{Optimization error and statistical error} 

To examine the decreasing behavior per iteration of the optimization error and statistical error, we fix $n = 10,000$, $d = 10$, and $L = 3$. We set $\mu_\ell^* = \Dt_0 / \sqrt{2}\be_\ell$ for all $\ell \in [L]$ and $\sw^* = 0.4^2 \bI_d$, with $\be_\ell$ being the $\ell$th canonical basis vector of $\RR^d$. Thus, $\Dtmin = \Dtmax = (\Dt_0/0.4)^2$. We fix $\Dt_0 = 1.4$ under model \eqref{model} with known $\sw^*$ and perform 10 independent trials. \cref{fig:cov_M_Swknow} shows the results of these simulations. The left plots of \cref{fig:cov_M_Swknow} correspond to
the balanced case $\pi= (1/3, 1/3, 1/3)$ , while the right plot corresponds to the imbalanced case
$\pi = (0.6, 0.2, 0.2)$. The red curves plot the log transformed statistical error, $\log d(\wh{M}^{(t)}, M^*)$, versus the iteration number, whereas the blue curves plot the log transformed optimization error, $\log d(\wh{M}^{(t)}, \wh{M})$, where $\wh{M}$ denotes the convergent point. As seen from the red curves, the statistical error decreases geometrically before leveling off at a plateau. In contrast, the optimization error decreases geometrically down to numerical tolerance. Moreover, comparing the left and right plots, a smaller $\pmin$ results in a slower convergence rate.

\paragraph{Convergence rate with parameters}
We now investigate how the performance of EM algorithms depends on various parameters such as $\Dtmin$. We fix $n=10,000, d=10$ and $L=3$. We set the  mean vectors as 
$\mu_1 = \Delta_0 /\sqrt{2} \be_1$, 
$\mu_2 = \Delta_0 /\sqrt{2} \be_7$, and 
$\mu_3 = \Delta_0 /\sqrt{2} \be_{10}$, 
with $\Sigma^* = 0.4^2 \bI_d$. 
Since the covariance is isotropic, this choice is equivalent to using $\{\be_1, \be_2, \be_3\}$, 
and we adopt it for numerical stability when varying $\Dt_0$ within $\{1.2,1.4, 1.6, 1.8,2.0\}$. We perform 10 independent trials for each case. The averages of $d(\wh{M}^{(t)}, M^*)$ and $d(\wh{\sw}^{(t)}, \sw^*)$ are plotted versus iterations. The left plots of \cref{fig:convergence_rate} correspond to the balanced case $\pi = (1/3, 1/3, 1/3)$, while the right plots correspond to the imbalanced case $\pi = (0.6, 0.2, 0.2)$. We observe that the statistical error decreases initially and then reaches a plateau after several iterations. As expected, a larger $\Dtmin$ yields a faster convergence rate. Comparing the left and right plots, a smaller $\pmin$ again results in a slower convergence rate. Finally, comparing the first row (known $\sw^*$) with the second row (unknown $\sw^*$), we see that the known $\sw^*$ case converges faster. Taken together, all of these findings are consistent with \cref{thm_EM_popu_unknown} and \cref{cor_EM_popu_known}.

\begin{figure}[htbp]
    \hspace{0.7cm}
    \includegraphics[width=1\linewidth]{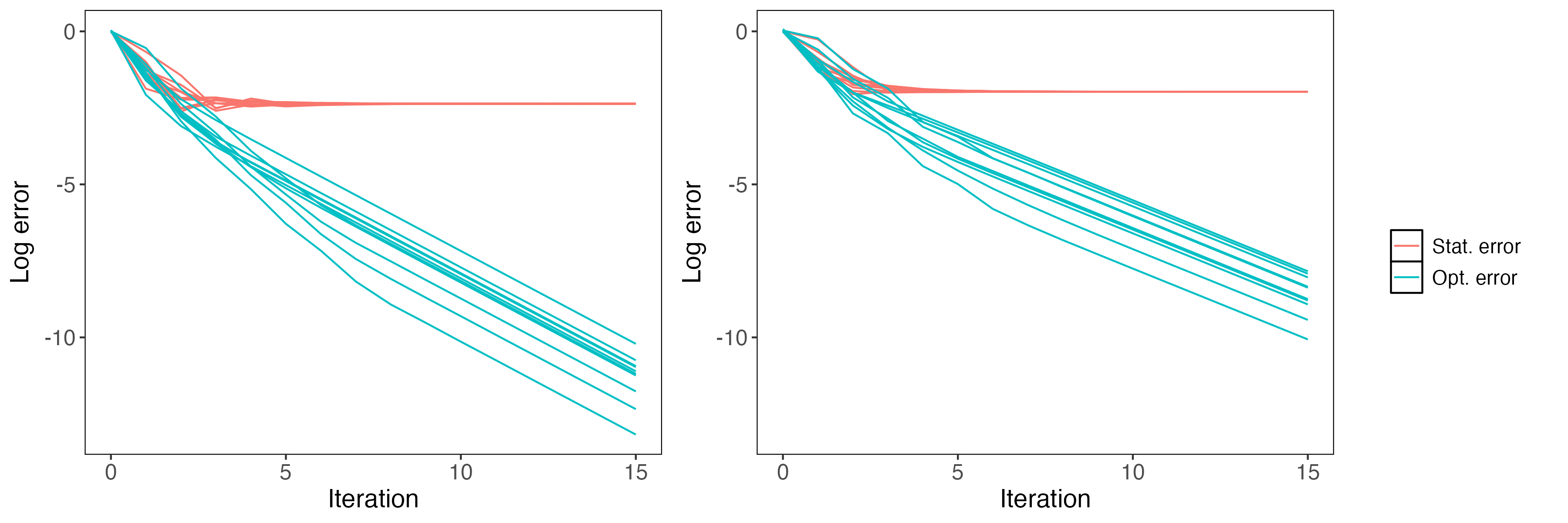}
    \caption{Plots of the iteration number versus log optimization error $\log(d(\wh{M}^{(t)},\wh{M}))$ and log statistical error $\log(d(\wh{M}^{(t)},M^*))$. Left: balanced mixing probability $\pi^*=(1/3,1/3,1/3)$. Right: imbalanced mixing probability $\pi^*=(0.6,0.2,0.2)$.}
    \label{fig:cov_M_Swknow}
\end{figure}

\begin{figure}[htbp]
    \centering
        \begin{subfigure}[t]{1\linewidth}
        \hspace{0.7cm}
        \includegraphics[width=\linewidth]{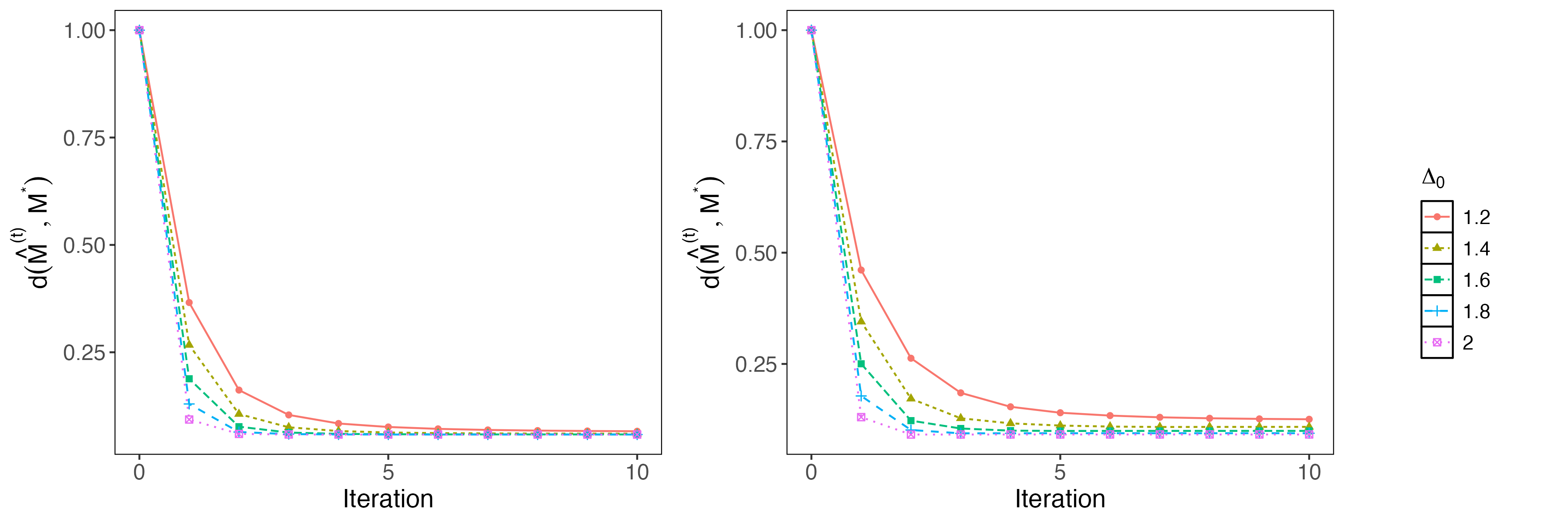}
        \caption{} 
    \end{subfigure}%
    ~\\
    \begin{subfigure}[t]{1\linewidth}
        \hspace{0.7cm}
        \includegraphics[width=\linewidth]{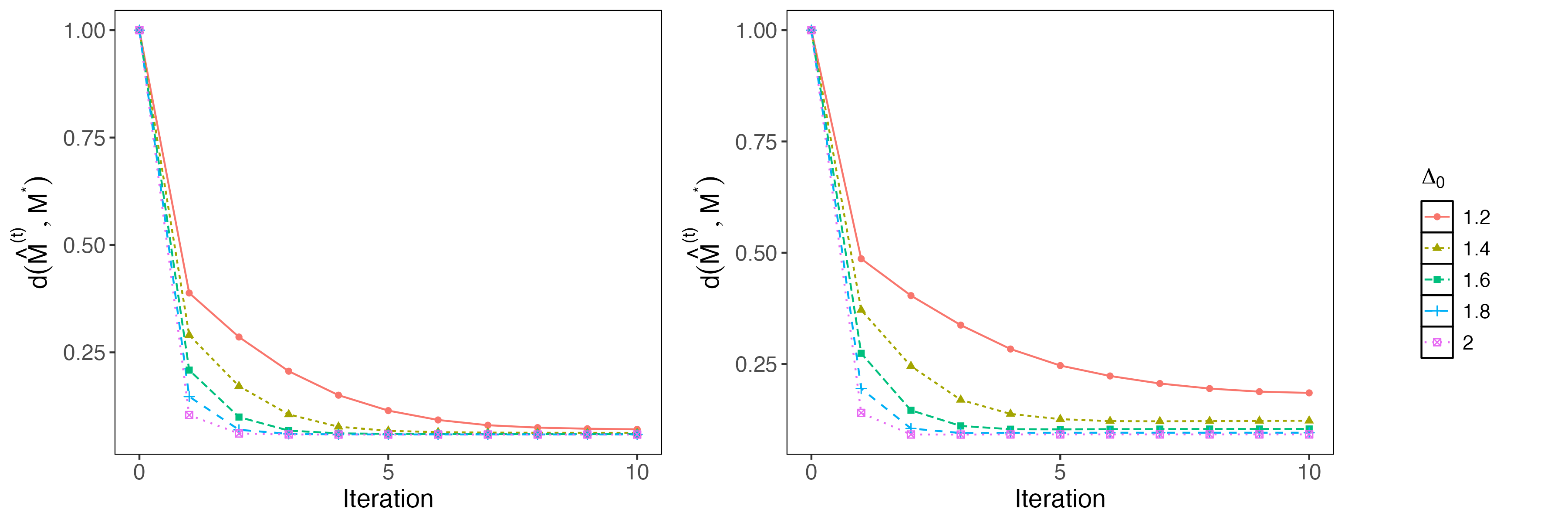}
        \caption{} 
    \end{subfigure} %
~ \\

    \begin{subfigure}[t]{1\linewidth}
        \hspace{0.7cm}
        \includegraphics[width=\linewidth]{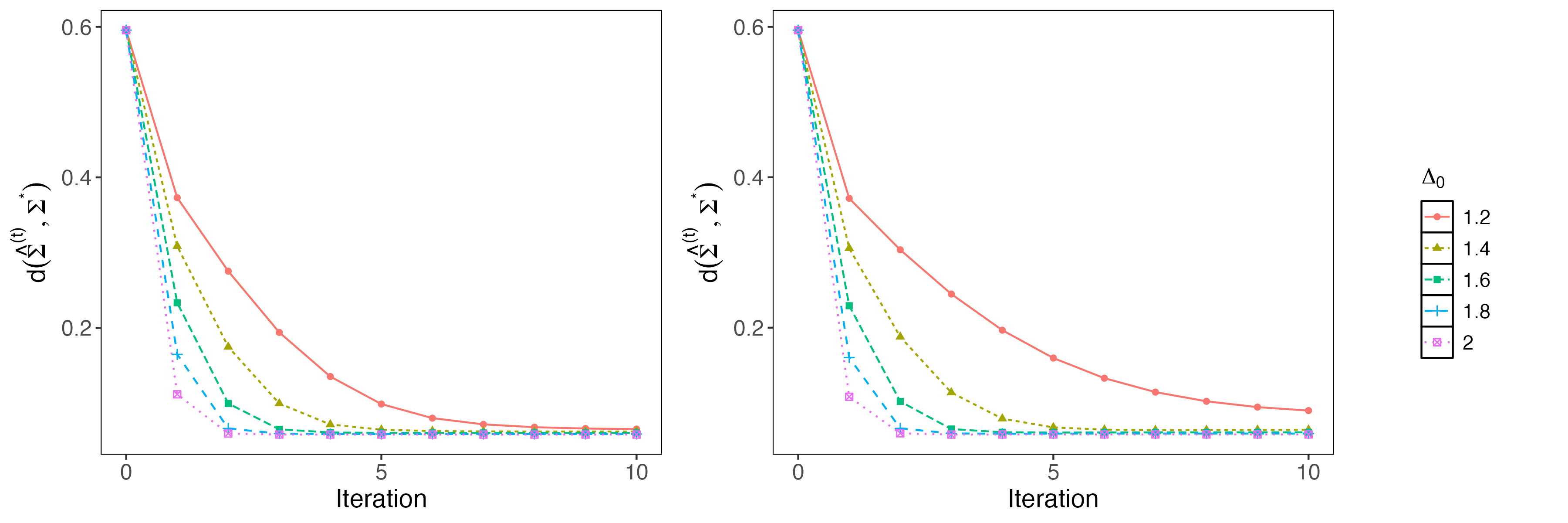}
        \caption{} 
    \end{subfigure}%

    \caption{Convergence of EM on the mixture of $L=3$ Gaussians in $\mathbb{R}^{10}$. The first row (known $\sw^*$) and the second row (unknown $\sw^*$) show average mean estimation statistical error under different separations; the bottom row (unknown $\sw^*$) shows average covariance estimation statistical error. Left: balanced mixing probability $\pi^*=(1/3,1/3,1/3)$. Right: imbalanced mixing probability $\pi^*=(0.6,0.2,0.2)$. }
    \label{fig:convergence_rate}
\end{figure}

\paragraph{Minimax Optimal Convergence Rate vs. Empirical Rate} 
In this section, we compare the empirical rates with the minimax optimal rates. Using similar data generation and initialization as before, we fix $L = 5$ and $d = 50$, set $\mu_\ell = 2\sqrt{2}\be_\ell$ and $\pi_\ell = 1/L$ for all $\ell \in [L]$, and vary $n$ over the sequence from 6,000 to 40,000 with equal increments of 2,000. We compare the isotropic model $\sw^* = 0.4^2 \bI_d$ with the compound symmetry model $\sw^* = 0.6 \cdot\bI_d + 0.4 \cdot \mathbf{1}\mathbf{1}^\T$. The optimal rates (see \cref{tab_EM_compare}) are $d(\wh M, M^*) \asymp \sqrt{d/(n\pmin)}$ and $d(\wh\sw,\sw^*) \asymp \sqrt{d/n}$. We then plot the empirical convergence rates against the optimal rates and fit a line through the origin. The empirical convergence rates under the two models exhibit similar behavior. The $R^2$ values for all cases are greater than $0.99$. The near-linear alignment in \cref{fig:emp-vs-minimax-rate} thus indicates that the observed rates closely follow the minimax theoretical rates. 

\begin{figure}[htbp]
  \centering
    \begin{minipage}[t]{0.48\textwidth}
   \centering
    \includegraphics[width=\textwidth]{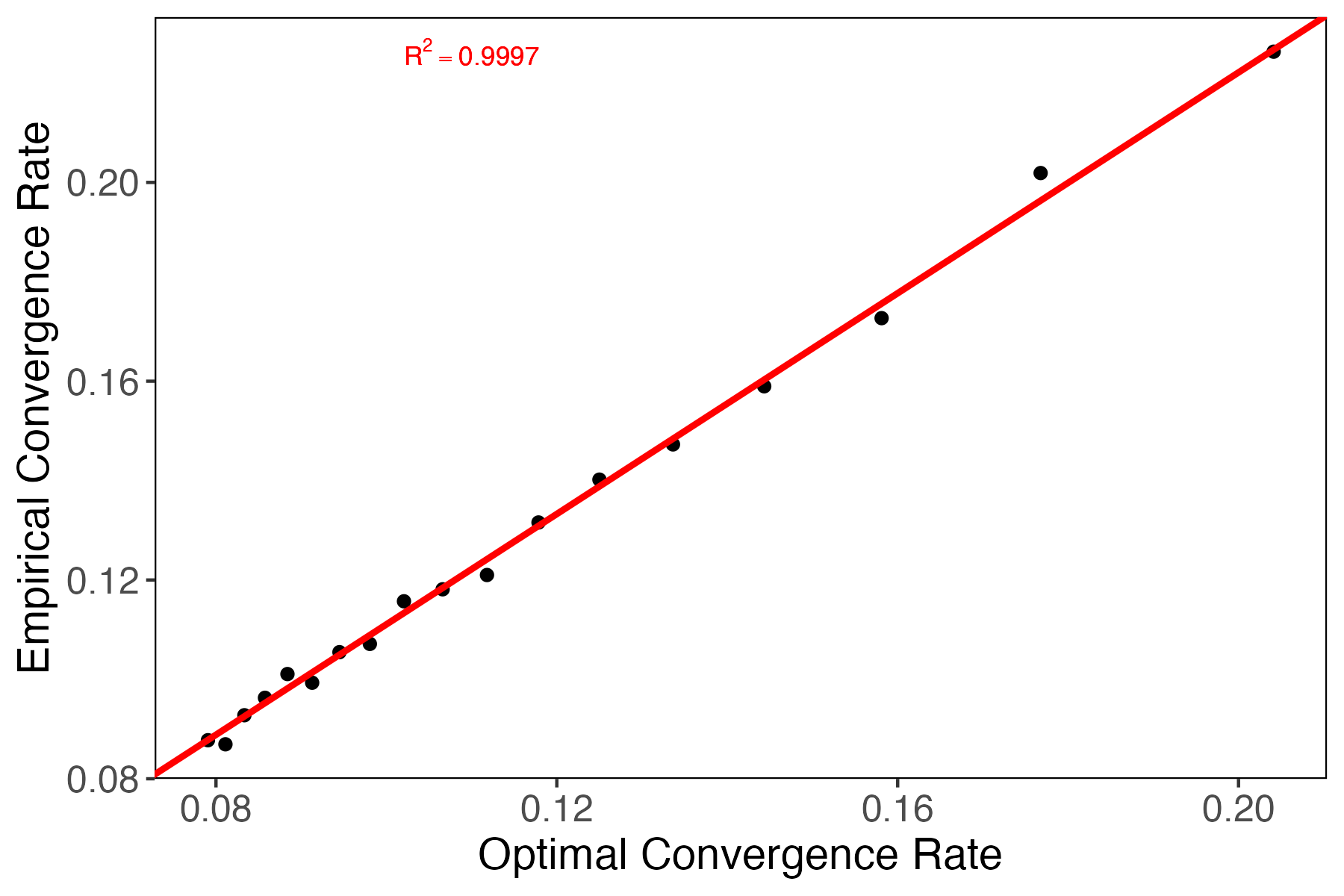}
  \end{minipage}
 \hfill
     \begin{minipage}[t]{0.48\textwidth}
   \centering
    \includegraphics[width=\textwidth]{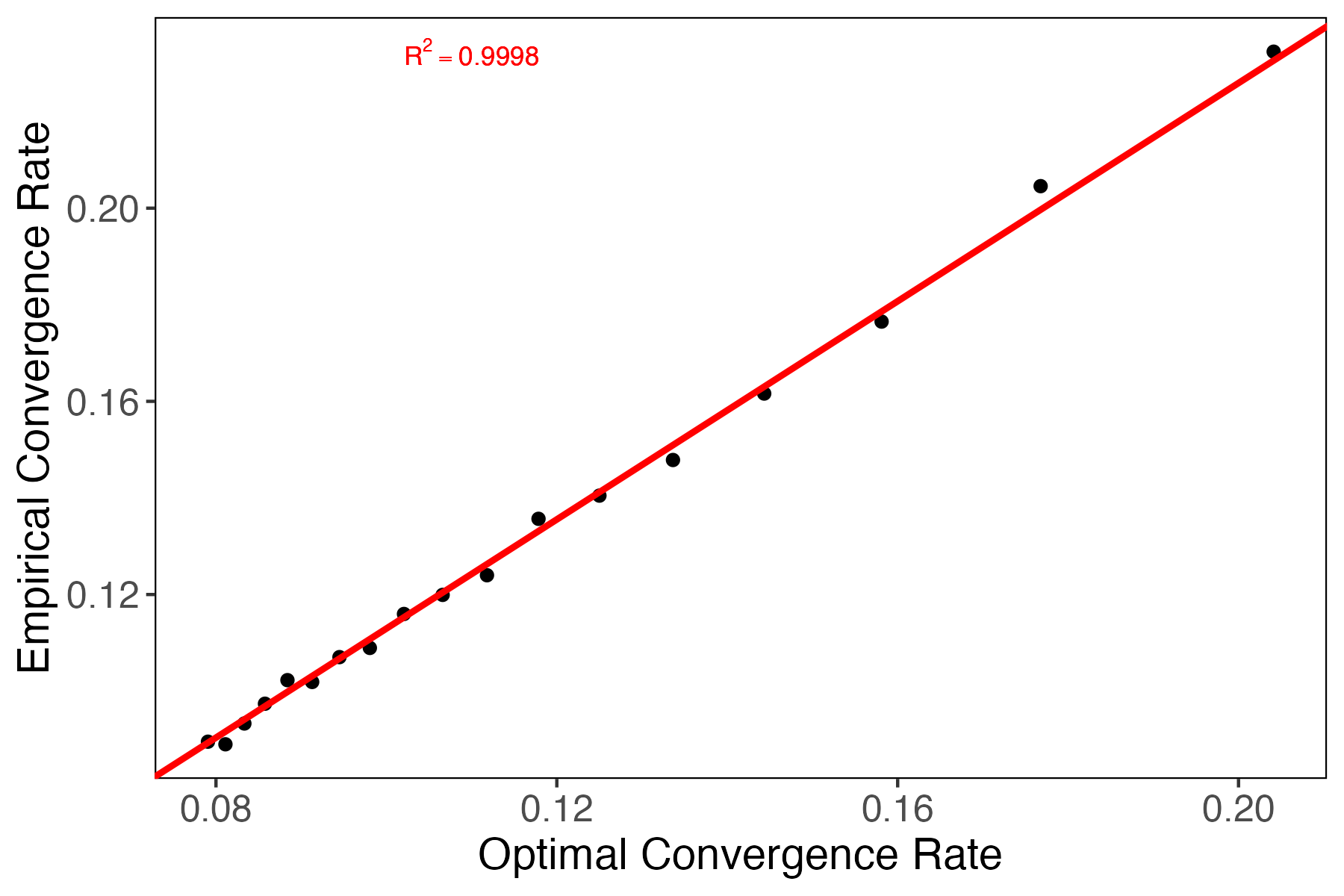}
  \end{minipage}
 \hfill
 
  \begin{minipage}[t]{0.48\textwidth}
    \centering
    \includegraphics[width=\textwidth]{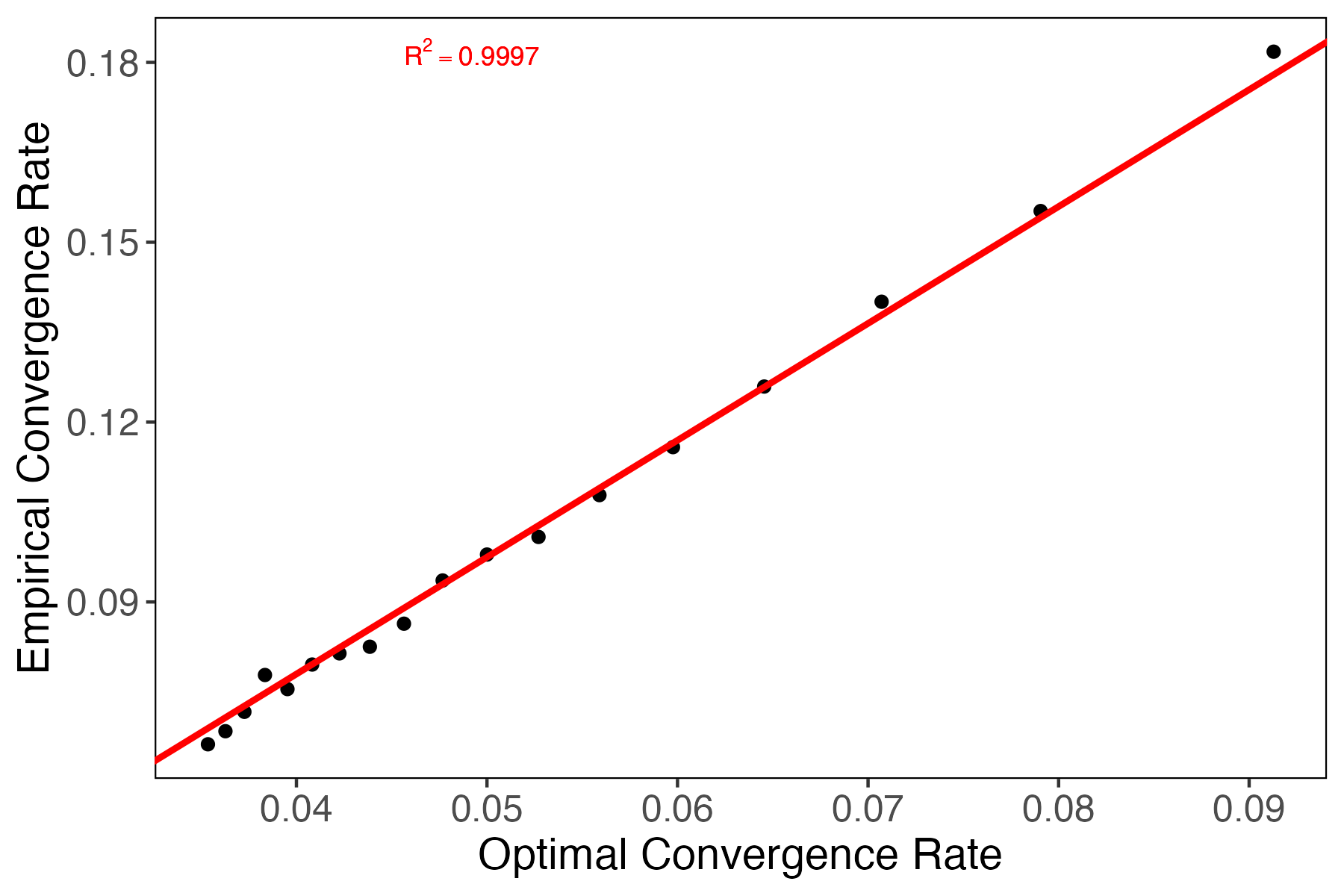}
  \end{minipage}
  \hfill
  \begin{minipage}[t]{0.48\textwidth}
    \centering
    \includegraphics[width=\textwidth]{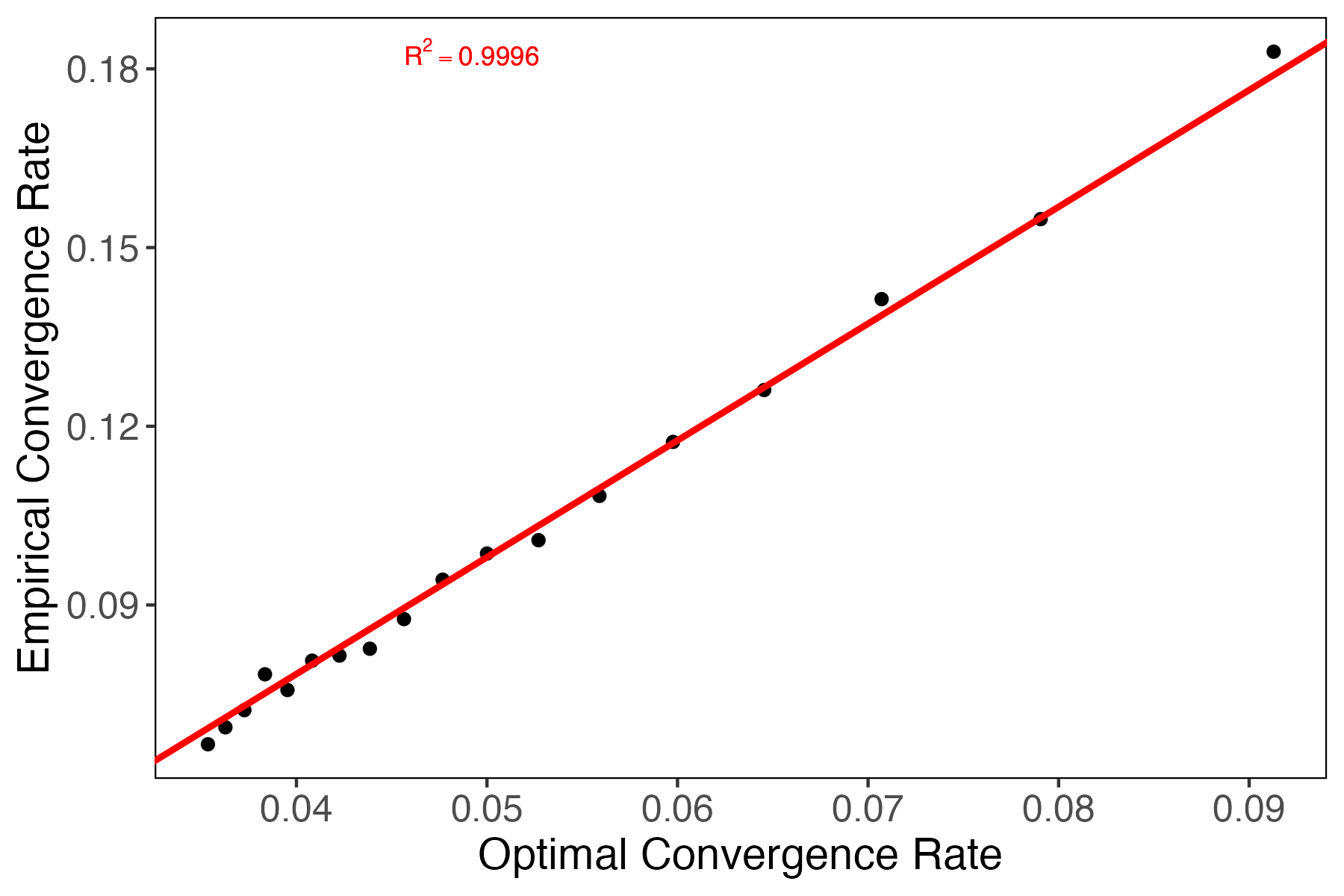}
  \end{minipage}
  
  \caption{Empirical vs.\ optimal rates. Left: isotropic model; Right: compound symmetry model. Top row: mean; Bottom row: covariance.}
  \label{fig:emp-vs-minimax-rate}
\end{figure}

    {\small 
        \setlength{\bibsep}{0.85pt}{
        \bibliographystyle{abbrvnat}
        \bibliography{ref}
        }
    }

\newpage

\appendix

\cref{app_sec_tab} contains the full comparison table with existing theoretical works of EM. Lemma \ref{lem_inv_sw_hat} is proved in 
\cref{app_sec_proof_lem_inv_sw_hat}. 
The proofs of population-level convergence of the EM algorithm are stated in \cref{app_sec_proof_EM_popu}.  \cref{app_sec_proof_sample_small} contains the proofs of the sample-level results in \cref{sec_theory_samp_small} when the separation is small to moderate while \cref{app_sec_proof_sample_large} contains the proofs when the separation is moderate to large. \cref{app_sec_proof_lower_bound} contains the proof of the minimax lower bounds. Auxiliary lemmas are collected in \cref{app_sec_aux}.

\section{A full comparison table with existing theoretical works of the EM algorithm under the $L$-GMM}\label{app_sec_tab}

\begin{table}[ht]
	\centering
	\caption{Comparison with existing theoretical results for the EM algorithm under the $L$-GMM. Symmetric 2-GMM means $\pi_1= \pi_2 = 1/2$ and $\mu_1=-\mu_2 = \mu$. In \cite{cai2019chime}, $\beta = \sw^{-1}(\mu_2-\mu_1)$.}
	\label{tab_EM_compare_full}
	\renewcommand{\arraystretch}{1.3}
	\resizebox{\textwidth}{!}{
		\begin{tabular}{l|c|c|c|c|c}
			\toprule
			& $L$ & Parameter settings & Separation $\Dtmin$  & Initialization & Convergence Rates \\
			\midrule
			\cite{EM2017}               & $2$      & symmetric, $\sw=\bI_d$      & $>C$           & $\|\mu^{(0)}-\mu\|_2 \le \|\mu\|_2/4$      & $d(\wh M, M) \lessapprox \|\mu\|_2^3\sqrt{d/n})$ \\
			\cite{Xu_Hsu_Maleki}        & $2$      &  symmetric, $\sw$ known       & $> 0$ & $ \mu^{(0)\T}\mu \ne 0$      & $d(\wh M, M) = o(1)$ \\
			\cite{daskalakis2017ten}    & $2$      &  symmetric,  $\sw$ known       & $\wt\Omega((d/n)^{1/3})$ & $\mu^{(0)\T}\mu \ne 0$    & $d(\wh M, M) \lessapprox (\sqrt{d/n} \wedge \|\mu\|_2^{-2}\sqrt{d/n})$ \\ 
             \cite{wu2021randomly} & $2$      &  symmetric,  $\sw=\bI_d$       & $(0, C]$ & $\mu^{(0)\T}\mu \ne 0$    & $d(\wh M, M) \lessapprox \left((d/n)^{1/4} \wedge  \|\mu\|_2^{-1}\sqrt{d / n}\right)$ \\ 
             \cite{weinberger2022algorithm} & $2$      & \makecell{$\pi_2>\pi_1$ known,  $\sw=\bI_d$ \\ $\mu_2 = -\mu_1 = \mu$}      & $(0, C]$ & $\mu^{(0)\T}\mu \ge  0$    &$d(\wh M, M) \lessapprox\left((d/n)^{1/4} \wedge  \|\mu\|_2^{-1}\sqrt{d / n} \wedge (\pi_2-\pi_1)^{-1}\sqrt{d/n}\right)$ \\ 
			\cite{cai2019chime}         & $2$      &  \makecell{$c\le \pi_1,\pi_2 \le 1-c$\\ $c \prec \sw \prec   C$}                               & $[C,C']$       &    \makecell{$d(\pi^{(0)},\pi) + d(M^{(0)}, M) $\\
            $+\|\beta^{(0)} - \beta\|_2 \le c\sqrt{\Dtmin / d}$}  & $\displaystyle d(\wh M,M) + \|\wh\beta-\beta\|_2 \lesssim \sqrt{d/n}$ \\
			\midrule  
			\cite{dasgupta2007probabilistic} & $\ge 2$ & $\sw=\sigma^2\bI_d$     & $\displaystyle\Omega\Bigl(\sqrt{d \log \pi_{\min}^{-1}} + \log \pi_{\min}^{-1}\Bigr)$          &  A particular initialization  & $\displaystyle d(\wh\pi,\pi)+ d(\wh M, M)  \lesssim \left( \sqrt{d /( n\pi_{\min})} + {\exp(-c\sqrt{d})/ \pi_{\min}}\right) $\\
            \cite{kwon2020algorithm}& $\ge 2$   &  $\sw=\sigma^2\bI_d$                            & $\log(1/\pi_{\min})$         & $d(\pi^{(0)},\pi) < 1,~ d(M^{(0)}, M) \le c \sqrt{\Dtmin}   $ 
            & A sample-splitting version of the EM \\
			\cite{yan2017convergence}   & $\ge 2$  & $\pi$ known, $\sw=\bI_d$                              & $ \displaystyle\Omega\Bigl((d \wedge L)\log{(\Dtmax / \pi_{\min})}\Bigr)$         & $ d(M^{(0)}, M) \le c \sqrt{\Dtmin}$ 
            & $\displaystyle d(\wh M, M) \lessapprox {(\pi_{\max} / \pi_{\min})} \sqrt{\Dtmax}(\sqrt{d} + L^3\Dtmax)\sqrt{d/ n}$ \\
			\cite{Zhao2020}   & $\ge 2$  & $\pi$ known, $\sw=\bI_d$                              & $\displaystyle\Omega\Bigl((d \wedge L)\log{(\Dtmax / \pi_{\min})}\Bigr)$         & $ d(M^{(0)}, M) \le c \sqrt{\Dtmin}$ 
            & $\displaystyle d(\wh M, M) \lessapprox({\sqrt{\Dtmax} / \pi_{\min}}) \sqrt{dL / n} $ \\
            \cite{segol2021improved} &$\ge2$  &   $\pi$ known, $\sw=\bI_d$                              & $ \log(1/\pi_{\min}) $         & $ d(M^{(0)}, M) \le c \sqrt{\Dtmin}   $ 
            & $\displaystyle d(\wh M, M) \lessapprox  (1 / \pi_{\min}) \sqrt{dL/ n}  $ \\
			\midrule 
			\multirow{2}{*}{This work}                   & \multirow{2}{*}{$\ge 2$}   & \multirow{2}{*}{Full generality}                            &  \multirow{2}{*}{$\displaystyle\Omega\Bigl(\log{(\Dtmax / \pi_{\min})}\Bigr)$}       &      $d(\pi^{(0)},\pi) < 1,~d(M^{(0)}, M)\le c\sqrt{\Dtmin}$   &$\displaystyle d(\wh\pi,\pi) + d(\wh M, M) \lessapprox \sqrt{d / (n\pi_{\min})} $\\
			& & & &  \makecell{$ d(\sw^{(0)}, \sw)\le  c $,  or via preliminary labels} & $\displaystyle d(\whsw, \sw) \lessapprox \sqrt{d / n} $\\
			\bottomrule
		\end{tabular}
	}
\end{table}

\section{Proof of \cref{lem_inv_sw_hat}}\label{app_sec_proof_lem_inv_sw_hat}

    We first state and proof a simple lemma that is used in proving \cref{lem_inv_sw_hat}.
    Let $\Gamma$ be any $n\times L$ matrix with each row $\Gamma_{i\cdot} \in \cS^L$ for all $i\in [n]$. Here  $\cS^L$ denotes the $L$-dimensional probability simplex. Let  $P_\bX \in \RR^{n\times n}$ denote the orthogonal projection matrix onto the column space of $\bX:=(X_1, X_2, \ldots, X_n)^\T \in \RR^{n\times d}$. Write $P_\bX^\perp = \bI_n -  P_\bX$.
\begin{lemma}\label{lem_Hhat_inv}
    If there exists at least one $i\in [n]$ such that $\Gamma_{i\ell} > 0$ for all $\ell \in [L]$, then the matrix
    $$
      \diag(\Gamma^{\T}1_n) -  \Gamma^\T P_\bX~ \Gamma
    $$
    is strictly positive definite. 
\end{lemma}
\begin{proof}
    Start with
    \begin{align*}
     \diag(\Gamma^{\T}1_n) -  \Gamma^{^\T}P_{\bX} \Gamma  =  \diag(\Gamma^{\T}1_n) -\Gamma^{^\T} \Gamma +\Gamma^{^\T} P_{\bX}^{\perp}\Gamma.
\end{align*} 
It remains to show that $\diag(\Gamma^{\T}1_n) -\Gamma^{^\T} \Gamma$ is  strictly positive definite.  
For any non-zero vector $v\in \RR^L$, we have
\begin{align*}
    v^\T \left(\diag(\Gamma^{\T}1_n) -\Gamma^{^\T} \Gamma\right) v 
    = & \sum_{i=1}^n v^\T \left(\diag(\Gamma_{i\cdot}) - \Gamma_{i\cdot}\Gamma_{i\cdot}^\T\right)v.
\end{align*}
Since $\Gamma_{i\cdot}\in \cS^L$ for all $i\in [n]$, the above is always non-negative. Moreover, it is strictly greater than zero as long as there exists at least one $i\in [n]$ such that $\Gamma_{i\ell} > 0$ for all $\ell\in [L]$. This completes the proof.  
\end{proof}

\begin{proof}[Proof of \cref{lem_inv_sw_hat}]
    For any $t\ge 0$ with any $\wh\theta^{(t)}$, write  
    \begin{equation}\label{def_Gamma_hat}
        \wh \Gamma^{(t)} := \Gamma(\bX, \wh\theta^{(t)})  :=  
    \begin{bmatrix}
    \gamma_1(X_1; \wh\theta^{(t)}) & \cdots & \gamma_L(X_1; \wh\theta^{(t)}) \\  
    \vdots & \ddots & \vdots \\  
    \gamma_1(X_n; \wh\theta^{(t)}) & \cdots & \gamma_L(X_n; \wh\theta^{(t)})
    \end{bmatrix} \in \RR^{n \times L}.
    \end{equation}
    Then
    using \cref{iter_pi_hat,iter_mu_hat}  gives that 
    \[
            \wh M^{(t)} =  \bX   \wh \Gamma^{(t)} \diag^{-1}( \wh \Gamma^{(t)\T}1_n), \qquad \wh\pi^{(t)} = {1\over n} \wh \Gamma^{(t)\T}1_n
    \]
    so that \eqref{iter_sw_hat_alter}  yields
    \begin{align*}
        \whsw^{(t)} & = \whsT - \wh M^{(t)} \diag(\wh\pi^{(t)}) \wh M^{(t)\T}\\
        &= {1\over n}\bX^\T\bX - {1\over n}\bX   \wh \Gamma^{(t)} \diag^{-1}( \wh \Gamma^{(t)\T}1_n) \wh \Gamma^{(t)\T}\bX^\T\\
        &= {1\over n}\bX^\T\left[\bI_n - P_{\bX}^{1/2}\wh \Gamma^{(t)} \diag^{-1}( \wh \Gamma^{(t)\T}1_n) \wh \Gamma^{(t)\T} P_{\bX}^{1/2} \right]\bX.
    \end{align*} 
    Observe that the matrix $H:= P_{\bX}^{1/2}\wh \Gamma^{(t)} \diag^{-1}( \wh \Gamma^{(t)\T}1_n) \wh \Gamma^{(t)\T} P_{\bX}^{1/2}$ is  positive semi-definite and has rank at most $L$. Moreover, for any $1\le r\le L$, its the $r$th largest eigenvalue satisfies 
    \[
        0\le \lambda_r(H) = \lambda_r\left(\diag^{-1/2}( \wh \Gamma^{(t)\T}1_n) \wh \Gamma^{(t)\T}P_{\bX} \wh \Gamma^{(t)}\diag^{-1/2}( \wh \Gamma^{(t)\T}1_n)\right) < 1
    \]
    by using \cref{lem_Hhat_inv} and the fact that each row of $\wh\Gamma^{(t)}$ has $L$ non-zero entries. Therefore $(\bI_n - H)$ has rank $n$ and the invertibility of $\whsw^{(t)}$ follows from that of $\whsT$, completing the proof. 
\end{proof}

\section{Proofs of \cref{sec_theory_popu}: convergence  of the population level EM algorithm} \label{app_sec_proof_EM_popu}

\noindent {\bf Notation.} Throughout the proof, we write 
\begin{equation}\label{def_Delta}
	\Delta_{k\ell} = \|\mu_k^*-\mu_\ell^*\|_{\sw^*}^2 = (\mu_k^*-\mu_\ell^*)^\T \sw^{*-1} (\mu_k^*-\mu_\ell^*),\qquad \forall ~ k,\ell \in [L].
\end{equation}
Recall $\Dtmax$ and $\Dtmin$ from \eqref{Delta_max_min}.
Due to $\EE[X]=0$, we have 
\begin{equation}\label{eq_Dt_max_Dt_inf}
	{1\over 4}\Dt_{\max} \le \max_{\ell \in [L]} ~ \mu_\ell^{*\T} \sw^{*-1} \mu_\ell^* \le \Dt_{\max}.
\end{equation}
For any $\pi = (\pi_1,\ldots,\pi_L)^\T$, we write 
$
\pi_{\min} = \min_{\ell\in [L]} \pi_\ell.
$

Recall that given any parameter $\theta$, the population level M-step in \cref{iter_pi,iter_mu,iter_sw} are
\begin{align}\label{def_E_step_popu_append}
  \bpi_\ell(\theta) = \EE_{\theta^*}[\gamma_\ell (X; \theta)],\qquad \bmu_\ell(\theta) = {\EE_{\theta^*}[\gamma_\ell (X; \theta) X] \over \EE_{\theta^*}[\gamma_\ell (X; \theta)]},\qquad \text{for all $\ell \in [L]$,}
\end{align}
and, by collecting $\bpi(\theta) = (\bpi_1(\theta),\ldots,\bpi_L(\theta))^\T$ and $\bM(\theta) = (\bmu_1(\theta), \ldots, \bmu_L(\theta))$, 
\begin{align}\label{def_E_step_sw_popu_append}
    \bsw(\theta) = \EE_{\theta^*}[XX^\T] - \bM(\theta) \diag(\bpi(\theta)) \bM(\theta)^\T.
\end{align}
Note that they all depend on $\gamma_\ell(X; \theta)$. We start by rewriting
    \begin{equation}\label{def_gamma_omega}
        \gamma_\ell(x; \theta) = \frac{\pi_\ell}{ \sum_{k =1}^L \pi_k \exp\left\{(x-{1\over 2}( \mu_\ell+\mu_k))^\T J(\theta) (\be_k-\be_\ell)\right\}} := \gamma_\ell(x;\omega(\theta))
    \end{equation}
for any  $x\in \RR^d$ and $\ell \in [L]$,  where we introduce
    \begin{equation}\label{def_omega}
         \omega(\theta)  = (\pi,  M, J( \theta)) 
    \end{equation}
    with $J(\theta):=\sw^{-1}M$.

\subsection{Self-consistency of the (population level) EM iterates}\label{app_sec_self_consistency}

The so-called {\em self-consistency} of the above update is easy to verify: for any $k\in [L]$, we have 
\begin{align*}
    &\bpi_k(\theta^*)  = \EE_{\theta^*}[\gamma_k(X;\theta^*)]  = \EE_{\theta^*}[ \PP_{\theta^*}(Y=k \mid X)] = \PP_{\theta^*}(Y = k) = \pi_k^*,\\
    &\bmu_k(\theta^*)  = \frac{\EE_{\theta^*}[\gamma_k(X;\theta^*)X]}{\EE_{\theta^*}[\gamma_k(X;\theta^*)]} = \frac{\EE_{\theta^*}[\PP_{\theta^*}(Y=k \mid X)X]}{\pi_k^*}= \frac{\EE_{\theta^*}[1\{Y=k\}X]}{\pi_k^*}=\frac{\pi_k^* \mu^*_k}{\pi_k^*} =\mu_k^*.
\end{align*} 
As a result, 
\begin{equation}\label{eq_sw_sT}
    \bsw(\theta^*) = \sT^* - M^*D_{\pi^*}M^{*\T} = \sw^*
\end{equation}
and $J(\theta^*) = \sw^{*-1}M^* =: J^*.$
In the subsequent proofs,  we drop the subscripts $\theta^*$ in $\EE_{\theta^*}$ and $\PP_{\theta^*}$ when there is no confusion. 

\subsection{A local neighborhood of $\omega^*$ and Lipschitz property of $J(\cdot)$ and $\sw(\cdot)$}

    For any $\omega = (\pi, M, J)$ and $\omega' = (\pi',M', J')$, by recalling $d(\pi,\pi')$ and $d(M,M')$ from \eqref{def_dist_pi_M}, define 
	\begin{equation}\label{def_dist_J}
		d(J, J') := {1\over 2} \max_{a,\ell \in[L]}\|\sw^{*1/2} (J-J')(\be_a - \be_{\ell})  \|_2
	\end{equation}  
    and the distance between $\omega$ and $\omega'$ as
    \begin{equation}\label{def_dist_omega}
        d(\omega,\omega') :=  d(\pi,\pi') + \sqrt{\Dtmax} ~ d(M,M') +\sqrt{\Dtmax}~   d(J,J').
    \end{equation}
    For some constant  $0<c_0<(\sqrt{2}-1)/2$, denote  the size-$c_0$ neighborhood of $\omega^*$ by  
    \begin{align}\label{cond_init_w_J} \nonumber
        \cB_d(\omega^*) \equiv \cB_d(\omega^*, c_0) = &\left\{
           (\pi,M,J):   d(\pi, \pi^*) \le {1\over 2}, ~ {d(M,M^*)\over \sqrt{\Dtmin}} \le c_0,  \right. \\
          & \hspace{1cm}\left. ~ \max_{a\ne \ell} {1\over  \sqrt{\Dt_{a\ell}}} \left\|\sw^{*1/2} (J  -J^*) (\be_a - \be_{\ell})\right\|_2 \le 2c_0
        \right\}.
    \end{align}  
	For future reference, \cref{lem_lip_J_basic} below ensures  that  for any $\theta$ satisfying \eqref{cond_init_unknown},  
	\begin{align*}
			 {1\over  \sqrt{\Dt_{a\ell}}} \left\|\sw^{*1/2} (J(\theta)  -J^*) (\be_a - \be_{\ell})\right\|_2 &\le    {\|\mu_a -\mu_a^*\|_{\sw^*} +  \|\mu_\ell -\mu_\ell^*\|_{\sw^*}\over (1 - c_\sw)\sqrt{\Dt_{a\ell}}}+ {c_\sw\over 1 - c_\sw} \\
			 &\le {2c_\mu + c_\sw \over 1-c_\sw}
	\end{align*}
    holds for all $a\ne \ell$. Consequently,  
	its corresponding $\omega(\theta) \in \cB_d(\omega^*,c_0)$ with $c_0 = (c_\mu + c_\sw/2)/(1-c_\sw)  $ and 
    \begin{equation}\label{init_dist_omega}
        d(\omega(\theta), \omega^*) \le {1\over 2} + c_0\sqrt{\Dtmax\Dtmin} + c_0 \Dtmax\le {1\over 2} + 2c_0\Dtmax.
    \end{equation}

   \begin{lemma}\label{lem_lip_J_basic}
   		Fix any  $\theta =(\pi, M, \sw)$ with $d(\sw,\sw^*)<1$. For any $a\ne \ell \in [L]$, one has
   		\[
   			  \|\sw^{*1/2}(J(\theta) - J(\theta^*))(\be_a-\be_{\ell})\|_2 \le    {\|\mu_a -\mu_a^*\|_{\sw^*} +  \|\mu_\ell -\mu_\ell^*\|_{\sw^*}\over 1 - d(\sw, \sw^*)}+ {d(\sw,\sw^*)\over 1 - d(\sw, \sw^*)}\sqrt{\Dt_{a\ell}}.
   		\]
   		Moreover, if $d(\sw, \sw^*) \le 1/2$, then   
   		\[
   			d(J(\theta),J(\theta^*))  \le   2d(M,M^*) +\sqrt{\Dtmax}~ d(\sw, \sw^*).
   		\]
   \end{lemma}
	\begin{proof} 
		Weyl's inequality ensures that 
		\begin{equation}\label{eigs_Sw}
			1-d(\sw, \sw^*)  \le \lambda_j(\sw^{*-1/2}\sw \sw^{*-1/2}) \le 1+d(\sw, \sw^*),\quad \forall ~ 1\le j\le d.
		\end{equation}
		Thus  $\sw$ is invertible.  The first result follows from 
		\begin{align}\label{decomp_J_diff}\nonumber
			&\|\sw^{*1/2}(J(\theta) - J(\theta^*))(\be_a -\be_{\ell})\|_2\\\nonumber
			&=\left\|\sw^{*1/2} \left(\sw^{-1}(M - M^*) + \sw^{-1}(\sw^* - \sw) \sw^{*-1}M^*\right)(\be_a -\be_{\ell})\right\|_2\\ \nonumber
			& \le    \| \sw^{*-1/2}\sw^{-1}\sw^{*-1/2}\|_\op\left\{\|(M - M^*)(\be_a -\be_{\ell})\|_{\sw^*}+d(\sw,\sw^*) \|  \mu_a^* -\mu_{\ell}^*\|_{\sw^*}\right\}\\  
			&\le  {1\over 1 - d(\sw, \sw^*)}\left\{\|\mu_a -\mu_a^*\|_{\sw^*} +  \|\mu_\ell -\mu_\ell^*\|_{\sw^*}+  d(\sw,\sw^*)\sqrt{\Dt_{a\ell}}\right\}
		\end{align}
		where the last step uses  \eqref{eigs_Sw}. The second result follows trivially from the first one.
	\end{proof}

    For $\sw = \sT - M \diag(\pi) M^\T$ and $ \sw'= \sT' - M' \diag(\pi') M^{'\T},$
    the following lemma relates $d(\sw, \sw')$ to $d(\pi,\pi')$, $d(M, M')$ and $d(\sT,\sT')$ 
    with 
    $$
        d(\sT, \sT') := \|\sw^{*-1/2}(\sT - \sT')\sw^{*-1/2}\|_\op.
    $$
    
    \begin{lemma}\label{lem_lip_sw} Fix any  $\theta =(\pi, M, \sw)$ with $
    	\sw = \sT - M \diag(\pi) M^\T 
    	$ and  $\theta' =(\pi', M', \sw')$ with 
        $
        \sw':= \sT' - M' \diag(\pi') M^{'\T}.
        $
        \begin{enumerate}  
        \item[(a)] 
        Taking $\theta'=\theta^*$, if  $d(M,M')\le  \|\sw^{*-1/2}M'\|_{2,\i} $, then 
        \begin{equation}\label{lip_sw}
        d(\sw, \sw')\le  3 \|\sw^{*-1/2}M'\|_{2,\i}  d(M,M') + 2 \|\sw^{*-1/2}M'\|_{2,\i} ^2 d(\pi,\pi')+d(\sT, \sT').
        \end{equation}
        \item [(b)] If  
        \begin{equation}\label{cond_dist_EM_Sigma}
        	6\sqrt{\Dtmax}d(M,M^*) + 4\Dtmax d(\pi,\pi^*) + 2d(\sT, \sT^*) \le 1,
        \end{equation}
        we have 
        \[
         	d(J(\theta),J(\theta^*)) \le 8\Dtmax d(M, M^*) + 4\Dtmax^{3/2}d(\pi,\pi^*) + 2\sqrt{\Dtmax} d(\sT, \sT^*).
        \] 
        \end{enumerate} 
    \end{lemma}
    \begin{proof}
    	 Part (a) follows by  
        \begin{align*}\nonumber
            d(\sw, \sw') & \le \|\sw^{*-1/2}M' \diag(\pi') (M'-M)^\T \sw^{*-1/2}\|_\op\\
            &\quad +\|\sw^{*-1/2}M'(\diag(\pi')  -  \diag(\pi))M^\T \sw^{*-1/2}\|_\op\\\nonumber
            &\quad +\|\sw^{*-1/2}(M'-M) \diag(\pi) M^{\T}  \sw^{*-1/2}\|_\op+ \|\sw^{*-1/2}(\sT - \sT')\sw^{*-1/2}\|_\op\\\nonumber
            &\le  \|\sw^{*-1/2}M'\|_{2,\i} d(M,M') + d(\pi,\pi')  \|\sw^{*-1/2}M'\|_{2,\i} \left[ \|\sw^{*-1/2}M'\|_{2,\i} + d(M,M')\right]\\\nonumber
            &\quad + d(M,M') \left[\|\sw^{*-1/2}M'\|_{2,\i} + d(M,M')\right] + d(\sT, \sT')\\
            &\le  3\|\sw^{*-1/2}M'\|_{2,\i}  d(M,M') + 2\|\sw^{*-1/2}M'\|_{2,\i}^2 d(\pi,\pi')+d(\sT, \sT').
        \end{align*}
        The last step uses $d(M,M')\le \|\sw^{*-1/2}M'\|_{2,\i} $. 
        
       To prove part (b), since condition \eqref{cond_dist_EM_Sigma} ensures that $d(\sw,\sw^*)\le 1/2$ holds, the last claim then follows by plugging the above bound with $\theta' = \theta^*$ into the result of \cref{lem_lip_J_basic}. 
    \end{proof}
    
\subsection{Proof of \cref{thm_EM_popu_unknown}: contraction of the population level EM iterates}\label{app_sec_proof_thm_EM_population}

	The following theorem states the one-step contraction property of the population level M-steps $\bpi(\omega)$, $\bM(\omega)$ and $\bsw(\omega)$ for any $\omega \in \cB_d(\omega^*)$. Recall $d(\pi,\pi^*)$ from \eqref{def_dist_omega}.
    
    \begin{theorem}\label{thm_EM_population_omega}
        Grant condition \eqref{cond_min_Delta}. For some absolute constants $c,C>0$, let 
        \begin{equation}\label{def_kappa_omega}
            \kappa_{\omega} =  \frac{C}{\pmin}\exp{(-c\Dtmin)}.
        \end{equation} 
        For any $\omega = (\pi, M, J) \in \cB_d(\omega^*)$, 
        we have 
        \begin{align}\label{rate_contract_pi}
            d(\bpi(\omega), \bpi(\omega^*)) & ~ \le ~  \kappa_{\omega} ~ d(\omega,\omega^*),\\\label{rate_contract_mu}
         d(\bM(\omega), \bM(\omega^*))  & ~ \le ~   \kappa_{\omega}  \sqrt{\Dtmax}  ~ d(\omega,\omega^*),\\\label{rate_contract_sw}
         d(\bsw(\omega), \bsw(\omega^*)) &~ \le ~  \kappa_{\omega}~  \Dtmax ~  d(\omega,\omega^*).
        \end{align}  
    \end{theorem}
    \begin{proof}
        The proofs of \eqref{rate_contract_pi} and \eqref{rate_contract_mu} can be found in \cref{app_sec_proof_rate_contract_pi_w_J} and \cref{app_sec_proof_rate_contract_mu_w_J}, respectively. The proof of \eqref{rate_contract_sw} uses \eqref{lip_sw} of \cref{lem_lip_sw} with $\theta' = \theta^*$ and $\sT = \sT^*$.
    \end{proof}

 Combining \cref{thm_EM_population_omega} and \cref{lem_lip_J_basic}, we are ready to prove \cref{thm_EM_popu_unknown}.

\begin{proof}[Proof of \cref{thm_EM_popu_unknown}]
    Fix any $\theta = (\pi, M, \sw)$ satisfying \eqref{cond_init_unknown}. By induction, it suffices to prove 
    \begin{equation}\label{one_step}
        d(\bpi(\theta), \pi^*) +  d(\bM(\theta), M^*) + d(\bsw(\theta), \sw^*) \le \kappa \left\{d(\pi,\pi^*) +   d(M,M^*)+  d(\sw, \sw^*)  \right\}
    \end{equation}
    and the updated parameters $\bpi(\theta), \bM(\theta),\bsw(\theta)$ still satisfy \eqref{cond_init_unknown}. 

    To prove \eqref{one_step},  since condition \eqref{cond_init_unknown}  ensures   $\omega = (\pi, M, J(\theta)) \in \cB_d(\omega^*)$, one has
    \begin{align*}
           d(\bpi(\theta), \pi^*)
         &    =   d(\bpi(\omega(\theta)), \bpi(\omega(\theta^*)))\\
         &  \le    \kappa_{\omega}  \left\{d(\pi,\pi^*) + \sqrt{\Dtmax} ~ d(M,M^*) +\sqrt{\Dtmax}~ d(J(\theta), J^*) \right\} &&\text{by \cref{thm_EM_population_omega}}\\
         &   \le \kappa_{\omega}  \left\{d(\pi,\pi^*) + 3\sqrt{\Dtmax} ~ d(M,M^*) + {\Dtmax}~ d(\sw, \sw^*) \right\} &&\text{by \cref{lem_lip_J_basic}}
    \end{align*}
    with $\kappa_{\omega}$ given by \eqref{def_kappa_omega}. 
    Similarly, we also have 
    \begin{align*}
         d(\bM(\theta), M^*)  & ~ \le ~   \kappa_{\omega}  \sqrt{\Dtmax}  \left\{d(\pi,\pi^*) + 3\sqrt{\Dtmax}~  d(M,M^*)+ {\Dtmax}~ d(\sw, \sw^*)  \right\}.
    \end{align*} 
    By \cref{lem_lip_sw} with $\sT = \sT^*$, we further have 
    \begin{align*}
         d(\bsw(\theta), \sw^*)    ~ &\le ~ 3 \sqrt{\Dtmax} ~ d(\bM(\theta),M^*) + 2\Dtmax d(\bpi(\theta),\theta^*) \\
         &~\le ~5 \kappa_{\omega}~  \Dtmax   \left\{  d(\pi,\pi^*) +3\sqrt{\Dtmax}~  d(M,M^*)+{\Dtmax}~ d(\sw, \sw^*)  \right\},
    \end{align*}
    thereby proving \eqref{one_step} with $\kappa = 21\Dtmax^2\kappa_w$. Furthermore, since condition \eqref{cond_min_Delta} ensures 
    \begin{align*}
    	 d(\bpi(\theta), \pi^*) & \le \kappa_w \left\{
    	{1\over 2} + 3c_\mu \sqrt{\Dtmax \Dtmin} +  \Dtmax c_\sw
    	\right\} \le {1\over 2},\\
    	 d(\bM(\theta), M^*)  & ~ \le ~   \kappa_{\omega}  \sqrt{\Dtmax}  \left\{{1\over 2}+ 3c_\mu \sqrt{\Dtmax \Dtmin} + {\Dtmax} c_\sw\right\} \le c_\mu \sqrt{\Dtmin},\\
    	 d(\bsw(\theta), \sw^*)  
    	 &\lesssim 5\kappa_{\omega}~  \Dtmax   \left\{ {1\over 2}+ 3c_\mu \sqrt{\Dtmax\Dtmin}+{\Dtmax} c_\sw  \right\} \le c_\sw,
    \end{align*} 
    we conclude that $\bpi(\theta), \bM(\theta)$ and $\bsw(\theta)$ still satisfy condition \eqref{cond_init_unknown} so that using induction completes the proof.
\end{proof}


\subsubsection{Proof of \cref{rate_contract_pi} of \cref{thm_EM_population_omega}: Contraction of the population level M-step of the mixing probabilities}\label{app_sec_proof_rate_contract_pi_w_J}

\begin{proof} 
    Fix any $\bar\omega = (\bar\pi, \bar M, \bar J)\in \cB_d(\omega^*)$  given by \eqref{cond_init_w_J} and pick any $\ell \in [L]$. Our goal is to bound from above 
    \[
        \left|\bpi_\ell(\bar \omega) - \bpi_\ell(\omega^*)\right| =  \left|\EE[\gamma_\ell (X;\omega^*)] - \EE[\gamma_\ell(X; \bar \omega)]\right|.
    \]
    For any $u\in (0,1)$, by writing 
    $$ \pi_u = (1-u)\pi^* + u \bar\pi,\quad  M_u = (1-u)M^*+u\bar M,\quad  J_u = (1-u)J^* + u\bar J,$$ 
    we have  
    \begin{align*}
        &\bpi_\ell(\bar \omega) - \bpi_\ell(\omega^*)
         = \nabla^\pi \bpi_\ell(\bar\omega) +  \nabla^M \bpi_\ell(\bar\omega) +  \nabla^J \bpi_\ell(\bar\omega)
    \end{align*}
    where 
    \begin{equation}\label{def_nabla_pi}
    \begin{split}
        \nabla^\pi \bpi_\ell(\bar\omega) &:= \sum_{k=1}^{L-1}\int_0^1 \EE \left\langle \frac{\partial \gamma_\ell(X;\omega)}{\partial \pi_k}\Big|_{\omega=(\pi_u,\bar M,\bar J)},\bar \pi_k - \pi_k^*\right\rangle \rd u \\ 
        \nabla^M \bpi_\ell(\bar\omega) &:= 
          \sum_{k=1}^L \int_0^1 \EE  \left\langle \frac{\partial\gamma_{\ell}(X;\omega)}{\partial \mu_k}\Big|_{\omega=(\pi^*,M_u,\bar J)},\bar \mu_k - \mu_k^*\right\rangle \rd u\\
          \nabla^J \bpi_\ell(\bar\omega) &:=  
        \int_0^1 \EE  \left\langle \frac{\partial\gamma_{\ell}(X;\omega)}{\partial J}\Big|_{\omega=(\pi^*,M^*,J_u)}, \bar J - J^*\right\rangle \rd u.
    \end{split}
    \end{equation} 
    Below we bound each terms separately in terms of
    \begin{equation}\label{def_cRs}
        \begin{split}
            \cR_I^{(0)} &:= \sup_{\omega\in\cB_d(\omega^*)}  \max_{\ell \in [L]} \sum_{a \in [L]\setminus\{\ell\}} \pi_a \EE\left[I_{a\ell}(X,\omega)\right],\\
            \cR_I^{(1)} &:= \sup_{\omega\in\cB_d(\omega^*)}\max_{\ell \in [L]}~   \Bigl\|\sum_{a \in [L]\setminus\{\ell\}} \pi_a\EE\left[I_{a\ell}(X,\omega) NN^\T \right]\Bigr\|_\op,\\ \cR_I^{(2)} &:= \sup_{\omega\in\cB_d(\omega^*)} \max_{\ell \in [L]}   \sum_{a \in [L]\setminus\{\ell\}} \pi_a \left\| \EE\left[I_{a\ell}(X,\omega) NN^\T\right]\right\|_\op,
        \end{split}
    \end{equation}
    where, for any $a,\ell \in [L]$ and $x\in \RR^L$, we write 
    \begin{align}\label{def_I_aell}
        I_{a\ell}(x,\omega) &:= \frac{\exp\{(x - \frac{1}{2}(\mu_\ell + \mu_a))^\T J (\be_a - \be_\ell)\}}{(\pi_\ell + \sum_{k\neq \ell} \pi_k \exp\{(x - \frac{1}{2}(\mu_\ell + \mu_k))^\T J (\be_k - \be_\ell\})^2}.
    \end{align} 
	Clearly, we have $\cR_I^{(1)}\le \cR_I^{(2)}$.

\paragraph{Controlling the partial derivative w.r.t. $\pi_k$'s.}
We bound from above 
\[
  | \nabla^\pi \bpi_\ell(\bar\omega)| ~ \le   \sup_{\omega \in \cB_d(\omega^*)}\left| \sum_{k=1}^{L-1}\EE\left[ \frac{\partial \gamma_\ell(X;\omega)}{\partial \pi_k} \bigg |_\omega    \right] \left(\bar \pi_k - \pi_k^*\right)\right|.
\]
Pick any $\omega \in \cB_d(\omega^*)$. Recall that for any  $x\in \RR^d$ and $\ell \in [L]$,  
\[
    \gamma_\ell(x;\omega) = \frac{\pi_\ell}{ \sum_{k =1}^L \pi_k \exp\left\{(x-{1\over 2}( \mu_\ell+\mu_k))^\T  J (\be_k-\be_\ell)\right\}}.
\] 
Since $\pi_L = 1 - \sum_{\ell = 1}^{L-1}\pi_\ell$, basic calculation gives that: for all  $1\leq k,\ell \leq L-1$,
\begin{align}\label{eq_deriv_pi_1}
    \frac{\partial \gamma_\ell (x;\omega)}{\partial \pi_k} &= 
     \pi_\ell I_{L\ell}(x,\omega)-\pi_\ell  I_{k\ell}(x,\omega), \qquad  k\neq \ell,\\\label{eq_deriv_pi_2}
    \frac{\partial \gamma_\ell(x;\omega)}{\partial \pi_\ell} &= 
    \sum_{a \in [L]\setminus\{\ell\}} \pi_a I_{a \ell}(x,\omega)  + \pi_\ell I_{L\ell}(x,\omega),\\\label{eq_deriv_pi_3}
    \frac{\partial \gamma_L(x;\omega)}{\partial \pi_k} &= 
    \sum_{a=1}^{L-1} \pi_a I_{aL}(x,\omega) - \pi_L I_{kL}(x,\omega), 
\end{align}
so that for all $1\le k,\ell < L$ with $k\ne \ell$, 
\begin{align}\label{bd_part_deriv_gamma_pi}\nonumber
 | \nabla^\pi \bpi_\ell(\bar\omega)| 
& ~ \le  \sup_{\omega \in \cB_d(\omega^*)}\left| \sum_{k<L, k\ne \ell}\EE\left[  \pi_\ell I_{L\ell}(X,\omega)-\pi_\ell  I_{k\ell}(X,\omega)   \right] \left(\bar \pi_k - \pi_k^*\right)\right. \\\nonumber
&\qquad\qquad +   \left.  \sum_{a \in [L]\setminus\{\ell\}} \pi_a \EE\left[I_{a \ell}(X,\omega) \right]\left(\bar \pi_\ell - \pi_\ell^*\right)+ \pi_\ell \EE\left[ I_{L\ell}(X,\omega)\right] \left(\bar \pi_\ell - \pi_\ell^*\right)\right|\\\nonumber
&~ \le   \sup_{\omega \in \cB_d(\omega^*)}   \Bigl\{  (\pi_\ell + \pi_\ell^*) \sum_{a \in [L]\setminus\{\ell\}} \pi_a\EE[I_{a \ell}(X,\omega)] \Bigr\}\max_{k<L}  \frac{|\bar \pi_k - \pi_k^*|}{\pi_k^*} \\\nonumber
&\qquad +   \sup_{\omega \in \cB_d(\omega^*)} \pi_\ell\EE\left[  I_{L\ell}(X,\omega)  \right]\left|   \sum_{k<L} \left(\bar \pi_k - \pi_k^*\right)\right|\\
  & ~   \lesssim  ~   \pi_{\ell}^*~   \cR_I^{(0)}  d(\bar\pi, \pi^*). 
\end{align}
In the last step, we  use $\sum_{k< L} (\bar \pi_k -\pi_k^*) = \pi_L^*-\bar \pi_L$, the definition of $ \cR_I^{(0)}$ in \eqref{def_cRs} and 
\begin{equation}\label{eq_order_pis}
   1/2\le  \pi_\ell/ \pi_\ell^* \le 3/2,\qquad \forall~  \ell\in [L],
\end{equation}
which is implied by $d(\pi,\pi^*)\le 1/2$ from \eqref{cond_init_w_J}.   By \eqref{eq_deriv_pi_3}, the bound in \eqref{bd_part_deriv_gamma_pi} also holds for $\ell = L$ as 
\begin{align*}
	| \nabla^\pi \bpi_L(\bar\omega)| 
	& ~ \le  \sup_{\omega \in \cB_d(\omega^*)}\left|\sum_{k=1}^{L-1}\left(\sum_{a=1}^{L-1} \pi_a \EE\left[I_{aL}(X,\omega)\right] - \pi_L\EE \left[ I_{kL}(X,\omega)\right] \right)  \left(\bar \pi_k - \pi_k^*\right)\right|\\
	&~ \le   \sup_{\omega \in \cB_d(\omega^*)}   \left|{\pi_L^*-\bar \pi_L\over \pi_L^*} \pi_L^*\sum_{a=1}^{L-1} \pi_a \EE\left[I_{aL}(X,\omega)\right] -\pi_L  \sum_{k=1}^{L-1}\pi_k^*\EE \left[ I_{kL}(X,\omega)\right]{\bar \pi_k - \pi_k^*\over \pi_k^*} \right|\\
	& ~   \lesssim  ~   \pi_L^*~  \cR_I^{(0)}~  d(\bar\pi, \pi^*). 
\end{align*}

\paragraph{Controlling the partial derivative w.r.t. $\mu_k$'s}

We bound from above 
\[
   | \nabla^M \bpi_\ell(\bar\omega)| ~ \le \sup_{\omega \in \cB_d(\omega^*)} \left|\sum_{k=1}^L \EE\left[\left\langle \frac{\partial \gamma_\ell(X;\omega)}{\partial \mu_k}\bigg |_\omega ,~  \bar \mu_k -\mu_k^*\right\rangle \right] \right|.
\]
For any $k\neq \ell$, note that  
    \begin{align}\label{eq_deriv_mu_1}
        \frac{\partial \gamma_\ell(x;\omega)}{\partial \mu_k} &= 
          \frac{\pi_\ell \pi_k}{2} I_{k\ell}(x,\omega)  J(\be_k - \be_\ell),\\ \label{eq_deriv_mu_2}
        \frac{\partial \gamma_\ell(x;\omega)}{\partial \mu_\ell} &= 
        \sum_{a\in[L]\setminus\{\ell\}}\frac{\pi_\ell \pi_a}{2} I_{a\ell}(x,\omega)   J(\be_a - \be_\ell)  = \sum_{a\in[L]\setminus\{\ell\}} \frac{\partial \gamma_\ell(x;\omega)}{\partial \mu_a}.
    \end{align}
Fix any $a\in [L]$ with $a\ne \ell$. By using \eqref{eq_deriv_mu_1}, we have 
\begin{align}
     \left| \EE  \left\langle \frac{\partial \gamma_\ell(X;\omega) }{\partial \mu_a}, \bar  \mu_a - \mu_a^* \right\rangle \right| 
    &= {\pi_\ell \pi_a\over 2} \EE  \left[I_{a\ell}(X,\omega) \right]  \left|  (\bar  \mu_a - \mu_a^*)^\T J (\be_a - \be_\ell) \right|  \notag \\
    &\leq {\pi_\ell \pi_a\over 2}  \EE  \left[I_{a\ell}(X,\omega) \right] \| \sw^{*1/2} J (\be_a - \be_\ell) \|_2 \|\bar  \mu_a - \mu_a^*\|_{\sw^*},\notag\\
    &\le   \pi_\ell \pi_a  \sqrt{\Delta_{a\ell}}  ~ \EE  \left[I_{a\ell}(X,\omega) \right] \|  \bar  \mu_a - \mu_a^* \|_{\sw^*} \label{bd_deriv_gamma_mu_a}
\end{align}
where in the last step we used 
\begin{align}
    &\| \sw^{*1/2} J (\be_a - \be_\ell) \|_2\notag\\ 
    &\le  \|\sw^{*1/2}  (J - J^*) (\be_a - \be_\ell) \|_2  +  \| \sw^{*1/2} J^* (\be_a - \be_\ell) \|_2 \notag\\
    & \le  2\sqrt{\Delta_{a\ell}} &&\text{by \eqref{cond_init_w_J} \& \eqref{def_Delta}}.\label{bd_J_col_diff}
\end{align}
Similarly, we also have 
\begin{align}
     \left| \EE  \left\langle \frac{\partial \gamma_\ell(X;\omega) }{\partial \mu_\ell}, \bar \mu_\ell - \mu_\ell^* \right\rangle \right|  
    &\leq   \pi_\ell  \sum_{a\in[L]\setminus\{\ell\}} \pi_a  \EE[I_{a\ell}(X,\omega)] \sqrt{\Delta_{a\ell}} ~   \|\bar  \mu_\ell - \mu_\ell^*\|_{\sw^*} \notag \\
    &\leq   \pi_{\ell} \sqrt{\Dtmax} ~ \cR_I^{(0)} \|\bar  \mu_\ell - \mu_\ell^*)\|_{\sw^*}. \label{db_deriv_gamma_mu_l}
\end{align}
Combining \eqref{bd_deriv_gamma_mu_a} and \eqref{db_deriv_gamma_mu_l} together with \eqref{eq_order_pis} gives that for all $ \ell \in [L]$,
\begin{align}\label{bd_part_deriv_gamma_mu}
       | \nabla^M \bpi_\ell(\bar\omega)| ~ \lesssim ~  \pi_\ell^*   \sqrt{\Dtmax} ~ \cR_I^{(0)}  d(\bar  M, M^*). 
\end{align}

\paragraph{Controlling the partial derivative w.r.t. $J$}    
   We proceed to bound 
    $$
     | \nabla^J \bpi_\ell(\bar\omega)| ~ \le \sup_{\omega \in \cB_d(\omega^*)} \left|\EE\left[\left\langle \frac{\partial \gamma_\ell(X;\omega)}{\partial J}\bigg|_{\omega}, ~ \bar J- J^*\right\rangle \right]\right|.
    $$
    Note that the partial derivatives of $\gamma_\ell(x,\omega)$ with respect to $J$ is
    \begin{align}\label{eq_deriv_J}
        \frac{\partial \gamma_\ell(x;\omega)}{\partial J} = 
        -  \sum_{a\in[L]\setminus\{\ell\}}  \pi_\ell \pi_a I_{a\ell}(x,\omega) \left(x - \frac{\mu_\ell + \mu_a}{2}\right)(\be_a - \be_\ell)^\T \in \RR^{d\times L}
    \end{align}
    with $I_{a\ell}(x,\omega)$ given in \eqref{def_I_aell}. 
    Since 
    \begin{align*}
        \left\langle  \left(X-\frac{\mu_a+\mu_\ell}{2} \right) (\be_a - \be_\ell)^\T,  \bar  J-J^*  \right \rangle
        = ~ &  
        \left(X-\frac{\mu_a+\mu_\ell}{2} \right)^\T  (\bar J-J^*) (\be_a - \be_\ell),
    \end{align*} 
    the main task is to  analyze
    \begin{align}
        &  \sum_{a\in[L]\setminus\{\ell\}}  \pi_a \EE  \left[ I_{a\ell}(X,\omega) \left(X - \frac{\mu_\ell +  \mu_a}{2}\right)^\T (\bar J-J^*) (\be_a - \be_\ell)   \right] \notag \\ 
         &= ~    \sum_{a\in[L]\setminus\{\ell\}}  \pi_a\sum_{b=1}^L \pi_b^*   \EE  \left[  I_{a\ell}(X,\omega)  \mid Y = b\right]  \left(\mu_b^* - \frac{\mu_\ell + \mu_a}{2}\right)^\T  (\bar J-J^*) (\be_a - \be_\ell)   \label{1} \\
        &\qquad +   \sum_{a\in[L]\setminus\{\ell\}}  \pi_a \EE  \left[ I_{a\ell}(X,\omega) ~   N^\T  \sw^{*1/2} (\bar J-J^*) (\be_a - \be_\ell)  \right]\label{2}
    \end{align}
    where we use the fact that 
    \begin{align}\label{cond_distr_Z}
         X ~  \stackrel{d}{=} ~ \sw^{*1/2} N + \sum_{b=1}^L 1\{Y=b\} \mu_b^*
    \end{align} 
    with $N\sim \cN_d(0,\bI_d)$.  
Since condition \eqref{cond_init_w_J} ensures that 
\begin{align}
    \| \mu_\ell^* - (\mu_{\ell}+\mu_a)/2\|_{\sw^*} 
    \leq {\sqrt{\Dt_{a\ell}}\over 2}   +{1\over 2}\|\mu^*_\ell - \mu_\ell\|_{\sw^*}+{1\over 2}\|\mu^*_a - \mu_a\|_{\sw^*}  
    \leq  \sqrt{\Delta_{a\ell}}, \label{mu_diff} 
\end{align} 
by the definition in \eqref{def_cRs}, we find 
\begin{align*}
    \eqref{1}  
    &\leq \sum_{a\in[L]\setminus\{\ell\}}  \pi_a \EE[I_{a\ell}(X,\omega)] \max_{b\in[L]}\left( 
    	\sqrt{\Dt_{b\ell}} + \sqrt{\Dt_{a\ell}}
    \right) \left\|\sw^{*1/2} (\bar J-J^*) (\be_a - \be_\ell) \right\|_2  
    \\ 
    & \le   4  \sqrt{\Dtmax} ~ \cR_I^{(0)} ~ d(\bar  J, J^*).
\end{align*}
To bound \eqref{2}, applying Cauchy-Schwarz inequality twice and  using  $\|\sw^{*1/2} (\bar J-J^*) (\be_a - \be_\ell)\|_2\le 2d(\bar J,J^*)$ yield
\begin{align}\label{CauchySchwarz1}\nonumber
 &\sum_{a\in[L]\setminus\{\ell\}}  \pi_a \EE  \left[ I_{a\ell}(X,\omega) ~   N^\T  \sw^{*1/2} (\bar J-J^*) (\be_a - \be_\ell)  \right]\\\nonumber
 &\leq  2d(\bar J,J^*) \sqrt{ \sum_{a \in [L]\setminus\{\ell\}} \pi_a \EE[I_{a\ell}(X,\omega)]}  \sqrt{\sum_{a \in [L]\setminus\{\ell\}} \pi_a \|\EE[I_{a\ell}(X,\omega) NN^\T \|_\op}\\
	&\leq  2d(\bar J,J^*)  \sqrt{    \cR_I^{(0)}  \cR_I^{(2)}}. 
\end{align}
The last step use the definitions of $\cR_I^{(0)}$ and $\cR_I^{(2)}$ in \eqref{def_cRs}.  
Combining the bounds of \eqref{1} and \eqref{2} and using \eqref{eq_order_pis} give 
\begin{align}\label{bd_part_deriv_gamma_J}
    & | \nabla^J \bpi_\ell(\bar\omega)|   ~ \lesssim \pi_{\ell}^*  \left(  \sqrt{    \cR_I^{(0)}  \cR_I^{(2)}} + \sqrt{\Dtmax} ~ \cR_I^{(0)}\right)    ~ d(\bar J,J^*).
\end{align}

Finally, by collecting the bounds in \cref{bd_part_deriv_gamma_pi,bd_part_deriv_gamma_J,bd_part_deriv_gamma_mu}, we conclude that
\begin{align*}
      {|\bpi_\ell(\bar \theta) - \pi_\ell^*|\over \pi_\ell^*}  & \le{1 \over \pi_\ell^*}\left|\EE[\gamma_\ell(X;\bar \omega)] - \EE[\gamma_\ell (X;\omega^*)] \right|\\ &\lesssim  \sqrt{    \cR_I^{(0)}  \cR_I^{(2)}}   d(\bar J,J^*) +  \cR_I^{(0)}  d(\bar \omega, \omega^*)\\
      &\lesssim \kappa_w ~ d(\bar \omega, \omega^*) &&\text{by  \cref{lem_I_fN}}.
\end{align*}
Tis completes the proof of \eqref{rate_contract_pi}.
\end{proof}

\subsubsection{Proof of \cref{rate_contract_mu} of \cref{thm_EM_population_omega}: Contraction of the population level M-step of the mean vectors}\label{app_sec_proof_rate_contract_mu_w_J}

\begin{proof}   
For any $\omega = (\pi, M,J)$ and $\ell \in [L]$, 
note that  
$$ 
\bmu_\ell(\omega) -\mu_{\ell}^*= \frac{\EE[\gamma_\ell(X;\omega)(X-\mu_{\ell}^*)]}{\EE[\gamma_\ell(X;\omega) ]} = {1\over \bpi_{\ell}(\omega)}\EE\left[(\gamma_\ell(X;\omega)-\gamma_\ell(X;\omega^*))(X-\mu_{\ell}^*)\right].
$$  
Fix any  $\bar \omega \in \cB_d(\omega^*)$ and any $\ell \in [L]$. 
It then follows that 
\begin{align*}
\bmu_\ell(\bar \omega) -\bmu_\ell(\omega^*)   &= \nabla^{\pi}\bmu_\ell(\bar\omega) + \nabla^{M}\bmu_\ell(\bar\omega)  + \nabla^{J}\bmu_\ell(\bar\omega) 
\end{align*}
where
\begin{equation}\label{def_nabla_mu}
\begin{split} 
 \nabla^{\pi}\bmu_\ell(\bar\omega)   
 &:=  \sum_{k=1}^{L-1} \int_0^1   
\frac{1}{\bpi_\ell(\omega)}\EE\left[ \left(X-\mu_\ell^*\right) \left\langle { \partial \gamma_{\ell}(X;\omega)\over \partial \pi_k }\Big|_{\omega = (\pi_u, \bar M, \bar J)}, ~ \left(\bar \pi_k - \pi_k^*\right) \right\rangle \right]  \rd u\\
 \nabla^{M}\bmu_\ell(\bar\omega)&:=  \sum_{k=1}^{L} \int_0^1  \frac{1}{\bpi_\ell(\omega)}\EE\left[ \left(X-\mu_\ell^*\right) \left\langle {\partial \gamma_{\ell}(X;\omega) \over \partial \mu_k} \Big|_{\omega = (\pi^*, M_u, \bar J)}, ~ \bar \mu_k - \mu_k^* \right\rangle \right] \rd u \\
 \nabla^{J}\bmu_\ell(\bar\omega) & :=   \int_0^1   \frac{1}{\bpi_\ell(\omega)}\EE\left[ \left(X-\mu_\ell^*\right) \left\langle{\partial \gamma_{\ell}(X;\omega)\over \partial J}\Big|_{\omega = (\pi^*, M^*, J_u)}, ~ \bar J-J^* \right\rangle  \right]  \rd u . 
\end{split}
\end{equation}
The main task below is to bound each terms  in terms of $\cR_I^{(j)}$ with $j\in \{0,1,2\}$.
For future reference, we notice that 
\begin{equation}\label{rate_equiv_pi}
      1/2  \le \bpi_\ell(\omega) /\pi_\ell^* \le 3/2,\qquad \forall~  \omega \in \cB_d(\omega^*),
\end{equation}
since 
\begin{equation}
    d(\bpi(\omega),\bpi(\omega^*)) \overset{\eqref{rate_contract_pi}}{\le} \kappa_w d(\omega,\omega^*) \overset{\eqref{init_dist_omega}}{\le} 3   \kappa_w \Dtmax \overset{\eqref{cond_min_Delta}}{\le} 1/2.\label{bd_dist_omega_ball}
\end{equation}

\paragraph{Controlling the partial derivative of $\bmu_\ell(\omega)$ w.r.t. $\pi_k$'s}
We need to bound
from above 
\[
  \left\|\nabla^{\pi}\bmu_\ell(\bar\omega)\right\|_{\sw^*} \le \sup_{\omega \in \cB_d(\omega^*)}  \frac{1}{\bpi_\ell(\omega)} \left\|  \sum_{k=1}^{L-1}(\bar \pi_k - \pi_k^*) ~ \EE\left[ \left(X-\mu_\ell^*\right)   {\partial \gamma_{\ell}(X;\omega) \over \partial \pi_k}  \right]\right\|_{\sw^*}.
\]
From the expressions of $\partial \gamma_{\ell}(x;\omega)$ in  \cref{eq_deriv_pi_1,eq_deriv_pi_2,eq_deriv_pi_3} as well as the definition of $I_{a\ell}(x,\omega)$ in \eqref{def_I_aell},  
by similar arguments of proving \eqref{bd_part_deriv_gamma_pi}, we obtain that for any $\ell <L$, 
\begin{align}\label{bd_part_mu_pi}\nonumber
   \left\|\nabla^{\pi}\bmu_\ell(\bar\omega)\right\|_{\sw^*}  &\lesssim d(\bar\pi,\pi^*) \sup_{\omega \in \cB_d(\omega^*)}  \frac{\pi_\ell + \pi_{\ell}^* }{\bpi_\ell(\omega)}  \left\|\sum_{a\in[L]\setminus\{ \ell\}} \pi_a \EE\left[I_{a\ell}(X,\omega)  (X-\mu_\ell^*) \right] \right\|_{\sw^*}\\ 
    &\lesssim   d(\bar\pi, \pi^*)  \sup_{\omega \in \cB_d(\omega^*)}  \left\| \sum_{a\in[L]\setminus\{ \ell\}} \pi_a \EE\left[I_{a\ell}(X,\omega)  (X-\mu_\ell^*) \right] \right\|_{\sw^*}.
\end{align}
The last step uses \eqref{eq_order_pis} and \eqref{rate_equiv_pi}.
For any $\omega \in \cB_d(\omega^*)$, using \eqref{cond_distr_Z} and the definitions of $\cR_I^{(0)}$ and $\cR_I^{(1)}$ in \eqref{def_cRs} gives
\begin{align}\nonumber
&\left\|  \sum_{a\in[L]\setminus\{ \ell\}}  \pi_a\EE\left[ I_{a\ell}(X,\omega)  (X -  \mu_\ell^*) \right] \right\|_{\sw^*}\\\nonumber 
 &\le   \left\|  \sum_{a\in[L]\setminus\{ \ell\}} \pi_a \EE \left[ I_{a\ell}(X,\omega)   N  \right] \right\|_2 + \left\|  \sum_{a\in[L]\setminus\{ \ell\}}\pi_a \sum_{b = 1}^L \pi_b^*   \EE \left[ I_{a\ell}(X,\omega)  \left(\mu_b^* - \mu_\ell^*\right) \mid Y = b \right] \right\|_{\sw^*}\\\nonumber
&\le \sqrt{\cR_I^{(0)}\cR_I^{(1)}}+    \cR_I^{(0)}  \max_{b\in [L]}\left\|  \mu_b^* - \mu_\ell^*\right\|_{\sw^*}\\\label{bd:expz}
&\le  \sqrt{\cR_I^{(0)}\cR_I^{(1)}}    +  \cR_I^{(0)} \sqrt{\Dtmax} .
\end{align}
We thus obtain that for any $\ell < L$,
\begin{align}\label{bd_part_deriv_M_pi} 
    \left\|\nabla^{\pi}\bmu_\ell(\bar\omega)\right\|_{\sw^*} & \lesssim   ~ \left(\sqrt{\cR_I^{(0)}\cR_I^{(1)}}    +  \cR_I^{(0)}\sqrt{\Dtmax}  \right)~   d(\bar\pi,\pi^*).
\end{align}
From \eqref{eq_deriv_pi_3} and by analogous argument, it is easy to see that  \eqref{bd_part_deriv_M_pi} also holds for $\ell = L$.

\paragraph{Controlling the partial derivative of $\bmu_\ell(\omega)$ w.r.t. $M$}
By the expression of $\partial \gamma_\ell(x;\omega)/\partial \mu_k$ in \eqref{eq_deriv_mu_1} -- \eqref{eq_deriv_mu_2} and the definition of $I_{a\ell}(x,\omega)$ in \eqref{def_I_aell},   $\|\nabla^{M}\bmu_\ell(\bar\omega)\|_{\sw^*} $ is bounded from above by
\begin{align*}
   &\sup_{\omega \in \cB_d(\omega^*)}  \left\| \sum_{k=1}^L  {1\over \bpi_\ell(\omega)}\EE \left[    (X - \mu_\ell^*)\left \langle\frac{\partial \gamma_\ell(X;\omega)}{\partial \mu_k}, \bar \mu_k - \mu_k^*\right\rangle\right]\right\|_{\sw^*}\\ 
    & \le  \sup_{\omega \in \cB_d(\omega^*)} {2\pi_\ell \over  \bpi_\ell(\omega)}\left\| \sum_{k\in [L]\setminus\{\ell\}}     \pi_k \EE\left[ I_{k\ell}(X,\omega)  (X - \mu_\ell^*) \right]  (\be_k -\be_\ell)^\T J^\T (\bar \mu_k - \mu_k^*)\right\|_{\sw^*}\\
    &\le   \sup_{\omega \in \cB_d(\omega^*)} 6 \left\| \sum_{k\in [L]\setminus\{\ell\}}     \pi_k \EE\left[ I_{k\ell}(X,\omega)  N \right]  (\be_k -\be_\ell)^\T J^\T (\bar \mu_k - \mu_k^*)\right\|_2\\
    &~ +   \sup_{\omega \in \cB_d(\omega^*)}  6\left\| \sum_{k\in [L]\setminus\{\ell\}}     \pi_k \EE\left[ I_{k\ell}(X,\omega)  \sum_{b=1}^L 1\{Y=b\} \right] (\mu_b^*- \mu_\ell^*) (\be_k -\be_\ell)^\T J^\T (\bar \mu_k - \mu_k^*)\right\|_{\sw^*}.
\end{align*}
The second step uses \eqref{eq_order_pis} and \eqref{rate_equiv_pi}.
By Cauchy-Schwarz inequality, the first term is no greater than 
\begin{align}\label{CauchySchwarz2}\nonumber
	 & \sup_{\omega \in \cB_d(\omega^*)}  6\Bigl\| \sum_{k\in [L]\setminus\{\ell\}}     \pi_k \EE\left[ I_{k\ell}(X,\omega)  NN^\T \right]  \Bigr\|_\op^{1/2}\\\nonumber
	 &\qquad \qquad \Bigl( \sum_{k\in [L]\setminus\{\ell\}}     \pi_k \EE\left[ I_{k\ell}(X,\omega)   \right] \Bigr)^{1/2}  \max_{k\in [L]} \left|(\be_k -\be_\ell)^\T J^\T (\bar \mu_k - \mu_k^*)\right|\\\nonumber
	 &\lesssim  \sqrt{\cR_I^{(0)}\cR_I^{(1)}} ~ \|\sw^{*1/2}J(\be_k -\be_\ell)\|_2  \|\bar \mu_k - \mu_k^*\|_{\sw^*}\\
	 &\lesssim  \sqrt{\cR_I^{(0)}\cR_I^{(1)}} ~ \sqrt{\Dtmax}~   d(\bar M,M^*)
\end{align}
with the last step due to  \eqref{bd_J_col_diff}.
Furthermore, the second term is bounded by 
\begin{align*}
	 6 \cR_I^{(0)} \max_{b\in[L]} \left\| \mu_b^*- \mu_\ell^*\right\|_{\sw^*} \max_{k\in [L]} \|\sw^{*1/2}J(\be_k -\be_\ell)\|_2  \|\bar \mu_k - \mu_k^*\|_{\sw^*} \lesssim   \Dtmax~  \cR_I^{(0)} ~ d(\bar M,M^*)
\end{align*}
using \eqref{bd_J_col_diff} and \eqref{rate_equiv_pi}. We thus have
\begin{align}\label{bd_part_deriv_M_mu} 
    \left\|\nabla^M\bmu_\ell(\bar\omega)\right\|_{\sw^*} ~   \lesssim ~    \left(\sqrt{\cR_I^{(0)}\cR_I^{(1)}} +  \cR_I^{(0)}\sqrt{\Dtmax} \right)  \sqrt{\Dtmax}~ d(\bar M,M^*). 
\end{align}  

\paragraph{Controlling the partial derivative of $\bmu_\ell(\omega)$ w.r.t. $J$}
We proceed to bound from above 
\begin{align*}
  \left\|\nabla^J \bmu_\ell(\bar\omega)\right\|_{\sw^*} & ~ \le   \sup_{\omega \in \cB_d(\omega^*)} \frac{1}{\bpi_\ell(\omega)} \left\|     \EE\left[ \left(X-\mu_\ell^*\right) \left\langle {\partial \gamma_{\ell}(X;\omega)\over \partial J}\Big|_\omega , \bar J-J^* \right\rangle  \right] \right\|_{\sw^*}.
\end{align*} 
 By \eqref{eq_deriv_J} and \eqref{rate_equiv_pi}, we need to bound the following term 
\begin{align*}
&\left\|  \sum_{a\in[L]\setminus\{\ell\}}  \pi_a\EE  \left[  (X - \mu_\ell^*) \left\langle I_{a\ell}(X,\omega) \left(X - \frac{\mu_{\ell} + \mu_{a}}{2}\right) (\be_{a} - \be_{\ell})^{\T}, \bar J-J^* \right\rangle \right] \right\|_{\sw^*} \\ \notag
& = \left\|    \sum_{a\in[L]\setminus\{\ell\}}  \pi_a\EE  \left[  I_{a\ell}(X,\omega)  \sw^{*-1/2}(X - \mu_\ell^*) \left(X - \frac{\mu_{\ell} + \mu_{a}}{2}\right)^{\T} (\bar J-J^*) (\be_{a} - \be_{\ell}) \right] \right\|_2. 
\end{align*}
Recalling \eqref{cond_distr_Z}, we have the decomposition
\begin{align}
 \sw^{*-1/2} (X - \mu_\ell^*) 
 \left( X - \frac{\mu_\ell + \mu_a}{2} \right)^\top   \sw^{*-1/2} 
\stackrel{d}{=} ~  N N^\T + T_1 + T_2 +T_3
\end{align}
where the matrices $T_1,T_2,T_3$ are given by
\begin{align*}
 T_1 &= \sw^{*-1/2}   \sum_{b=1}^L 1\{Y=b\}(\mu_b^* - \mu_\ell^*)  N^\top \\
T_2 &= N  \sum_{b=1}^L 1\{Y=b\}  \left(\mu_b^* - \frac{\mu_\ell + \mu_a}{2} \right)^\top  \sw^{*-1/2}\\
T_3 &= \sw^{*-1/2} \sum_{b=1}^L  1\{Y=b\} \left(\mu_b^* - \mu_\ell^*  \right) 
\sum_{b=1}^L 1\{Y=b\}  \left( \mu_b^* - \frac{\mu_\ell + \mu_a}{2} \right)^\T\sw^{*-1/2}.  
\end{align*} 
By using \cref{CauchySchwarz1}, we find
\begin{align*}
   &\left\|   \sum_{a\in[L]\setminus\{\ell\}}  \pi_a\EE\left[I_{a\ell}(X,\omega)  T_1  \right]\sw^{*1/2} (\bar J-J^*) (\be_{a} - \be_{\ell})\right\|_2 \\
   &\le  \left|    \sum_{a\in[L]\setminus\{\ell\}}  \pi_a\EE\left[I_{a\ell}(X,\omega)   N^\top  \right]\sw^{*1/2} (\bar J-J^*) (\be_{a} - \be_{\ell})\right|\max_{b\in[L]} \left\|  \mu_b^* - \mu_\ell^*\right\|_{\sw^*} \\ 
    &\le 2 \sqrt{\Dtmax}\sqrt{ \cR_I^{(0)}\cR_I^{(2)}}  ~ d(\bar J,J^*). 
\end{align*}
Similarly, by  \eqref{mu_diff} and the argument in \eqref{CauchySchwarz2},  we also find
\begin{align*}
     &\left\|   \sum_{a\in[L]\setminus\{\ell\}}  \pi_a \EE\left[I_{a\ell}(X,\omega)  T_2\right]\sw^{*1/2} (\bar J-J^*) (\be_{a} - \be_{\ell})\right\|_2\\
     &\le \max_{b\in[L]} \left\|   \sum_{a\in[L]\setminus\{\ell\}}  \pi_a \EE\left[I_{a\ell}(X,\omega)  N    \right]   \left(\mu_b^* - \frac{\mu_\ell + \mu_a}{2} \right)^\top (\bar J-J^*) (\be_{a} - \be_{\ell})\right\|_2\\ 
     &\leq 
     	 \Bigl\|   \sum_{a\in[L]\setminus\{\ell\}}  \pi_a \EE\left[I_{a\ell}(X,\omega)  N N^\T   \right]  \Bigr\|_\op ^{1/2} \\
     	 &\qquad \Bigl(
       \sum_{a\in[L]\setminus\{\ell\}}  \pi_a \EE\left[I_{a\ell}(X,\omega)   \right]  
     \Bigr)^{1/2}\max_{b\in[L]}  \left\|\mu_b^* - \frac{\mu_\ell + \mu_a}{2}\right \|_{\sw^*} \|\sw^{*1/2} (\bar J-J^*) (\be_{a} - \be_{\ell})\|_2 \\
      & \leq 2\sqrt{\Dtmax}   \sqrt{\cR_I^{(0)}\cR_I^{(1)}}  ~d(\bar J,J^*).
\end{align*}
as well as  
\begin{align*}
   &\left\|   \sum_{a\in[L]\setminus\{\ell\}}  \pi_a\EE\left[I_{a\ell}(X,\omega)   T_3\right]\sw^{*1/2} (\bar J-J^*) (\be_{a} - \be_{\ell})\right\|_2\\
   &=   \sum_{a\in[L]\setminus\{\ell\}}  \pi_a\EE\left[I_{a\ell}(X,\omega)   \right]   \max_{b\in[L]} \left\|  \mu_b^* - \mu_\ell^*\right\|_{\sw^*}
 \left\|\mu_b^* - \frac{\mu_\ell + \mu_a}{2} \right\|_{\sw^*} 2d(\bar J,J) \\  
   &\le  4\Dtmax  \cR_I^{(0)}d(\bar J,J).
\end{align*}
Finally, since one can very by using the Cauchy-Schwarz inequality that 
\begin{align*}
	\left\|  \sum_{a\in[L]\setminus\{\ell\}}  \pi_a\EE[I_{a\ell}(X,\omega) NN^\T]\sw^{*1/2} (\bar J-J^*) (\be_{a} - \be_{\ell})\right\|_2   \le 2\sqrt{\cR_I^{(1)}\cR_I^{(2)}} d(\bar J,J^*),
\end{align*}
combining with the previous three bounds gives 
\begin{align}\label{bd_part_deriv_M_J} 
    \left\|\nabla^{J}\bmu_\ell(\bar\omega)\right\|_{\sw^*}   & \lesssim  \left[\sqrt{\cR_I^{(1)}\cR_I^{(2)}}  +  \sqrt{\Dtmax}\left( \sqrt{\cR_I^{(0)}\cR_I^{(2)}}  +\sqrt{\Dtmax} \cR_I^{(0)}\right)\right]d(\bar J,J^*).
\end{align}
 
In view of \cref{bd_part_deriv_M_pi,bd_part_deriv_M_mu,bd_part_deriv_M_J} and using $\cR_I^{(1)}\le \cR_I^{(2)}$, we 
conclude that
\begin{equation}\label{bd_part_derive_M_total}
\begin{split}
     \left\|\bmu_\ell(\bar\omega) - \bmu_\ell(\omega^*)\right\|_{\sw^*}  
    & \lesssim  \left( \sqrt{\cR_I^{(0)}\cR_I^{(2)}}    +  \cR_I^{(0)} \sqrt{\Dtmax} \right) ~   d(\bar\omega,\omega^*)+ \sqrt{\cR_I^{(1)}\cR_I^{(2)}} d(\bar J,J^*)\\
    &\lesssim  \kappa_w \sqrt{\Dtmax}~  d(\bar \omega,\omega^*)
\end{split}
\end{equation}  
where the last step invokes \cref{lem_I_fN}. This completes the proof of \eqref{rate_contract_mu}.
\end{proof}

\subsection{Key technical lemmas used in the proof of \cref{thm_EM_population_omega}}\label{app_sec_tech_lemmas_EM_popu}

For any $k,\ell \in [L]$ and $x\in \RR^L$, define 
\begin{align}\label{def_W_theta}
      W_{k\ell}(x, \omega) &:= \left(x - \frac{\mu_\ell + \mu_k}{2}\right)^\T J (\be_k - \be_\ell).
\end{align} 
so that for any $j\in [L]$, 
\begin{equation}\label{cond_distr_W}
   \left( W_{k\ell}(X, \omega) \mid Y = j \right)  ~ \stackrel{d}{=} ~ \delta^j_{k\ell}(\omega) + N ^\T \sw^{*1/2}J(\be_k - \be_\ell)
\end{equation}
where $N \sim \cN_d(0,\bI_d)$ and 
\begin{align}\label{def_sigma_theta}
    \delta_{k\ell}^j(\omega)=  \bigl(\mu_j^* - \frac{\mu_\ell + \mu_k}{2}\bigr)^\T J (\be_k - \be_\ell).
\end{align}
For future reference, we note that 
\[
	\|\sw^{*1/2}J^*(\be_k - \be_\ell)\|_2^2 = \Delta_{k\ell} = \delta_{k\ell}^k(\omega^*) =-\delta_{k\ell}^\ell (\omega^*).
\] 
Consider the event 
\begin{equation}\label{def_event_N}
    \cE_{N}= \left\{|N^\T \sw^{*1/2}J(\be_k - \be_\ell)| < \frac{c_2}{2c_1} \sqrt{\Dt_{k\ell} } ~ \|\sw^{*1/2}J(\be_k - \be_\ell)\|_2\right\}
\end{equation}
with $c_2 = (1-4c_0-4c_0^2)/2$ and $c_1 = 1+2c_0$.
We know that 
\begin{equation}\label{bd_tail_prob_N1c}
    \PP(\cE_N^c) \leq 2\exp\left(- \frac{c_2^2}{8c_1^2}\Dt_{k\ell}\right).
\end{equation}
\cref{lem_delta_J_order} further ensures that for any fixed $\omega \in \cB_d(\omega^*)$ and any $k,\ell \in [L]$, on the event $\cE_N$,
\begin{align}
    (W_{k\ell}(X,\omega) \mid Y = \ell) & \leq -c_2 \Delta_{k\ell} + {c_2\over 2}\sqrt{\Delta_{k\ell}}\sqrt{\Dt_{k\ell}} \le -{c_2 \over 2} \Delta_{k\ell}, \label{lb_W_kell_ell} \\\label{lb_W_kell_k}
    (W_{k\ell}(X,\omega) \mid Y = k) & \ge  ~ c_2 \Delta_{k\ell} -{c_2\over 2}\sqrt{\Delta_{k\ell}}\sqrt{\Dt_{k\ell}} \ge  ~ {c_2 \over 2} \Delta_{k\ell}. 
\end{align} 
\begin{lemma}\label{lem_delta_J_order}
    For any $\omega = (\pi, M, J) \in \cB_d(\omega^*,c_0)$ given in \eqref{cond_init_w_J} and any $k,\ell\in [L]$, one has
    \begin{align*} 
            &\| \sw^{*1/2}J(\be_k - \be_\ell)\|_2 \le (1+2c_0) \sqrt{\Delta_{k\ell}}
    \end{align*}
    and 
    \[
        \delta_{k\ell}^k(\omega) \ge  {1\over 2}\left(
           1  - 4c_0 - 4c_0^2
         \right) \Delta_{k\ell},\qquad \delta_{k\ell}^\ell(\omega) \le - {1\over 2}\left(
           1  - 4c_0 - 4c_0^2
         \right) \Delta_{k\ell}. 
    \]
\end{lemma}
\begin{proof}
    The first result follows from
    \begin{align*}
        \| \sw^{*1/2}J(\be_k - \be_\ell)\|_2
        &\le \|\sw^{*1/2} J^*(\be_k - \be_\ell)\|_2 + \|\sw^{*1/2} (J -J^*)(\be_k - \be_{\ell})\|_2\hspace{-2cm}\\
        &\le   (1 + 2c_0)\sqrt{\Delta_{k\ell}}  &&\text{by \eqref{def_Delta} and \eqref{cond_init_w_J}}.
    \end{align*}
    For the other result, notice that 
    \begin{align*}
          \delta_{k\ell}^\ell(\omega) &=(\be_k - \be_\ell)^\T J^\T \left(\mu_\ell^* - \frac{\mu_\ell + \mu_k}{2})\right)\\
          &= (\be_k - \be_\ell)^\T J^{*\T} \left(\mu_\ell^* - \frac{\mu_\ell^* + \mu_k^*}{2}\right) + (\be_k - \be_\ell)^\T (J - J^*)^\T \left(\mu_\ell^* - \frac{\mu_\ell + \mu_k}{2}\right)\\
          &\qquad + {1\over 2}(\be_k - \be_\ell)^\T J^{*\T}  (\mu_\ell^*-\mu_\ell +  \mu_k^*-\mu_k).
    \end{align*}
    The first term equals to
    \[
         - {1\over 2}(\mu_k^*-\mu_\ell^*)^\T \sw^{*-1}(\mu_k^*-\mu_\ell^*) \overset{\eqref{def_Delta} }{=} -{1\over 2}\Dt_{k\ell},
    \]
    and the second and third terms are no greater than
    \begin{align*}
        & {1\over 2} \|\sw^{*1/2} (J - J^*)(\be_k - \be_\ell)\|_2 \left(
       \|\mu_{\ell}-\mu_{\ell}^*-\mu_{k}+\mu_{k}^*\|_{\sw^*} +\sqrt{\Dt_{k\ell}}
        \right)\\
        &\quad  +    {1\over 2} \|\mu_{\ell}-\mu_{\ell}^*-\mu_{k}+\mu_{k}^*\|_{\sw^*}    \sqrt{\Dt_{k\ell}}.
    \end{align*}
    By \eqref{cond_init_w_J}, we conclude 
    $$
         - \left(
            {1\over 2} - 2c_0 - 2c_0^2
         \right)\Delta_{k\ell} \ge   \delta_{k\ell}^\ell (\omega) \ge  -\left(
            {1\over 2} + 2c_0 + 2c_0^2
         \right)\Delta_{k\ell}. 
    $$
    By the symmetric argument, we can also prove the result for $\delta_{k\ell}^k (\omega)$. The proof is complete.
\end{proof}

	The following lemma controls the (conditional) expectation of various quantities related with $\gamma_\ell(X;\omega)$ and the Gaussian noise $N$. Note that due to independence between $N$ and $Y$, these conditional expectations are deterministic.

	\begin{lemma}\label{lem_var_gamma_X}
		For any $\omega \in \cB_d(\omega^*,c_0)$ given by \eqref{cond_init_w_J}, there exists some absolute positive constants $c=c(c_0)$ and $C>0$ such that  for any  $b\ne \ell\in [L]$ and any $a \ge 1$, 
		\begin{align}\label{bd_exp_gam_b}
			\EE[\gamma_{\ell}(X;\omega)^a\mid Y=b] &~ \le~  C \left({\pi_\ell^*  \over \pi_b^*}\exp(-c\Dt_{b\ell})\right)^a+\exp(-c\Dt_{b\ell}),\\\label{bd_exp_gam_N_b}
			\left\|\EE\left[\gamma_{\ell}(X;\omega) N  \mid Y=b\right]\right\|_2&~ \le~  C{\pi_\ell^*+\pi_b^* \over \pi_b^*} \exp(-c\Dt_{b\ell}),\\\label{bd_exp_gam_NN_op_b}
			\left\|\EE\left[\gamma_{\ell}(X;\omega) N  N^\T \mid Y=b\right]\right\|_{\op}&~ \le~  C{\pi_\ell^*+\pi_b^* \over \pi_b^*} \exp(-c\Dt_{b\ell}),\\\label{bd_exp_gam_N_ell2_b}
			\EE\left[\gamma_{\ell}(X;\omega)\|N\|_2^2 \mid Y=b\right] &~ \le~  C{\pi_\ell^*+\pi_b^* \over \pi_b^*} d\exp(-c\Dt_{b\ell}),\\\label{bd_exp_gamma_N_ell4_b}
			\left\|\EE\left[\gamma_{\ell}(X;\omega)\|N\|_2^2 NN^\T \mid Y=b\right]\right\|_\op &~ \le~  C{\pi_\ell^*+\pi_b^* \over \pi_b^*} d\exp(-c\Dt_{b\ell}).
		\end{align}
		In particular, the above holds for $\omega = \omega^*$ with $c_0 = 0$.
	\end{lemma}
	\begin{proof}
		Pick any $b\ne \ell$ and $\omega \in \cB_d(\omega^*)$. Recall $\cE_N$ from \eqref{def_event_N}.  Start with
		 \begin{align*}
			  \EE[\gamma_{\ell}(X;\omega)^a\mid Y=b]  
			&\le  	\EE\left[\gamma_{\ell}(X;\omega)^a 1\{\cE_N\} \mid Y=b\right]    + \PP(\cE_N^c).
		\end{align*}  
		Note that on $\cE_N$ and conditioning on $Y=b$, 
		\begin{align}
			\gamma_{\ell}(X;\omega)^a    \le \left(\frac{\pi_{\ell}}{ \pi_b \exp(W_{b\ell}(X,\omega))} \right)^a \overset{\eqref{lb_W_kell_k}}{\le}  \left(\frac{\pi_{\ell}}{ \pi_b}\,\exp\Bigl(-{c_2\over 2}\Dt_{b\ell}\Bigr)\right)^a.\label{bd_gamma_ell_N1}
		\end{align}
		By using \eqref{eq_order_pis} and \eqref{bd_tail_prob_N1c}, we obtain 
		\begin{align*} 
			 \EE[\gamma_{\ell}(X;\omega)^a\mid Y=b]   
			&    \le  3^a \left(\frac{\pi^*_{\ell}}{ \pi^*_b}\,\exp\Bigl(-{c_2\over 2}\Delta_{b\ell}\Bigr)\right)^a +   \exp\Bigl(-\frac{c_2^2}{8c_1^2}\Dt_{b\ell}\Bigr).
		\end{align*}
		 Setting $c=\min\{c_2/2, c_2^2/(16 c_1^2)\}$  gives  \eqref{bd_exp_gam_b}.
		
		 \cref{bd_exp_gam_N_b} follows by Cauchy-Schwarz inequality and \eqref{bd_exp_gam_b} with $a = 2$ as  
		\begin{align*}
			\left\|\EE\left[\gamma_{\ell}(X;\omega) N  \mid Y=b\right]\right\|_2&=\sup_{v \in \Sp^d} \EE \left[ \gamma_{\ell}(X;\omega^*) v^\T N \mid Y = b \right]\\ 
			&\le \sqrt{\EE \left[\gamma_{\ell}(X;\omega^*)^2\mid Y = b \right]}\sup_{v \in \Sp^d} \sqrt{\EE[(v^\T N)^2]}.
		\end{align*}
	
		 We use similar arguments to prove the other results. For  \eqref{bd_exp_gam_NN_op_b}, we find  that 
		 \begin{align*}
		 		\left\|\EE\left[\gamma_{\ell}(X;\omega) N  N^\T \mid Y=b\right]\right\|_{\op}  &\le   \left\|\EE\left[\gamma_{\ell}(X;\omega) N  N^\T  1\{\cE_N\}   \mid Y=b\right]\right\|_{\op}\\
		 		&\quad   +\left\|\EE\left[  N  N^\T   1\{\cE_N^c \}  \right]\right\|_{\op}.  
	 	\end{align*}
	 	The second term is bounded by
	 	\begin{align}\label{bd_NN_N1_comp}
	 		\sup_{v\in \Sp^d} \EE\left[   v^\T  N N^\T v 1\{\cE_N^c\}  \right] 
	 		&\le  \sup_{v\in \Sp^d} \sqrt{\EE\left[ (N^\T v)^4\right]} \sqrt{\PP(\cE_N^c\}} \le 2 \exp\Bigl(-\frac{c_2^2}{16c_1^2}\Dt_{b\ell}\Bigr). 
	 	\end{align}
	 	Since the first term can be bounded by using \eqref{bd_gamma_ell_N1} with $a=1$:
	 	\begin{align*}
	 	\left\|\EE\left[\gamma_{\ell}(X;\omega) N  N^\T  1\{\cE_N\}   \mid Y=b\right]\right\|_{\op}
	 		&\le {\pi_{\ell} \over \pi_b} \exp\left(-{c_2 \over 2}\Dt_{b\ell}\right) \sup_{v\in\Sp^d}\EE \left[  (N^\T v)^2    \right]\\
	 		&\le  3{\pi_{\ell}^* \over \pi_b^*} \exp\left(-{c_2 \over 2}\Dt_{b\ell}\right),
	 	\end{align*} 
 		we obtain \eqref{bd_exp_gam_NN_op_b}.    \cref{bd_exp_gam_N_ell2_b} follows by noting that
		\begin{align*}
			\EE\left[\gamma_\ell(X;\omega) \|N\|_2^2\mid Y=b\right] &\le \sqrt{\EE[\gamma_{\ell}(X;\omega^*)^2 \mid Y = b]}\sqrt{\EE[\|N\|_2^4]} \le 6{\pi_\ell^*+\pi_b^* \over \pi_b^*}d  \exp(-c\Dt_{b\ell}).
		\end{align*}
		Finally, by \eqref{bd_gamma_ell_N1}, we   have 
		\begin{align*}
			&\left\|\EE\left[\gamma_{\ell}(X;\omega)\|N\|_2^2 NN^\T \mid Y=b\right]\right\|_\op\\&\le \sup_{v \in \Sp^d} {\pi_{\ell}\over \pi_b}\exp(-c\Dt_{b\ell}) \EE\left[ \|N\|_2^2 (v^\T N)^2 \right] + \sup_{v \in \Sp^d} \sqrt{\EE[(N^\T u)^2 \|N\|_2^2]}\sqrt{\PP(\cE_N^c)}\\
			&\lesssim  {\pi_\ell^*+\pi_b^* \over \pi_b^*}d  \exp(-c\Dt_{b\ell}),
		\end{align*}
		proving \eqref{bd_exp_gamma_N_ell4_b}. This completes the proof.
	\end{proof}

    The following lemma establishes upper bounds of $\cR_I^{(j)}$, for $j\in \{0,1,2\}$, given by \eqref{def_cRs}. 
\begin{lemma}\label{lem_I_fN}
    For some constant $c>0$ depending only on $c_1$ and $c_2$ given in \eqref{def_event_N}, and some absolute constant $C>0$, one has
    \begin{align}\label{bd_cRs}
         \max_{j\in \{0,1\}} \cR_I^{(j)}  ~ \leq~  {C\over \pmin } \exp\left(-c\Dtmin\right),\qquad \cR_I^{(2)} ~ \leq~  {CL\over \pmin } \exp\left(-c\Dtmin\right).
    \end{align} 
\end{lemma}
\begin{proof}
	We first prove $j=0$.
	 For any $\ell\in [L]$ and $\omega$, define 
	\[
	I_\ell(X,\omega)  := \pi_{\ell} \sum_{a \in [L]\setminus\{\ell\}} \pi_a  I_{a\ell}(X,\omega) = \gamma_{\ell}(X;\omega)(1-\gamma_{\ell}(X;\omega)).
	\]
	so that $\cR_I^{(0)} = \sup_{\omega\in \cB_d(\omega^*)} \max_{\ell \in [L]}  \pi_{\ell}^{-1} \EE[I_\ell (X,\omega) ]$. 
	Note that $\EE[I_\ell (X,\omega) ] $ equals to 
	\begin{align*}
		 &  \sum_{b=1}^L \pi_b^* \EE\left[I_\ell (X,\omega) \mid Y = b \right]    \\
		&\le  \pi_\ell^* ~ \EE\left[(1-\gamma_{\ell}(X;\omega)) \mid Y = \ell \right] + \sum_{b\in[L]\setminus\{\ell\}} \pi_b^* ~ \EE\left[ \gamma_{\ell}(X;\omega) \mid Y = b \right]  \\
		&= \pi_\ell^* \sum_{a\in [L]\setminus\{\ell\}} \EE\left[\gamma_a(X;\omega)) \mid Y = \ell \right] + \sum_{b\in[L]\setminus\{\ell\}} \pi_b^* ~ \EE\left[ \gamma_{\ell}(X;\omega) \mid Y = b \right] \\
		&\lesssim    \pi_\ell^* \sum_{a\in [L]\setminus\{\ell\}}   {\pi_a^*+\pi_\ell^* \over \pi_\ell^*}\exp(-c\Dtmin) + \sum_{b\in[L]\setminus\{\ell\}} \pi_b^* {\pi_\ell^*+\pi_b^* \over \pi_b^*}\exp(-c\Dtmin) &&\text{by  \eqref{bd_exp_gam_b}}\\
		&\lesssim (1+L\pi_{\ell}^*)\exp(-c\Dtmin).
	\end{align*} 
	This together with \eqref{eq_order_pis} gives
	\[
		\cR_I^{(0)} \lesssim \max_{\ell \in [L]} \left({1\over \pi_{\ell}^*}+ L\right)  \exp\left(-c\Dtmin\right)\lesssim  {1\over \pmin}  \exp\left(-c\Dtmin\right)
	\]
	proving the result for $j=0$. 

To prove the result of $j=1$, we similarly bound $\|\EE[I_\ell (X,\omega)NN^\T ] $ by 
\begin{align*}
     \pi_\ell^* \sum_{a\in [L]\setminus\{\ell\}}    \left\|\EE\left[\gamma_a(X;\omega))NN^\T  \mid Y = \ell \right]\right\|_\op + \sum_{b\in[L]\setminus\{\ell\}} \pi_b^*   \left\| \EE\left[ \gamma_{\ell}(X;\omega)NN^\T  \mid Y = b \right]  \right\|_\op.
\end{align*}
Invoking \eqref{bd_exp_gam_NN_op_b} and \eqref{eq_order_pis} gives the result of $j=1$.  

Finally, for $j =2$, note that 
\begin{align*}
	\cR_I^{(2)}  &\le   \sup_{\omega\in\cB_d(\omega^*)} \max_{\ell \in [L]} \sup_{\max_a \|v_a\|_2=1} \sum_{a \in [L]\setminus\{\ell\}} \pi_a \EE\left[I_{a\ell}  (N^\T v_a)^2\right]\\
	&\le  \sup_{\omega\in\cB_d(\omega^*)} \max_{\ell \in [L]} \sup_{\max_b \|v_b\|_2=1} \sum_{b \in [L]\setminus\{\ell\}}\sum_{a \in [L]\setminus\{\ell\}} \pi_a \EE\left[I_{a\ell}   (N^\T v_b)^2\right]\\
	&\le (L-1) \sup_{\omega\in\cB_d(\omega^*)}  \left\|\sum_{a \in [L]\setminus\{\ell\}} \pi_a \EE\left[I_{a\ell} N N^\T \right]\right\|_\op\\
	&= (L-1)\cR_I^{(1)}.
\end{align*}
The proof is complete. 
\end{proof}

\section{Proofs of the sample level results in \cref{sec_theory_samp_small}}\label{app_sec_proof_sample_small}

 \subsection{Proof of \cref{thm_concent}: concentration inequalities of  the M-steps at the true parameter}\label{app_sec_proof_concentration}
 
 Recall the sample level M-steps $\wh \bpi_\ell(\cdot)$, $\wh \bmu_\ell(\cdot)$ and $\bwhsw(\cdot)$ from \cref{iter_pi_hat,iter_mu_hat,iter_sw_hat,iter_sw_hat_alter}.
 As done in \cref{app_sec_proof_thm_EM_population}, we write $\gamma_\ell(x;\theta) = \gamma_\ell(x;\omega)$ with $\omega = (\pi, M, J)$.

 \subsubsection{Proof of \cref{concent_rate_pi}: concentration of $\wh\pi_{\ell}(\cdot)$}
 \begin{proof}
      For any $\ell \in [L]$, we apply the Bernstein inequality to bound
 \begin{align*}
 	\wh\bpi_\ell(\omega^*) - \bpi_\ell(\omega^*) &= \EE_n[\gamma_\ell(X;\omega^*)] - \EE[\gamma_\ell(X;\omega^*)]  = {1 \over n}\sum_{i=1}^n \left\{\gamma_\ell(X_i;\omega^*) - \EE[\gamma_\ell(X_i;\omega^*)] \right\}.
 \end{align*}
 Since the variance of each $\gamma_\ell(X_i;\omega^*)$ is no greater than 
 \begin{align} \nonumber
 	\EE[\gamma_\ell^2(X_i;\omega^*)]  
 	& \le \pi_\ell^* + \sum_{b\in [L]\setminus\{\ell\}} \pi_b^*~  \EE[\gamma_\ell(X_i;\omega^*)^2 \mid Y=b]  \\\nonumber
 	& \le \pi_\ell^* + (L  \pi_\ell^*+ 1) \exp{(-c\Dtmin)} &&\text{by \cref{lem_var_gamma_X}} \notag\\
 	& \lesssim \pi_\ell^* &&\text{by \eqref{cond_min_Delta}},\label{bd_var_gam_l}
 \end{align} 
 applying the Bernstein's inequality gives that for all $t>0$,  
 \[
 \PP\left\{|\wh\bpi_\ell(\omega^*) - \bpi_\ell(\omega^*)| \lesssim \sqrt{\pi_\ell^* t\over n} + {t\over n} \right\} \ge 1-e^{-t}.
 \]
 Taking the union bounds over $\ell \in [L]$, choosing $t=\log(n)$ and using $n\pmin \ge C\log n$  yield  \eqref{concent_rate_pi}.
 \end{proof}

 \subsubsection{Proof of \cref{concent_rate_M}: concentration of  $\wh \mu_{\ell}(\cdot)$}
 \begin{proof} 
 We work on the event  
 \begin{equation}\label{def_event_pi}
 	\cE_\pi := \bigcap_{\ell \in [L]}\left\{ {| \wh \bpi_{\ell}(\omega^*) -\pi_\ell^*|\over \pi_\ell^*} + {| \wt \pi_{\ell} -\pi_\ell^*|\over \pi_\ell^*} \le C \sqrt{\log n\over n\pi_\ell^*}\right\}.
 \end{equation}
 which, by  \eqref{concent_rate_pi} and   \cref{lem_concen_pis},  holds with probability at least $1-n^{-1}$ under $n\pmin \ge C \log n$.  On the event $\cE_\pi$, we have
 \begin{equation}\label{sandwich_pis}
 	1/2 \le \wh \bpi_{\ell}(\omega^*) / \pi_{\ell}^* \le 3/2,\qquad c  \le  \wt\pi_{\ell}/\pi_\ell^* \le C,\qquad \forall ~ \ell \in [L].
 \end{equation} 
 Pick any $\ell \in [L]$. Notice that
 \begin{align}\label{decomp_mu_hat_diff}
 	\wh\bmu_\ell(\omega^*) - \mu_\ell^* = {\EE_n[\gamma_\ell(X;\omega^*)(X-\mu_{\ell}^*)] \over \EE_n [\gamma_\ell(X;\omega^*)]} &= {1\over \wh \bpi_{\ell}(\omega^*)}{1\over n} \sum_{i =1}^n \underbrace{\gamma_\ell(X_i;\omega^*)(X_i-\mu_{\ell}^*)}_{V_{\ell,i}}.
 \end{align} 
 By 
 $\EE[V_{\ell,i}] =  \EE[\gamma_\ell(X_i;\omega^*)](\bmu_{\ell}(\omega^*) - \mu_{\ell}^*) =
 	0 $, we have  
$$
 	\|\wh\bmu_\ell(\omega^*) - \bmu_\ell(\omega^*)\|_{\sw^*} 
 	   =   {1\over n \wh\bpi_{\ell}(\omega^*)} \left\|   \sum_{i=1}^n \left(V_{\ell,i}-\EE[V_{\ell,i}]\right)  \right\|_{\sw^*}.
 $$
 to which we use a truncation argument together with the vector-valued Bernstein inequality to analyze. From \citet[Lemma 5.3]{segol2021improved} we know $\sw^{*-1/2}V_{\ell,i}$ is a sub-Gaussian random vector with some absolute sub-Gaussian constant $C_v>0$. Consider the event
 \begin{align}\label{def_event_N_2}
 	\cE_2 = \bigcap_{i=1}^n \cE_{2,i},\qquad \cE_{2,i}:= \left\{  \|V_{\ell,i}\|_{\sw^*} \leq   C_v(\sqrt{d} + 2\sqrt{\log n})  \right\}
 \end{align}
 and  by \cref{subG_quad_bd}, one has $\PP(\cE_{2,i}) \ge 1-n^{-2}$ and $\PP(\cE_2)\ge 1-n^{-1}$. We first find that 
 \begin{align*}
 	\EE[ \|V_{\ell,i}\|_{\sw^*}^2] &\le    \EE\left[
 	 \gamma_\ell(X_i;\omega^*)  \|X_i-\mu_{\ell}^*\|_{\sw^*}^2 
 	\right] \\
 	&\le  \pi_\ell^* \EE\left[
 	  \|N_i\|_2^2  
 	\right] +2\sum_{a\in [L]\setminus\{\ell\}}  \pi_a^* \EE\left[
 	\gamma_\ell(X_i;\omega^*)(\|N_i\|_2^2 + \Dt_{a\ell}) \mid Y_i = a
 	\right] \\
 	 &\lesssim d \pi_\ell^*  + \sum_{a \in [L]\setminus\{\ell\}}  (\pi_{\ell}^*+\pi_a^*) (d+\Delta_{a\ell})   \exp(-c\Dt_{a\ell}) &&\text{by  \cref{lem_var_gamma_X}}\\
 	 &\lesssim d  \pi_{\ell}^* +   (L\pi_{\ell}^*+1) d  \exp{(-c'\Dtmin)} &&\text{by }\cE_\pi\\
 	 &\lesssim d \pi_{\ell}^* &&\text{by \eqref{cond_min_Delta}}.
 \end{align*}
 Then  
 \begin{align}\label{prob_comp_eX} 
  	{1 \over n} \left\|   \sum_{i=1}^n\EE[V_{\ell,i}  1\{\cE_{2,i}^c\}] \right\|_{\sw^*} \le ~ \max_{i \in [n]} \EE\left[ \|V_{\ell,i}\|_{\sw^*}^2\right]^{1/2} \PP(\cE_{2,i}^c)^{1/2} \lesssim \sqrt{d \pi_\ell^* \over n^2}.
\end{align}
 On the event $\cE_2$, 
 it remains to  bound from above 
 \begin{align}\label{targe_V_ell}
 	{1 \over n}	\left\| \sum_{i=1}^n  \Bigl\{V_{\ell,i} 1\{\cE_{2,i}\} - \EE\left[V_{\ell,i} 1\{\cE_{2,i}\}   \right] \Bigr\}\right\|_{\sw^*}.
 \end{align}
Clearly,  the following holds almost surely,
\[
\max_{i \in [n]} \|V_{\ell,i} 1\{\cE_{2,i}\} \|_{\sw^*}\le  C_v(\sqrt{d}+ 2\sqrt{\log n}).
\]
In conjunction with the above variance bound, applying the vector-valued Bernstein's inequality in \cref{lem_bernstein_vector} gives that for all $t\ge 0$,
 \[
  \eqref{targe_V_ell}   \lesssim  \sqrt{t  \over n}\sqrt{d  \pi_{\ell}^*}+   {t \over n}  \sqrt{d+\log n} 
 \] 
 with probability at least $1- e^{-t}-2n^{-1}$.  Together with    \eqref{prob_comp_eX},   by choosing $t = \log n$ and using $n\pmin  \ge C \log(n)(1 \vee  \log(n)/d)$, we obtain that for any $t\ge 1$, with probability $1-3n^{-1}$, 
\begin{align}
	\|\wh\bmu_\ell(\omega^*) - \bmu_\ell(\omega^*)\|_{\sw^*} &\lesssim   \sqrt{  d\log n \over n\pi_{\ell}^*}  + { \log n  \over n\pi_{\ell}^*}\sqrt{d+\log n}  \lesssim  \sqrt{  d\log n \over n\pi_{\ell}^*} . \label{bd_mu_hat_diff}
\end{align}
This completes the proof of \eqref{concent_rate_M}.
 \end{proof}

 \subsubsection{Proof of \cref{concent_rate_sw}: concentration of $\whsw(\cdot)$}
 \begin{proof}
 	We analyze the term
 	\[
 	d(\bwhsw(\omega^*), \bsw(\omega^*)) =  \left\|
 	\sw^{*-1/2}(\bwhsw(\omega^*) - \bsw(\omega^*))\sw^{*-1/2}
 	\right\|_{\op}.
 	\]
 	It is easy to verify the decomposition
 	\begin{align}\label{decomp_sw_hat}\nonumber
 		\whsw (\omega^*) &= \EE_n \left[\sum_{\ell = 1}^L \gamma_\ell(X, \omega^*)(X - \wh\bmu_\ell(\omega^*))(X - \wh\bmu_\ell(\omega^*))^\T \right] \\
 		& = {1 \over n}  \sum_{i=1}^n   \sum_{\ell = 1}^L\gamma_\ell(X_i;\omega^*)(X_i - \mu_\ell^*)(X_i - \mu_\ell^*)^\T -   \sum_{\ell = 1}^L \wh \bpi_\ell(\omega^*)(\mu_\ell^* - \wh \bmu_\ell(\omega^*))(\mu_\ell^* - \wh \bmu_\ell(\omega^*))^\T.
 	\end{align}
 	For any  $i\in [n]$, define 
 	\[
 		H_i:= \sum_{\ell = 1}^L  \sw^{*-1/2}\gamma_\ell(X_i;\omega^*)(X_i - \mu_\ell^*)(X_i - \mu_\ell^*)^\T \sw^{*-1/2}
 	\]
 	and note that   $\EE[H_i ] =\bI_d.$
 	Adding and subtracting terms gives  
 	$d(\bwhsw(\omega^*), \bsw(\omega^*))  \le   \rI + \rII $
 	where 
 	\begin{align}\nonumber
 		\rI := ~ &{1 \over n} \left\|    \sum_{i=1}^n \left(  H_i - \EE[H_i]\right)\right\|_\op\\ \label{def_rIII_cov}
 		\rII := ~ & \left\|   \sum_{\ell = 1}^L \wh \bpi_\ell(\omega^*)\sw^{*-1/2}(\mu_\ell^* - \wh \bmu_\ell(\omega^*))(\mu_\ell^* - \wh \bmu_\ell(\omega^*))^\T \sw^{*-1/2} \right\|_\op .
 	\end{align} 
    By \eqref{bd_mu_hat_diff} and $\cE_\pi$ in \eqref{def_event_pi},  with probability at least $1-3n^{-1}$,
 	\begin{align}\label{bd_rII_cov_dev}
 		\rII  &\lesssim   \sum_{\ell = 1}^L  \pi_{\ell}^* \left\|\mu_\ell^* - \wh \bmu_\ell(\omega^*)\right\|_{\sw^*}^2 \lesssim   {  dL\log n \over n }. 
 	\end{align}

    To bound $\rI$, note that $[\gamma_\ell(X_i;\omega^*)]^{1/2}(X_i-\mu_\ell^*)$ is also sub-Gaussian with some constant $C>0$. 
 	In the following we work on the event   
    \[
        \cE_2 = \bigcap_{i=1}^n \cE_{2,i},\qquad \cE_{2,i}:= \left\{  \|[\gamma_\ell(X_i;\omega^*)]^{1/2}(X_i-\mu_\ell^*)\|_{\sw^*} \leq   C_v(\sqrt{d} + 2\sqrt{\log n})  \right\}
    \]
    We have $\rI \le \rI_1 + \rI_2$ with  
 	\begin{align*}
 			\rI_1 &:= {1 \over n} \left\|    \sum_{i=1}^n \left(  H_i 1\{\cE_{2,i}\}- \EE[H_i 1\{\cE_{2,i}\} ]\right)\right\|_\op,\\
 				\rI_2 &:=  {1 \over n} \left\|    \sum_{i=1}^n   \EE\left[H_i 1\{\cE_{2,i}^c\} \right]\right\|_\op.
 	\end{align*} 
 	
     We first bound $\rI_1$ by the matrix-valued Bernstein inequality of \cref{lem_bernstein_mat}. For any $i\in [n]$, we have that almost surely,
 	\begin{align*}
 		 \|H_i 1\{\cE_{2,i}\}\|_\op &\le \sum_{\ell=1}^L \gamma_{\ell}(X_i;\omega^*)\|X_i-\mu_{\ell}^*\|_{\sw^*}^21\{\cE_{2,i}\} \lesssim L(d + \log n).
 	\end{align*}
 	To bound the variance term in the Bernstein inequality, we decompose 
 	\begin{align*}
 		&\left\|   \EE[H_i^2 1\{\cE_{2,i}\}]
 		\right\|_\op \\
 		&\le \left\| \sum_{a=1}^L \pi_a^* \EE\left[ \gamma_a^2(X_i;\omega^*) \sw^{*-1/2}(X_i-\mu_a^*)(X_i-\mu_a^*)^\T \sw^{*-1/2} \|X_i-\mu_a^*\|_{\sw^*}^2  \mid Y_i = a\right]
 		\right\|_\op \\
 		& \quad + \left\| \sum_{a=1}^L \pi_a^* \EE\left[ \left(\sum_{\ell\in[L]\setminus\{a\}}\gamma_\ell(X_i;\omega^*) \sw^{*-1/2}(X_i-\mu_{\ell}^*)(X_i-\mu_{\ell}^*)^\T \sw^{*-1/2}\right)^2  \mid Y_i = a\right]
 		\right\|_\op.
 	\end{align*}
 	The first term is bounded by 
 	\begin{align}\label{bd_var_1}
 		&\left\| \sum_{a=1}^L \pi_a^* \EE\left[ \gamma_a(X_i;\omega^*)  N_iN_i^\T  \|N_i\|_2^2  \mid Y_i = a \right]\right\|_\op \le \sup_{v\in \Sp^d}   \EE[ (N_i^\T v)^2 \|N_i\|_2^2] \le 3 d. 
 	\end{align}
 	while the second term is no greater than 
 	\begin{align}\label{bd_var_2}\nonumber
 		&\sup_{v\in \Sp^d} \sum_{a=1}^L \pi_a^* \EE\left[ \sum_{\ell\in[L]\setminus\{a\}}\gamma_\ell(X_i;\omega^*) \langle v, X_i-\mu_{\ell}^*\rangle_{\sw^*}^2 
 		\sum_{\ell\in[L]\setminus\{a\}}\gamma_\ell(X_i;\omega^*) \left\| X_i-\mu_{\ell}^*\right\|_{\sw^*}^2  \mid Y_i = a\right]\\\nonumber
 		&\le 4\sup_{v\in \Sp^d} \sum_{a=1}^L \pi_a^*	\sum_{\ell\in[L]\setminus\{a\}} \EE\left[
 	\gamma_\ell(X_i;\omega^*)  \left((v^\T N_i)^2 + \Dt_{a\ell}\right) \left(\|N_i\|_2^2 + \Dt_{a\ell} \right) \mid Y_i = a
 		\right]\\\nonumber
 		&\lesssim  \sum_{a = 1}^L   \sum_{\ell\in[L]\setminus\{a\}} (\pi_{\ell}^*+\pi_a^*)
 		(d+\Dt_{a\ell} )\Dt_{a\ell}  \exp(-c\Dt_{a\ell})\\
 		&\lesssim   dL \exp(-c\Dt_{\min})
 	\end{align}
 	where the penultimate step uses \cref{lem_var_gamma_X}. Combining \eqref{bd_var_1} and \eqref{bd_var_2}  together with \eqref{cond_min_Delta} yields that 
 	\begin{align*}
 		\left\| \sum_{i=1}^n  \EE[H_i^2 1\{\cE_{2,i}\}]
 		\right\|_\op \lesssim nd +  n d L  \exp(-c\Dtmin) \lesssim nd.
 	\end{align*}
 	An application of the matrix-valued Bernstein inequality in \cref{lem_bernstein_mat} gives: for all $t\ge 0$, 
 	\begin{align} \label{bd_rI_cov_dev}
 	  	\rI_1  \lesssim \sqrt{dt \over n} + {t L(d+\log n )\over n} 
 	\end{align}
 with probability at least $1-de^{-t}-2n^{-1}$. Regarding $\rI_2$, by using the arguments in \eqref{bd_var_1} and \eqref{bd_var_2}, one can deduce that
 	\begin{equation*}
 		\rI_2 \le \max_{i \in [n]} \left\| \EE\left[H_i 1\{\cE_{2,i}^c\}
 		\right]\right\|_{\op}  \le \max_{i \in [n]} \sqrt{\|\EE[ H_i \|H_i\|_\op]\|_\op}\sqrt{\PP(\cE_{2,i}^c)} \lesssim {\sqrt{d}\over n}.
 	\end{equation*}
 	Thus the bound in  \eqref{bd_rI_cov_dev}  holds for $\rI$ for all $t\ge 1$. In conjunction with \eqref{bd_rII_cov_dev}, by choosing $t = 2\log n$ and using $d\le n$ and $\log n \le n\pmin \le n/L$, we conclude that with probability at least $1-4/n$, 
 	\begin{align*}
 		d(\bwhsw(\omega^*), \bsw(\omega^*)) & \lesssim \sqrt{d\log n \over n}  +     {  (d+\log n)L\log  n \over n }
 	\end{align*}
    completing the proof of \eqref{concent_rate_sw}.
    \end{proof}

\subsection{Contraction of the sample level M-steps}\label{app_sec_proof_contraction_samp}

For any $\theta$ satisfying \eqref{cond_init_unknown}, by recalling that $\omega = \omega(\theta)= (\pi, M,J(\theta))$, we establish in this subsection the sample level contraction.
We introduce a few additional quantities:
\begin{equation}\label{def_cRs_samp}
	\begin{split}
		\wh\cR_I^{(0)}   & := \sup_{\omega\in\cB_d(\omega^*)} \max_{\ell \in [L]} \sum_{a \in [L]\setminus\{\ell\}} \pi_a   \EE_n\left[I_{a\ell}(X,\omega)\right]\\
		\wh\cR_I^{(1)} & := \sup_{\omega\in\cB_d(\omega^*)} \max_{\ell \in [L]} ~ \Bigl\|\sum_{a \in [L]\setminus\{\ell\}} \pi_a\EE_n\left[I_{a\ell}(X,\omega) NN^\T\right]\Bigr\|_\op\\
		\wh\cR_I^{(2)} & :=  \sup_{\omega\in\cB_d(\omega^*)}  \max_{\ell \in [L]} \sum_{a \in [L]\setminus\{\ell\}} \pi_a  \left\|\EE_n\left[I_{a\ell}(X,\omega) NN^\T\right]\right\|_{\op}\\
	\end{split}
\end{equation}
which are the sample level counterparts of those in \eqref{def_cRs}.  In \cref{lem_I_fN_samp} we show that with probability at least $1-n^{-1}$, 
\[
	\wh\cR_I^{(0)}+\wh\cR_I^{(1)}  =\cO\left( \xi_n \right),\qquad 	\wh\cR_I^{(2)} = \cO\left( L\xi_n\right).
\]
with 
\begin{equation}\label{def_xi_n}
	\xi_n = {1\over \pminsq}\exp(-c\Dtmin) + {dL\log(q)\log(n)\over n\pmin}.
\end{equation}
and $q := n(\Dtmax + 1/\pmin)$.
	
	The following theorem gives the one-step contraction of the sample level M-steps. Recall $d(\omega,\omega^*)$ from \eqref{def_dist_omega}.

   \begin{theorem}\label{thm_EM_samp_omega}
	Grant condition \eqref{cond_min_Delta}.  On the event 
	\begin{equation}\label{def_event_cR_hat}
	 \Dtmax(\Dtmax  +L) \xi_n^2+ d(\wh\bpi(\omega^*), \bpi(\omega^*))\le c,
	\end{equation}
	for any $\omega \in \cB_d(\omega^*)$, 
	we have 
	\begin{align}\label{rate_contract_pi_hat}
		d(\wh\bpi(\omega), \wh\bpi(\omega^*)) & ~ \le ~ \xi_n (
		1+\sqrt{L/ \Dtmax}
		) ~ 
		d(\omega, \omega^*),\\\label{rate_contract_mu_hat}
		d(\wh \bM(\omega), \wh \bM(\omega^*))  & ~ \le ~ \xi_n   \sqrt{L + \Dtmax} ~ d(\omega,\omega^*).
	\end{align}
Furthermore, if the event 
	\begin{equation}\label{def_event_Mhat} 
				 d(\wh \bM(\omega^*),  \bM(\omega^*)) \le  \sqrt{\Dtmax} 
	\end{equation}
holds, then 
\begin{align}\label{rate_contract_Sigma_hat}
		d(\bwhsw(\omega),\bwhsw(\omega^*)) &~\le ~  \xi_n   \sqrt{(L + \Dtmax)\Dtmax}    ~ 
		d(\omega, \omega^*).
	\end{align}  
\end{theorem}

\begin{proof}
We work on the events in \eqref{def_event_cR_hat} and \eqref{def_event_Mhat}.
\paragraph{Contraction of $\wh\pi_\ell$.}
Fix any $\ell \in [L]$ and $\bar\omega$. We have 
\begin{align*}
     \wh\bpi_\ell(\bar\omega) - \wh\bpi_\ell(\omega)  &=    \EE_n[\gamma_\ell(X,\bar\omega) - \gamma_\ell(X,\omega^*) ] = \nabla^\pi \wh\bpi_\ell(\bar\omega) + \nabla^M \wh\bpi_\ell(\bar\omega) + \nabla^J \wh\bpi_\ell(\bar\omega)
\end{align*} 
where, similar as \eqref{def_nabla_pi},
\begin{align*}
     \nabla^\pi \wh\bpi_\ell(\bar\omega) &:= \sum_{k=1}^{L-1}\int_0^1 \EE_n \left\langle \frac{\partial \gamma_\ell(X;\omega)}{\partial \pi_k}\Big|_{\omega=(\pi_u,\bar M,\bar J)},\bar \pi_k - \pi_k^*\right\rangle \rd u \\ 
        \nabla^M \wh\bpi_\ell(\bar\omega) &:= 
          \sum_{k=1}^L \int_0^1 \EE_n  \left\langle \frac{\partial\gamma_{\ell}(X;\omega)}{\partial \mu_k}\Big|_{\omega=(\pi^*,M_u,\bar J)},\bar \mu_k - \mu_k^*\right\rangle \rd u\\
          \nabla^J \wh\bpi_\ell(\bar\omega) &:=  
        \int_0^1 \EE_n  \left\langle \frac{\partial\gamma_{\ell}(X;\omega)}{\partial J}\Big|_{\omega=(\pi^*,M^*,J_u)}, \bar J - J^*\right\rangle \rd u.
\end{align*}
By repeating the same argument of proving \cref{bd_part_deriv_gamma_pi,bd_part_deriv_gamma_mu,bd_part_deriv_gamma_J}, we have
\begin{align*}
    | \nabla^\pi \wh\bpi_\ell(\bar\omega)| &~ \lesssim ~\pi_\ell^* ~  \wh \cR_I^{(0)} ~ d(\bar\pi, \pi^*) ,\\
    |\nabla^M \wh\bpi_\ell(\bar\omega)|&~\lesssim~  \pi_\ell^*  ~   \wh  \cR_I^{(0)} \sqrt{\Dtmax} ~ d(\bar M, M^*),\\
    |\nabla^J \wh\bpi_\ell(\bar\omega)|&~\lesssim ~ \pi_\ell^*    \left( \wh \cR_I^{(0)}\sqrt{\Dtmax} +  \sqrt{\wh \cR_I^{(0)} \wh \cR_I^{(2)}}\right) d(\bar J,J^*)
\end{align*}
so that by recalling \eqref{def_dist_omega}, 
\begin{align*}
    &{1\over \pi_\ell^*}|\wh\bpi_\ell(\bar\omega) - \wh\bpi_\ell(\omega^*)| ~ \lesssim  ~   \sqrt{\wh \cR_I^{(0)} \wh \cR_I^{(2)}}  d(\bar J,J^*)  + \wh \cR_I^{(0)}  
 d(\bar\omega, \omega^*).
\end{align*} 
 
 \paragraph{Contraction of $\wh \mu_\ell$.} Similar to \eqref{def_nabla_mu}, we have 
\[
    \wh\bmu_\ell(\bar\omega) - \wh\bmu_\ell(\omega^*) = \nabla^{\pi}\wh \bmu_\ell(\bar\omega)  + \nabla^{M}\wh \bmu_\ell(\bar\omega) +\nabla^{J}\wh \bmu_\ell(\bar\omega) 
\]
where 
\begin{align*}
     \nabla^{\pi}\wh \bmu_\ell(\bar\omega)  &:= \sum_{k=1}^{L-1} \int_0^1   
\frac{1}{\wh\bpi_\ell(\omega)}\EE_n\left[ \left(X-\wh\bmu_\ell(\omega^*)\right) \left\langle { \partial \gamma_{\ell}(X;\omega)\over \partial \pi_k }\Big|_{\omega = (\pi_u, \bar M, \bar J)}, ~ \left(\bar \pi_k - \pi_k^*\right) \right\rangle \right]  \rd u\\
 \nabla^{M}\wh \bmu_\ell(\bar\omega)&:= \sum_{k=1}^{L} \int_0^1  \frac{1}{\wh\bpi_\ell(\omega)}\EE_n\left[ \left(X-\wh\bmu_\ell(\omega^*)\right) \left\langle {\partial \gamma_{\ell}(X;\omega) \over \partial \mu_k} \Big|_{\omega = (\pi^*, M_u, \bar J)}, ~ \bar \mu_k - \mu_k^* \right\rangle \right] \rd u \\
 \nabla^{J}\wh \bmu_\ell(\bar\omega) & :=    \int_0^1   \frac{1}{\wh\bpi_\ell(\omega)}\EE_n\left[ \left(X-\wh\bmu_\ell(\omega^*)\right) \left\langle{\partial \gamma_{\ell}(X;\omega)\over \partial J}\Big|_{\omega = (\pi^*, M^*, J_u)}, ~ \bar J-J^* \right\rangle  \right]  \rd u. 
\end{align*}
Note that on the event \eqref{def_event_cR_hat},  for any $\omega \in \cB_d(\omega^*)$,
\begin{align*}
    d(\wh\bpi(\omega), \bpi(\omega^*)) &\le d(\wh\bpi(\omega), \wh \bpi(\omega^*)) +  d(\wh\bpi(\omega^*), \bpi(\omega^*))\\
    &\le  \sqrt{\wh \cR_I^{(0)} \wh \cR_I^{(2)}}   d(\bar J,J^*)  + \wh \cR_I^{(0)}  
 d(\omega, \omega^*) + d(\wh\bpi(\omega^*), \bpi(\omega^*))\\
 &\le 1/2 &&\text{by \eqref{def_dist_omega} \& \eqref{bd_dist_omega_ball}}
\end{align*} 
which further ensures 
\[
    1/2 \le \wh \bpi_\ell(\omega) / \pi_\ell^* \le 3/2,\qquad \forall ~ \ell\in[L].
\]
Meanwhile the event \eqref{def_event_Mhat} implies 
\begin{equation}\label{def_D_mu_hat}
   \|\wh \bmu_\ell(\omega^*) - \mu_b^*\|_{\sw^*} \le \sqrt{\Dtmax} +  d(\wh \bM(\omega^*),  \bM(\omega^*))  \le 2\sqrt{\Dtmax}.
\end{equation}
By similar arguments of proving \cref{bd_part_deriv_M_pi,bd_part_deriv_M_mu,bd_part_deriv_M_J}, we find that
\begin{align*}
 \left\|\nabla^{\pi}\wh \bmu_\ell(\bar\omega)\right\|_{\sw^*} &\lesssim ~\left[\sqrt{\wh \cR_I^{(0)} \wh \cR_I^{(1)}} + \wh\cR_I^{(0)}  \sqrt{\Dtmax}  \right]  d(\bar\pi, \pi^*)\\
   \left\|\nabla^{M}\wh \bmu_\ell(\bar\omega)\right\|_{\sw^*}     
        & \lesssim  ~  \left[\sqrt{\wh \cR_I^{(0)} \wh \cR_I^{(1)}}  +\wh\cR_I^{(0)}  \Dtmax   \right] d(\bar M,M^*)\\
    \left\|\nabla^{J}\wh \bmu_\ell(\bar\omega)\right\|_{\sw^*}
   &\lesssim ~ \left[\sqrt{\wh\cR_I^{(1)}\wh\cR_I^{(2)}} +   \left( \sqrt{\wh \cR_I^{(0)} \wh \cR_I^{(2)}} +\wh\cR_I^{(0)}\sqrt{\Dtmax} \right) \sqrt{\Dtmax} \right]d(\bar J, J^*).
\end{align*}
Therefore, similar as \eqref{bd_part_derive_M_total},
\begin{align*}
     \left\| \wh\bmu_\ell(\bar\omega) - \wh\bmu_\ell(\omega^*) \right\|_{\sw^*} ~ &\le ~  \sqrt{\wh\cR_I^{(1)}\wh\cR_I^{(2)}}  d(\bar J,J^*)   +  
     \left(\sqrt{\wh \cR_I^{(0)} \wh \cR_I^{(2)}}  +\wh\cR^{(0)}_I \sqrt{\Dtmax}\right) d(\bar\omega,\omega^*),
\end{align*} 
as desired.

\paragraph{Contraction of $\bwhsw$.} 
Since 
\[
		\bwhsw(\bar\omega)  = \whsT - \sum_{\ell = 1}^L  
		\wh \bpi_{\ell}(\bar\omega) \wh \bmu_{\ell}(\bar\omega)\wh \bmu_{\ell}(\bar\omega)^\T
\]
and using 
  \eqref{rate_contract_mu_hat} gives
\[
d(\wh \bM(\bar\omega),\wh \bM(\omega^*))\le  \xi_n\sqrt{L+\Dtmax}  ~ d(\bar\omega,\omega^*) \overset{\eqref{bd_dist_omega_ball}}{\le}  3 \Dtmax\xi_n\sqrt{L+\Dtmax}   \overset{\eqref{def_event_cR_hat}}{\le} \sqrt{\Dtmax},
\]
invoking part (a) of \cref{lem_lip_sw} 
with $\theta = (\wh \bpi(\bar \omega), \wh \bM(\bar \omega), \bwhsw(\bar \omega))$ and  $\theta' = (\wh \bpi(\omega^*), \wh \bM(\omega^*), \bwhsw(\omega^*))$ gives  
	\begin{align*}
			&  d(\bwhsw(\bar\omega),\bwhsw(\omega^*)) ~ \le  3 \max_{\ell \in[L]}  \|\wh \bmu_{\ell}(\omega^*)\|_{\sw^*} d(\wh \bM(\bar \omega),\wh \bM(\omega^*)) +   \max_{\ell \in[L]} 2\|\wh \bmu_{\ell}(\omega^*)\|_{\sw^*}^2 d(\wh \bpi(\bar\omega),\wh \bpi(\omega^*)).
	\end{align*}
    Since the event  \eqref{def_event_Mhat} ensures
    \[
    	 \|\wh \bmu_{\ell}(\omega^*) \|_{\sw^*} \le \|\bmu_{\ell}(\omega^*) \|_{\sw^*}  +  d(\wh \bM(\omega^*),  \bM(\omega^*))  \le 2\sqrt{\Dtmax}
    \] 
   invoking the contraction results of $\wh \pi$ and $\wh M$ completes the proof.
\end{proof}

\subsubsection{Deviation inequalities of $\wh\cR_I^{(j)}$ for $j\in \{0,1,2\}$}

The following lemma states  stochastic bounds for  $\wh\cR_I^{(j)}$, $j\in \{0,1,2\}$, given by \eqref{def_cRs_samp}. 
\begin{lemma}\label{lem_I_fN_samp}
	With probability at least $1-n^{-dL}$, 
	\begin{align}\label{bd_cR_hat_0}
		\wh \cR_I^{(0)}  ~ &\lesssim  ~  {L^2\over \pmin}\exp(-c\Dtmin) + {dL\log(q)\over n\pmin},\\\label{bd_cR_hat_1}
		\wh \cR_I^{(1)}  ~ &\lesssim  ~  {L\over \pminsq}\exp(-c\Dtmin) + {dL\log(q)\log(n)\over n\pmin},\\\label{bd_cR_hat_2}
		\wh \cR_I^{(2)}  ~ &\lesssim  ~  {L^2\over \pminsq}\exp(-c\Dtmin) + {dL^2\log(q)\log(n)\over n\pmin} 
	\end{align} 
\end{lemma}
We prove \cref{lem_I_fN_samp} in the following a few subsections.
Since it requires to establish  uniform convergence  over $\omega \in \cB_d(\omega^*)$,
our proof  is based on a discretization argument of this ball. We start by constructing an $\epsilon$-net of $\cB_d(\omega^*)$ in the following subsection.

\subsubsection{An $\epsilon$-net construction of $\cB_d(\omega^*)$}
Let $\cN_{\epsilon_1}$ be an $\epsilon_1$-net of $\{v\in \RR^d: \|v\|_2 \le c_0\sqrt{\Dtmin}\}$. 
Define 
\[
    \cN_{\epsilon_1}(\mu_1^*) = \left\{
        \mu_1^* + \sw^{*1/2}v: v\in \cN_{\epsilon_1}
    \right\}.
\]
Then for any $\mu$ such that $\|\mu-\mu_1^*\|_{\sw^*} \le c_0\sqrt{\Dtmin}$, there exists 
$v_{\mu}\in \cN_{\epsilon_1}$ such that 
$\| \sw^{*-1/2}(\mu-\mu_1^*) - v_\mu \|_2\le \epsilon_1$ which further implies that 
$\mu' := \mu_1^* + \sw^{*1/2}v_{\mu} \in \cN_{\epsilon_1}(\mu_1^*) $ satisfies
$
    \|\mu-\mu'\|_{\sw^*} = \| \sw^{*-1/2}(\mu-\mu_1^*) - v_\mu \|_2 \le \epsilon_1. 
$
Therefore, $\cN_{\epsilon_1}(\mu_1^*)$ is an $\epsilon_1$-net of $\{\mu: \|\mu-\mu_1^*\|_{\sw^*} \le c_0\sqrt{\Dtmin}\}$ with respect to   $\|\cdot\|_{\sw^*}$.  Moreover
\[
    \cN_{\epsilon_1}(M^*) = \cN_{\epsilon_1}(\mu_1^*)\otimes \cdots \otimes \cN_{\epsilon_1}(\mu_L^*)
\]
is an $\epsilon_1$-net of $\{M\in \RR^{d\times L}: d(M,M^*) \le c_0\sqrt{\Dtmin}\}$ with respect to the distance $d(M,M')$ in \eqref{def_dist_pi_M}. By construction, we have  $d(M,M^*) \le c_0\sqrt{\Dtmin}$ for any $M \in \cN_{\epsilon_1}(M^*)$.  

By analogous argument, we can construct an $\epsilon_2$-net of $\{J\in \RR^{d\times L}: d(J,J^*) \le c_0\sqrt{\Dtmin}\}$ with respect to $d(J,J')$ in \eqref{def_dist_J}, 
\[
    \cN_{\epsilon_2}(J^*) = \cN_{\epsilon_2}(J_1^*)\otimes \cdots \otimes \cN_{\epsilon_2}(J_L^*)
\]
with 
\[
    \cN_{\epsilon_2}(J_a^*) = \left\{J_a^* + \sw^{*-1/2} v: v\in \cN_{\epsilon_2}\right\},\quad \forall~ a\in [L].
\]
Then for any $J_a$ such that $\|\sw^{*1/2}(J_a - J^*_a) \|_2 \leq c_0 \sqrt{\Dtmin}$, there exists $v \in \cN_{\epsilon_2}$ such that $\| \sw^{*1/2}(J_a - J_a^*) - v\|_2 \leq \epsilon_2$ which further implies that $J_a':=J_a^* + \sw^{*-1/2}v \in \cN_{\epsilon_2}(J_a^*)$ satisfies
$
\|\sw^{*1/2}(J_a - J_a')\|_2 = \|\sw^{*1/2}(J_a - J_a^*) - v\|_2 \leq \epsilon_2.
$

Finally, regarding $\pi$, there exists an $\epsilon_3$-net $\cN_{\epsilon_3}(\cS^L)$ of the simplex in $\RR^L$ with respect to the sup-norm 
\citep{ghosal2001entropies}[Lemma A.4]. As a result, for any $\pi$ such that $d(\pi, \pi^*) \le c_0$, there exists $\pi'\in \cN_{\epsilon_3}(\cS^L)$ such that 
\[
    d(\pi,\pi') \le {\epsilon_3\over \pmin},\qquad d_{\pi}(\pi',\pi^*) \le c_0 +{\epsilon_3\over \pmin} \le 2c_0 
\]
provided that $\epsilon_3 \le c_0 \pmin$.

Combining all three nets constructed above, we have  
\begin{equation}\label{card_nets}
   |\cNeps| := \left|\cN_{\epsilon_1}(M^*) \otimes  \cN_{\epsilon_2}(J^*) \otimes \cN_{\epsilon_3}(\cS^L)\right| \le \left(5 \over \epsilon_3\right)^{L-1} \left(3c_0\sqrt{\Dtmin} \over \epsilon_1\epsilon_2\right)^{dL}.
\end{equation}

\subsubsection{Proof of \cref{lem_I_fN_samp} for $j=0$}

\begin{proof}
	The entire proof is based on the event   $\cE_2$ given in \eqref{def_event_N_2}.
    By the definitions in \eqref{def_cRs_samp} and \eqref{def_I_aell}, we have  
    \begin{align*}
        \wh\cR_I^{(0)} &= \sup_{\omega\in\cB_d(\omega^*)} \max_{\ell \in [L]} {1\over \pi_{\ell}}  \EE_n\left[I_{\ell}(X,\omega)\right] 
    \end{align*} 
	where we write 
	\begin{equation}\label{def_I_ell}
		 I_{\ell}(X,\omega) =   \pi_{\ell}\sum_{a \in [L]\setminus\{\ell\}} \pi_a  I_{a\ell}(X,\omega).
	\end{equation}
    We use a discretization argument to handle the supremum. Recall $\cNeps$ from \eqref{card_nets}. For any $\omega \in \cB_d(\omega^*)$, there exists some $\omega'\in \cNeps$ such that for all $\ell\in [L]$, 
    \begin{align*}
         \EE_n\left[I_{\ell}(X,\omega)\right]  ~ \leq~  
         \EE_n\left[I_{\ell}(X,\omega')\right] +  \left|  \EE_n\left[I_{\ell}(X,\omega)-I_{\ell}(X,\omega')\right]\right| .
    \end{align*}
    By invoking \cref{lem_taylor} and choosing  
    \[
    	\epsilon_1 = \frac{dL}{n\sqrt{\Dtmax}},\qquad \epsilon_2 =\frac{L}{n}\sqrt{d\log n\over \Dtmax}, \qquad \epsilon_3= \pmin {dL \over n},
    \]
    one has that for all  $\omega\in\cB_d(\omega^*)$,
        \begin{align} 
         \max_{\ell \in [L]}\max_{i \in [n]}    \left| I_\ell(X_i, \omega) - I_\ell(X_i, \omega')\right|  & \lesssim  \epsilon_2  \sqrt{d +\log n + \Dtmax}+ \epsilon_1 \sqrt{\Dtmax}  +  \frac{\epsilon_3}{\pmin} \notag\\
          &\lesssim {dL\log n\over n} \label{bd_taylor_max_I_ell}.
    \end{align}
	Thus, for all $\ell \in[L]$
	\begin{align}\label{bd_taylor_I_ell} 
    	 \left|  \EE_n\left[I_{\ell}(X,\omega)-I_{\ell}(X,\omega')\right]\right|
        &\lesssim {dL\log n\over n}.
    \end{align}
    We proceed to bound 
    $\max_{\ell \in[L]}\EE_n[I_{\ell}(X,\omega)]$ over $\omega \in \cNeps$.   
    For any $\omega\in\cNeps$,  define the set 
    \begin{equation}\label{def_set_S}
    	S(\omega) = \left\{
    	i\in [n]: \max_{k\ne \ell \in [L]} \left(
    	 |N_i^\T  \sw^{*1/2}J(\be_k - \be_\ell)| - \nu\Delta_{k\ell}
    	\right) > 0
    	\right\}
    \end{equation}
	for some constant  $0<\nu < (
	1/2  - 2c_0 - 2c_0^2)$.  Further let $[S(\omega)]^c := [n]\setminus S(\omega)$ be its complement set. 

     Recall the model \eqref{model}, we partition $[n] = \bigcup_{a=1}^L \wh G_a$ with 
 \begin{equation}\label{def_groups}
 	\wh G_a = \left\{i\in[n]: Y_i = a\right\}, \qquad n_a:= |\wh G_a|, \qquad \forall \ a\in [L].
 \end{equation}
  Further let 
 \begin{equation}\label{def_pi_tilde}
 \wt\pi_{\ell} := {n_\ell \over n},\qquad \forall~  \ell \in [L].
 \end{equation}
	We first observe that for any $i\in \wh G_\ell \cap [S(\omega)]^c$, using   \eqref{cond_distr_W}, \cref{lem_delta_J_order} and the choice of $\nu$ gives 
	\begin{align}\label{bd_I_ell_ell}
		I_{\ell}(X_i,\omega) \le  {1 \over \pi_{\ell} } \sum_{a \in [L]\setminus\{\ell\}} \pi_a \exp(N_i^\T \sigma_{a\ell}(\omega) + \delta_{a\ell}^\ell (\omega))\le  {1\over \pi_{\ell}} \exp(-c'\Dtmin).
	\end{align}
	Similarly, for any $b\in [L]\setminus \{\ell\}$ and $i\in \wh G_b\cap [S(\omega)]^c$, we have 
    \begin{align}
    		I_{\ell}(X_i,\omega) \le   {\pi_{\ell} \over  \pi_b \exp(N_i^\T \sigma_{b\ell}(\omega) + \delta_{b\ell}^b (\omega))} \overset{\eqref{lem_delta_J_order}}{\le}  {\pi_{\ell} \over \pi_b} \exp(-c'\Dtmin).\label{bd_I_ell_b}
    \end{align}
	Combining these two gives 
	\begin{align}\label{bd_Iell_S_comp} \nonumber
		  {1\over n}\sum_{i \in [S(\omega)]^c} I_{\ell}(X_i, \omega) 
		 &\le   {1\over n}\sum_{i \in [S(\omega)]^c} \left(1\{Y_i = \ell\}+   \sum_{b\in [L]\setminus\{\ell\}}  1\{Y_i = b\}\right) I_{\ell}(X_i, \omega)\\\nonumber
		 &\le {n_\ell \over n\pi_{\ell}} \exp(-c'\Dtmin) + \sum_{b\in [L]\setminus\{\ell\}} {n_b\over n} {\pi_{\ell} \over \pi_b} \exp(-c'\Dtmin)\\
		 &\lesssim (1 + L \pi_{\ell} )\exp(-c'\Dtmin).
	\end{align}
    On the other hand, by letting 
    \[
    	Z_i(\omega) :=1 \left\{\max_{k\ne \ell \in [L]\setminus \{\ell\}} \left( |N_i^\T \sw^{*1/2}J(\be_k - \be_\ell)|
     - \nu\Delta_{k\ell}\right) > 0\right\}
    \]
    we know that $	Z_1(\omega),\ldots, 	Z_n(\omega) $ are i.i.d. with 
    \begin{align}\label{tail_prob_Z_i}\nonumber
    	\EE[Z_i(\omega)] &\le 2\sum_{k\ne \ell \in [L]} \PP\left(
    		N_i^\T  \sw^{*1/2}J(\be_k - \be_\ell)  > \nu \Delta_{k\ell}
    	\right)\\\nonumber
    	&\le ~ 2\sum_{k\ne \ell \in [L]}  \exp\left(-\nu^2\Delta_{k\ell}^2 / \| \sw^{*1/2}J(\be_k - \be_\ell)\|_2^2\right)\\
    	&\le ~ 2 L^2 \exp(-c\Dtmin) &&\text{by \cref{lem_delta_J_order}}.
    \end{align}
	An application of the Bernstein inequality together with taking the  union bound over  $\omega \in \cNeps$ gives that with probability at least $1-n^{-dL}$, 
    \begin{align}\label{deviation_Zi}
     	{1\over n}\left| \sum_{i\in [n]} \left(  Z_i(\omega)  - \EE[Z_i(\omega)]\right)\right| \lesssim \sqrt{ dL^3\log q \over n} \exp(-c\Dtmin) + {dL\log q\over n}
    \end{align}
	holds for all  $\omega \in \cNeps$.   Here we let  $q:=  n (\pi_{\min}^{*-1}\vee\Dtmax)$,  With  the same probability, we therefore have
	\begin{align}\nonumber
		 \max_{\omega\in \cNeps} ~ 	|S(\omega)|  
		 &\le \sqrt{n dL^3\log(q)  } \exp(-c\Dtmin) + dL\log(q) +  \max_{\omega\in \cNeps} \ 	  \EE[Z_i(\omega)]\notag\\
		&\lesssim ~ dL\log(q) + n L^2\exp(-c\Dtmin) .\label{bd_S_card}
	\end{align}
	The last step uses $n\ge d$, \eqref{tail_prob_Z_i} and \eqref{cond_min_Delta}.
	Armed with \eqref{bd_Iell_S_comp} and \eqref{bd_S_card}, we find that 
	\begin{align*}
		\max_{\omega \in \cNeps}  \max_{\ell\in[L]} 
		{1\over \pi_{\ell}} \EE_n\left[I_{\ell}(X,\omega)\right]&\le \max_{\omega \in \cNeps}  \max_{\ell\in[L]}{1\over n\pi_{\ell}} \left\{
		 \sum_{i \in S(\omega)}+ \sum_{i \in [S(\omega)]^c} 
		\right\} I_{\ell}(X_i,\omega)\\
		&\le \max_{\omega \in \cNeps}  \max_{\ell\in[L]} \left\{{|S(\omega)|\over n\pi_{\ell}} +  \left({1\over \pi_{\ell}} +  L\right)\exp(-c'\Dtmin)\right\}\\
		&\lesssim ~ {L^2\over \pmin}\exp(-c\Dtmin) + {dL\log(q)\over n\pmin}.
	\end{align*}
	The last step also uses \eqref{eq_order_pis}. This completes the proof of \eqref{bd_cR_hat_0}.
\end{proof}

The following  lemma controls the discretization error in $I_\ell(X_i,\omega)$.

\begin{lemma}\label{lem_taylor}
     For any $\omega \in \cB_d(\omega^*)$ and $\omega' \in \cNeps$, on the event $\cE_2$,  we have 
    \begin{align*}
         &\max_{\ell \in [L]}\max_{i \in [n]} ~    \left| I_\ell(X_i, \omega) - I_\ell(X_i, \omega')\right|   \lesssim  \epsilon_2  \sqrt{d +\log n + \Dtmax}+ \epsilon_1 \sqrt{\Dtmax}  +  \frac{\epsilon_3}{\pmin}.
    \end{align*}
\end{lemma}
\begin{proof}
	Recall from \eqref{def_I_ell} and \eqref{def_I_aell} that 
	\[
		 I_\ell(x, \omega) =   \frac{\pi_{\ell}\sum_{a\ne \ell} \pi_a \exp\{(x - \frac{1}{2}(\mu_\ell + \mu_a))^\T J (\be_a - \be_\ell)\}}{(\pi_\ell + \sum_{k\neq \ell} \pi_k \exp\{(x - \frac{1}{2}(\mu_\ell + \mu_k))^\T J (\be_k - \be_\ell\})^2} \overset{\eqref{def_gamma_omega}}{=} \gamma_\ell(x;\omega) \left(
		 1-\gamma_{\ell}(x;\omega)\right).
	\]
	It is easy to see that  
	\begin{align*}
		\left| I_\ell(x, \omega)  - I_\ell(x, \omega')\right| &\le 3\left| \gamma_{\ell}(x;\omega) -  \gamma_{\ell}(x;\omega')\right|.
	\end{align*}
     Define $$f_{h}^{(\ell)}(x,\omega)=(x - \frac{1}{2}(\mu_h + \mu_\ell))^\T J(\be_h - \be_\ell)+\log{\frac{\pi_h}{\pi_\ell}},\quad \forall ~ h\in [L].$$ 
     For any $a\in \RR^L$ with $a_\ell = 0$, define $\psi_\ell (a)= 1 / (\sum_{h=1}^L e^{a_h}).$  Since
\begin{align*}
    \psi_\ell(a) - \psi_\ell(b) 
    & = \frac{\sum_{h}e^{a_h}(1-e^{b_h -a_h})}{\sum_{h}e^{a_h} \sum_{h}e^{b_h}} \\
    & \leq \frac{\sum_{h}e^{a_h}|b_h -a_h|}{\sum_{h}e^{a_h} \sum_{h}e^{b_h}} &&\text{by } 1-e^x \leq   |x| \text{ and } e^{b_\ell} = 1 \\ 
    &\leq   \max_{h \in [L]} |a_h-b_h|, 
\end{align*}
applying it with $a = f^{(\ell)}(X,\omega)$ and $b=  f^{(\ell)}(X,\omega')$ gives 
\begin{align*}
	  \left| \gamma_{\ell}(X,\omega) -  \gamma_{\ell}(X,\omega')\right|
	&=  \left|  \psi(f^{(\ell)}(X,\omega)) - \psi(f^{(\ell)}(X,\omega'))   \right|   \leq   \max_{h \in [L]} |f^{(\ell)}_{h}(X,\omega) - f^{(\ell)}_{h}(X, \omega')| .
\end{align*}
By using \eqref{cond_distr_Z}, we have 
\begin{align*}
     & |f_{h}^{(\ell)}(X,\omega) - f^{(\ell)}_{h}(X,\omega')| \\
    & \leq \left| N^\T\sw^{*1/2}(J-J')(\be_h - \be_\ell)\right|  + \max_{b\in[L]}\left|\left(\mu_b^*- \frac{1}{2}(\mu_h + \mu_\ell)\right)^\T (J - J^\prime)(\be_h -\be_\ell)\right| \\
    &\quad + \frac{1}{2} \left| (\mu_h -\mu_h^\prime + \mu_\ell - \mu_\ell^\prime)^\T J(\be_h - \be_\ell^\prime) \right|  + \left|\log(\pi_h/\pi_h') - \log(\pi_\ell/\pi_\ell')\right| \\
    & \lesssim \epsilon_2 \|N\|_2  + \sqrt{\Dtmax} (\epsilon_1 + \epsilon_2) + \frac{\epsilon_3}{\pi_{\min}} 
\end{align*} 
where the last step uses \eqref{mu_diff}, \eqref{bd_J_col_diff} and 
$$\left|\log \frac{\pi_h}{\pi_h'} \right| \leq \frac{|\pi_h - \pi'_h|}{\pi_h  \wedge \pi'_h} \lesssim  {\epsilon_3 \over \pi_{\min}}$$
due to   the concavity of $x \mapsto \log(x)$.
 Invoking $\cE_2$ and \eqref{eq_order_pis} completes the proof.  
\end{proof}

\subsubsection{Proof of \cref{lem_I_fN_samp} for $j=1$}
\begin{proof}
	The proof is based on $\cE_2$ and 
	\[
		\cE_{\op} := \left\{ {1\over n}\left\|
		\sum_{i=1}^n N_iN_i^\T  
		\right\|_\op \le   {2d \over n}+12\right\}
	\]
	which according to \cref{lem_N_op_norm} holds with probability $1-e^{-n}$.
    By definition in \eqref{def_cRs_samp} and \eqref{def_I_ell}, we have 
    \begin{align*}
    	 \wh\cR_I^{(1)} &= \sup_{\omega\in\cB_d(\omega^*)} \max_{\ell \in[L]} {1\over \pi_{\ell}}\left\| \EE_n\left[I_{\ell}(X,\omega) N N^\T\right]\right\|_\op.
    \end{align*} 
	Since for any $\omega \in \cB_d(\omega^*)$, there exists some $\omega' \in \cNeps$ such that 
	\begin{align*}
		&\left\| \EE_n\left[I_{\ell}(X,\omega) N N^\T\right]\right\|_\op -  \left\| \EE_n\left[I_{\ell}(X,\omega') N N^\T\right]\right\|_\op\\ 
		&\le {1\over n} \left\|\sum_{i=1}^n N_i N_i^\T \right\|_\op \max_{i \in [n]} \left|
		I_\ell(X_i,\omega) - I_\ell(X_i,\omega') 
		\right|\\
		&\lesssim   {dL\log n\over n} &&\text{by \eqref{bd_taylor_max_I_ell} and $\cE_{\op}$},
	\end{align*} 
	it suffices to bound, for all $\omega \in \cNeps$ and $\ell \in[L]$,
	\begin{align*}
		  \left\| \EE_n\left[I_{\ell}(X,\omega) N N^\T\right]\right\|_{\op}
		&\le   {1\over n}\left\|
		\sum_{i \in S(\omega)} I_{\ell}(X_i,\omega) N_iN_i^\T  
		\right\|_\op + {1\over n}\left\|
		\sum_{i \in [S(\omega)]^c} I_{\ell}(X_i,\omega) N_iN_i^\T  
		\right\|_\op \\
		&\le    {1\over n}\left\|
		\sum_{i \in S(\omega)}  N_iN_i^\T  
		\right\|_\op + {1\over n}\left\|
		\sum_{i=1}^n N_iN_i^\T  
		\right\|_\op\max_{i\in  [S(\omega)]^c} I_{\ell}(X_i,\omega) \\
		&\lesssim {1\over n}\left\|
		\sum_{i \in S(\omega)}  N_iN_i^\T  
		\right\|_\op +  \left({1\over \pi_{\ell}}+ {\pi_{\ell} \over \pi_{\min}}\right) \exp(-c'\Dtmin).
	\end{align*}
	The last step uses $\cE_{\op}$, \eqref{bd_I_ell_ell} and \eqref{bd_I_ell_b}.
	Let $m := C(dL\log(q) + n L\exp(-c\Dtmin))$ for some  large constant $C>0$.  It suffices to consider $m \le n$ as we can set $m=n$ otherwise. 
  Invoking  \cref{lem_N_op_norm} for $S = S(\omega) \le m$ together with \eqref{bd_S_card} gives that with probability at least $1-e^{-m} \ge 1 - n^{-dL}$,
  \begin{align*} 
  	{1\over n}\left\|
  	\sum_{i \in S(\omega)}  N_iN_i^\T  
  	\right\|_\op  &\le  
  	{2d + 12m\log(en/m) \over n} \lesssim { dL\log(q) \log(n)  \over n} +L^2\exp(-c\Dtmin),
  \end{align*}  
   thereby completing the proof.
\end{proof}

 \subsubsection{Proof of \cref{lem_I_fN_samp} for $j=2$}
 \begin{proof}
 	By repeating the same argument of bounding $\cR_I^{(2)}$ in \eqref{lem_I_fN} with $\EE$ replaced by $\EE_n$, we have 
 	$\wh\cR_I^{(2)} \le (L-1)\wh\cR_I^{(1)}$. 
 \end{proof}

\subsection{Proof of \cref{thm_EM_samp_known} \& \cref{thm_EM_samp}}\label{app_sec_proof_thm_EM_samp}

\begin{proof}
	For any $t\ge 1$ and each $h\in \{\pi, M,\Sigma\}$,  triangle inequality gives
	\begin{align*}
		d(\wh h^{(t)},  \bh(\theta^*))  &\le d\left(\wh \bh(\wh\theta^{(t-1)}), \wh \bh(\theta^*)\right) + d\left(\wh \bh(\theta^*), \bh(\theta^*)\right) \\
		&=d\left(\wh \bh(\wh\omega^{(t-1)}), \wh \bh(\omega^*)\right) + d\left(\wh \bh(\omega^*), \bh(\omega^*)\right) 
	\end{align*}
	where in the last step we also re-parametrize by using $\wh\omega^{(t-1)} := \omega(\wh\theta^{(t-1)})$.
	Provided that $\wh\theta^{(t-1)}$ satisfies \eqref{cond_init_unknown}, its reparametrization  $\wh\omega^{(t-1)}\in \cB_d(\omega^*)$ given by \eqref{cond_init_w_J}.  We work on the event under which \cref{thm_concent} and \cref{lem_I_fN_samp} holds and for notational simplicity, define 
	\[
	\epsilon_\pi = C\sqrt{\log n\over n\pmin},\quad \epsilon_M = C\sqrt{d\log n\over n\pmin},\quad \epsilon_\sw= C\sqrt{d\log n\over n}.
	\]

	For known $\sw^*$, invoking \cref{thm_EM_samp_omega} and \cref{lem_lip_J_basic} thus gives  
	\begin{align}\label{one_step_contract_pi}\nonumber
		d(\wh\bpi(\wh\omega^{(t-1)}),\wh \bpi(\omega^*)) 
		&\le ~ (\kappa/3) \left\{d(\wh\pi^{(t-1)},\pi^*) +   d(\wh M^{(t-1)},M^*) + d(J(\wh \theta^{(t-1)}),J(\theta^*)) \right\}\\
		&\le ~  \kappa \left\{d(\wh\pi^{(t-1)},\pi^*) +   d(\wh M^{(t-1)},M^*)  \right\}
	\end{align}
	where 
	\[
	\kappa = C\xi_n\sqrt{\Dtmax+L}  \overset{ \eqref{def_xi_n}}{=} C \left({1\over \pminsq}\exp(-c\Dtmin) + {dL\log(q)\log(n)\over n\pmin}\right)\sqrt{\Dtmax+L}.  
	\]
	We note $\log q\asymp \log n$ under \eqref{rate_cond_param_known}.
	Similarly, we have
	\begin{align}\label{one_step_contract_M}
		d( \wh \bM(\wh\omega^{(t-1)}), \wh \bM(\omega^*)) 
		&\le    \kappa \sqrt{\Dtmax} \left\{d(\wh\pi^{(t-1)},\pi^*) + d(\wh M^{(t-1)},M^*)  \right\}.
	\end{align}
	Together with \cref{thm_concent}, we find that
	\begin{align}\nonumber
		d(\wh\pi^{(t)}, \pi^*)  &  \le  \epsilon_\pi  +  \kappa  \left\{d(\wh\pi^{(t-1)},\pi^*) +   d(\wh M^{(t-1)},M^*)  \right\}\\\nonumber
		&\le \epsilon_\pi  +   (1/2+c_0\sqrt{\Dtmin}) \kappa\notag\\
		& \le 1/2\label{init_cond_iter_pi}
	\end{align}
	where the last step uses \eqref{cond_init_unknown} with $c_\sw=0$ and conditions \eqref{cond_min_Delta} \& \eqref{rate_cond_param_known}. Similarly, 
	\begin{align}\label{init_cond_iter_M}\nonumber
		d(\wh M^{(t)}, M^*)   
		&\le \epsilon_M  +  \kappa \sqrt{\Dtmax} \left\{d(\wh\pi^{(t-1)},\pi^*) +   d(\wh M^{(t-1)},M^*)  \right\}\\\nonumber
		& \le \epsilon_M  +  (1/2+c_0\sqrt{\Dtmin})\kappa\sqrt{\Dtmax}\\
		&  \le c_0 \sqrt{\Dtmin}.
	\end{align} 
	The last two displays ensure that $\wh\theta^{(t)}$ satisfies \eqref{cond_init_unknown} hence $\wh\omega^{(t)}\in \cB_d(\omega^*)$. By  \cref{one_step_contract_pi,one_step_contract_M} and induction, we obtain
	\begin{align*} 
		d(\wh\pi^{(t)}, \pi^*)  &\le \epsilon_\pi  +  \kappa  \left\{d(\wh\pi^{(t-1)},\pi^*) +   d(\wh M^{(t-1)},M^*)  \right\}\\
		&\le \epsilon_\pi + \kappa(\epsilon_\pi+\epsilon_M) + \kappa^2(1+\sqrt{\Dtmax}) \left\{d(\wh\pi^{(t-2)},\pi^*) +   d(\wh M^{(t-2)},M^*)  \right\}\\
		&\le \epsilon_\pi+  (\epsilon_\pi + \epsilon_M) \kappa  \left[1 +\kappa(1+\sqrt{\Dtmax})+ \cdots +  \kappa^{t-2}(1+\sqrt{\Dtmax})^{t-2}\right]\\
		&\qquad +  \kappa^t (1+\sqrt{\Dtmax})^{t-1}\left\{d(\wh\pi^{(0)},\pi^*) +    d(\wh M^{(0)},M^*)  \right\}\\
		&\le  \epsilon_\pi+  {\kappa  \over 1-\kappa-\kappa\sqrt{\Dtmax}} (\epsilon_\pi + \epsilon_M) + (2\kappa\sqrt{\Dtmax})^t \left\{d(\wh\pi^{(0)},\pi^*) +   d(\wh M^{(0)},M^*)  \right\}
	\end{align*}
	as well as, by using \eqref{one_step_contract_M},
	\begin{align*}
		d(\wh M^{(t)}, M^*)&\le \epsilon_M  +  \kappa \sqrt{\Dtmax} \left\{d(\wh\pi^{(t-1)},\pi^*) +   d(\wh M^{(t-1)},M^*)  \right\}\\ 
		&\le  \epsilon_M+  {\kappa\sqrt{\Dtmax} \over 1-2\kappa\sqrt{\Dtmax}}(\epsilon_\pi + \epsilon_M) +( 2\kappa\sqrt{\Dtmax})^t \left\{d(\wh\pi^{(0)},\pi^*)  +   d(\wh M^{(0)},M^*)  \right\}.
	\end{align*}
	Therefore,  with $\kappa_n = 2\kappa \sqrt{\Dtmax}$, one has that
	for all $t\ge 1$,  
	\begin{align*} 
		d(\wh\pi^{(t)}, \pi^*) +d(\wh M^{(t)}, M^*) 
		&~ \lesssim ~   \sqrt{d\log n\over n\pmin}  +\kappa_n^t \left\{d(\wh\pi^{(0)},\pi^*) +   d(\wh M^{(0)},M^*)  \right\},
	\end{align*}   
	completing the proof of \cref{thm_EM_samp_known}.\\
	
	For unknown $\sw^*$,  denote the distance between the initialization $\wh \theta^{(0)}$ and $\theta^*$ by
	\begin{equation}\label{def_dist_theta}
		d(\wh \theta^{(0)},\theta^*) := d(\wh\pi^{(0)},\pi^*) +   d(\wh M^{(0)},M^*) +\sqrt{\Dtmax} ~ d(\whsw^{(0)}, \sw^*).
	\end{equation}
	Note that condition \eqref{cond_init_unknown} ensures $d(\wh \theta^{(0)},\theta^*) \le 1/2 + c_\pi \sqrt{\Dtmin}+c_\sw\sqrt{\Dtmax}$.
	By similar arguments, we note from \cref{thm_EM_samp_omega} and \cref{lem_lip_J_basic} that 
	\begin{align*}
		d(\wh\bpi(\wh\omega^{(t-1)}),\wh \bpi(\omega^*))  &~ \le~  \kappa  ~  d(\wh \theta^{(t-1)},\theta^*)\\
		d( \wh \bM(\wh\omega^{(t-1)}), \wh \bM(\omega^*))  &~ \le~ \kappa\sqrt{\Dtmax}  ~ d(\wh \theta^{(t-1)},\theta^*)\\
		d(\bwhsw(\wh \omega^{(t-1)}),\bwhsw(\omega^*))  &~ \le~    \kappa {\Dtmax}  ~ d(\wh \theta^{(t-1)},\theta^*).
	\end{align*} 
	Invoking \cref{thm_concent}  and using the fact that   $\wh\theta^{(t-1)}$  satisfies \eqref{cond_init_unknown} give  
	\begin{align*}  
		d(\whsw^{(t)},\sw^*) &\le \epsilon_{\sw} +  \kappa{\Dtmax} ~ d(\wh \theta^{(t-1)},\theta^*)\\
		&\le \epsilon_{\sw}  + \kappa \Dtmax \left(1/2 + (c_\mu+c_\sw) \sqrt{\Dtmax}    \right) \\
		&\le c_\sw &&\text{by \eqref{rate_cond_param_unknown}}.
	\end{align*} 
	By similar arguments, it is easy to see that \eqref{init_cond_iter_pi} and \eqref{init_cond_iter_M} still hold
	so that $\wh\theta^{(t)}$ satisfies \eqref{cond_init_unknown} hence   $\wh\omega^{(t)}\in \cB_d(\omega^*)$. By the analogous induction argument, letting $\kappa_n'= \kappa + \kappa \sqrt{\Dtmax} + \kappa \Dtmax^{3/2}$ and using \eqref{def_dist_theta} give
	\begin{align*}
		d(\wh\pi^{(t)}, \pi^*)   &\le  \epsilon_\pi   + {\kappa  \over 1-\kappa_n'} \left(\epsilon_\pi+\epsilon_M +\sqrt{\Dtmax}~ \epsilon_{\sw}\right) + (\kappa_n')^t  ~ d(\wh \theta^{(0)}, \theta^*),\\
		d(\wh M^{(t)}, M^*) 
		&\le \epsilon_M   + {\kappa \sqrt{\Dtmax} \over 1-\kappa_n'} \left(
		\epsilon_\pi + \epsilon_M+ \sqrt{\Dtmax} ~ \epsilon_{\sw}
		\right)+ (\kappa_n')^t  ~ d(\wh \theta^{(0)}, \theta^*)\\
		d(\whsw^{(t)},\sw^*)   
		& \leq \epsilon_{\sw} + {\kappa \Dtmax  \over 1-\kappa_n'}(\epsilon_\pi + \epsilon_M + \sqrt{\Dtmax} ~ \epsilon_{\sw}) + (\kappa_n')^t  ~ d(\wh \theta^{(0)}, \theta^*).
	\end{align*}
	This in conjunction with \eqref{rate_cond_param_unknown} yields 
	\begin{align*}
		d(\wh\pi^{(t)}, \pi^*)  + d(\wh M^{(t)}, M^*) 
		&~ \lesssim ~ \sqrt{d \log n\over n\pmin}  +(\kappa_n')^t ~ 	d(\wh \theta^{(0)},\theta^*) ,\\
		d(\whsw^{(t)},\sw^*)   
		& ~ \lesssim ~  \sqrt{d\log n\over n} +  (\kappa_n')^t  ~ 	d(\wh \theta^{(0)},\theta^*),    \end{align*}
	thereby completing  the proof of \cref{thm_EM_samp}. 
\end{proof}

\subsection{Proof of \cref{rem_sT_samp}}\label{app_sec_proof_rem_sT_samp}

\begin{proof}
    By \eqref{lip_sw}, we have 
    \[
        d(\whsw^{(0)}, \sw^*) \le 3\sqrt{\Dtmax} d(\wh M^{(0)},M^*) + 2\Dtmax d(\wh\pi^{(0)},\pi^*) + d(\whsT, \sT^*).
    \]
    The claim follows by invoking \cref{lem_Sigma_hat_sup} stated and proved below together with $\Dtmax\sqrt{d\log n / n} \le c$ under condition \eqref{rate_cond_param_unknown}.
\end{proof}

Recall $\whsT = n^{-1}\sum_{i=1}^n X_i X_i^\T$. The following bounds from above $d(\whsT,\sT^*)$.
\begin{lemma}\label{lem_Sigma_hat_sup}
    	Under model~\eqref{model}, assume $d\log n\le n$. With probability at least $1-n^{-d}$, one has
        \[
            d(\whsT, \sT^*) ~ \lesssim ~ (\Dtmax + 1) \sqrt{d\log n \over n}.
        \]
    \end{lemma}
    \begin{proof}
    	  Let $\cN_{1/3}$ be the $(1/3)$-net of $\Sp^d$ and $|\cN_{1/3}|\leq 7^d$. Standard discretization argument gives 
    $$ d(\whsT, \sT^*) = \sup_{v \in \Sp^d} v^\T  \sw^{*-1/2}(\whsT - \sT^*) \sw^{*-1/2}  v \leq 3 \max_{v \in \cN_{1/3}} v^\T \sw^{*-1/2}(\whsT - \sT^*) \sw^{*-1/2}  v.$$
    Fix any $v\in \cN_{1/3}$. By definition, we have  
    	\begin{align*}
    		\whsT &= {1\over n}\bX^\T \bX  = {1\over n}\sum_{k=1}^L \sum_{i=1}^n 1\{ Y_i = k\}  X_i X_i^\T \\
    		&=  {1\over n}\sum_{k=1}^L\left[
    		\sum_{i=1}^n 1\{ Y_i = k\} (X_i - \mu_k)(X_i-\mu_k)^\T + n_k(\wt\mu_k  \mu_k^{*\T} + \mu_k^* \wt\mu_k^\T) - n_k\mu_k\mu_k^\T
    		\right]\\
    		&= \bar \sw +  \sum_{k=1}^L{n_k\over n}\mu_k^*\mu_k^{*\T} + \sum_{k=1}^L {n_k\over n}\left[
    		(\wt\mu_k - \mu_k^*) \mu_k^{*\T} +  \mu_k (\wt\mu_k - \mu_k^*)^\T
    		\right]
    	\end{align*}
    	where $\wt \mu_k$ is given in \eqref{def_theta_td} and  
    	\[
    	\bar\sw := {1\over n}\sum_{k=1}^L 
    	\sum_{i=1}^n 1\{ Y_i = k\} (X_i - \mu_k^*)(X_i-\mu_k^*)^\T,  
    	\]
    	By further using the decomposition of $\sT$ in \eqref{eq_sw_sT} and $\wt\pi_k = n_k/n$ for $k\in [L]$, we find that 
    	\begin{align*}
            & v^\T \sw^{*-1/2}(\whsT - \sT^*) \sw^{*-1/2}  v   \le  v^\T\sw^{*-1/2}(\bar\sw - \sw^*)\sw^{*-1/2} v \\
            & \qquad + \sum_{k=1}^L (\wt\pi_k - \pi_k) v^\T\sw^{*-1/2} \mu_k\mu_k^\T \sw^{*-1/2} v +  2\left|\sum_{k=1}^L \wt\pi_k 
    		v^\T\sw^{*-1/2} (\wt\mu_k - \mu_k^*) \mu_k^{*\T} \sw^{*-1/2} v\right|.
    	\end{align*}
        \citet[Lemmas 16 \& 17]{bing2024lineardiscriminantregularizedregression} give that for any $t\ge 0$, with probability at least $1-4e^{-t}$,
        $$
            v^\T\sw^{*-1/2}(\bar\sw - \sw^*)\sw^{*-1/2} v  \lesssim   \sqrt{t\over n} + {t\over n} 
     $$ 
     as well as  
    	\begin{align}\label{bd_pi_diff}
    		 \sum_{k=1}^L (\wt\pi_k - \pi_k) v^\T\sw^{*-1/2} \mu_k\mu_k^\T \sw^{*-1/2} v  &\lesssim \Dtmax\left(\sqrt{t  \over n} + {t  \over n}\right).
    	\end{align}
    By inspecting the proof of \citet[Lemma 15]{bing2024lineardiscriminantregularizedregression}, one can also deduce that 
    	\begin{align*}
    		\left|\sum_{k=1}^L \wt\pi_k 
    		v^\T\sw^{*-1/2} (\wt\mu_k - \mu_k^*) \mu_k^{*\T} \sw^{*-1/2} v\right| \lesssim  \|M^\T \sw^{*-1/2} v\|_\i \sqrt{t  \over n }  \le \sqrt{\Dtmax}\sqrt{t  \over n } 
    	\end{align*} 
        holds
    	with probability at least $1-2e^{-t}$. 
    The proof is completed by collecting all three bounds, choosing $t= Cd\log(n)$ for some large constant $C>0$ and using $d\log n\le n$.
    \end{proof}

\section{Proofs of the sample level results in \cref{sec_theory_samp_large}}\label{app_sec_proof_sample_large}
\subsection{Proof of \cref{thm_conv_phi}}\label{app_sec_proof_thm_conv_phi}

\begin{proof}
	We first prove \cref{thm_conv_phi} when $\wh \theta^{(t)}$ satisfies \eqref{cond_init_unknown} and
	\begin{equation}\label{cond_init_phi}
		L \left(
		1 + {1\over \pminsq \Dtmin^2} 
		\right){\phi(\wh \theta^{(t)})\over n} \le c.
	\end{equation}  
	Later we will argue that these conditions are met under either of the initialization requirements in \cref{thm_conv_phi}.

	 Pick any $\theta$ satisfying \eqref{cond_init_unknown} and rewrite $\omega = (\pi, M, J(\theta)) \in \cB_d(\omega^*)$.   For any $a\in [L]$ and $i\in \wh G_a$, we have 
	\begin{align*}
		\gamma_{\ell}(X_i; \omega)&\le {1\over 1 + \exp(N_i^\T \sw^{*1/2} J(\be_a - \be_\ell)  + \delta_{a\ell}^a(\omega) + \log(\pi_a/\pi_{\ell}))} &&\text{by \eqref{cond_distr_W}}\\
		&\le  {1\over 1 + \exp(N_i^\T \sw^{*1/2} J(\be_a - \be_\ell) + 2c\Delta_{a\ell}+ \log(\pi_a^*/\pi_{\ell}^*))} &&\text{by \eqref{cond_init_w_J} \& \cref{lem_delta_J_order}}\\
		&\le {1\over 1 + \exp(N_i^\T \sw^{*1/2} J(\be_a - \be_\ell) + c\Delta_{a\ell})} &&\text{by } \Dtmin \ge C\log(1/\pmin)\\
		&\le {1\over 1 + \exp(N_i^\T \sw^{*1/2} J^*(\be_a - \be_\ell) + (c/2)\Delta_{a\ell})}  \\
		&\quad + 1\left\{
		\left| N_i^\T \sw^{*1/2} (\wt J - J^*)(\be_a - \be_\ell)\right| >  (c/4)\Delta_{a\ell}
		\right\} \\
		&\quad + 1\left\{
		\left| N_i^\T  \sw^{*1/2} (J-\wt J)(\be_a - \be_\ell) \right| >  (c/4)\Delta_{a\ell}
		\right\}.
	\end{align*}
	Here we write 
	$
	 \wt J:= \wt \Sigma^{-1} \wt M
	$
	with $\wt M = (\wt \mu_1,\ldots, \wt \mu_L)$ and $\wt \Sigma$ being the estimators that use the true labels $Y_1,\ldots, Y_n$.  Specifically, recalling $\wt \pi_{\ell} = n_\ell / n$ from \eqref{def_pi_tilde}, we have
	\[
	\wt \mu_{\ell} = {1\over n_\ell} \sum_{i \in \wh{G}_\ell} X_i,\qquad \wt \sw = {1\over n}\sum_{\ell=1}^{L}\sum_{i \in \wh G_\ell} (X_i - \wt \mu_{\ell})(X_i - \wt \mu_{\ell})^\T.
	\] 
	By applying the above bound of $\gamma_{\ell}(X_i;\theta)$ to $\wh\theta^{(t)} =   (\wh \pi^{(t)}, \wh M^{(t)}, \whsw^{(t)}) $ with the corresponding $\wh \omega^{(t)} =  (\wh \pi^{(t)}, \wh M^{(t)},  J(\wh \theta^{(t)})) \in \cB_d(\omega^*)$, we obtain
	\begin{align}\label{decomp_phi}
		\phi(\wh\theta^{(t)}) =   \sum_{\ell = 1}^L \sum_{a\in[L]\setminus\{\ell\}}\sum_{i\in \wh G_a} \gamma_{\ell}(X_i;\wh \theta^{(t)}) \|\mu_{\ell}^* - \mu_a^*\|_{\sw^*}^2
		\le \rI +  \rII+  \rIII 
	\end{align}
	with
	\begin{align*}
		\rI &= \sum_{\ell = 1}^L \sum_{a\in[L]\setminus\{\ell\}}\sum_{i\in \wh G_a} {1\over 1 + \exp(N_i^\T \sw^{*1/2} J^*(\be_a - \be_\ell) + (c/2)\Delta_{a\ell})}  \Dt_{a\ell}\\
		\rII &= \sum_{\ell = 1}^L \sum_{a\in[L]\setminus\{\ell\}}\sum_{i\in \wh G_a} 1\left\{
		\left| N_i^\T\sw^{*1/2} (\wt J - J^*)(\be_a - \be_\ell)  \right| >  (c/4)\Delta_{a\ell}
		\right\}\Dt_{a\ell}\\
		\rIII&=\sum_{\ell = 1}^L \sum_{a\in[L]\setminus\{\ell\}}\sum_{i\in \wh G_a} 1\left\{
		\left| N_i^\T \sw^{*1/2} (J(\wh \theta^{(t)}) - \wt J)(\be_a - \be_\ell) \right| >  (c/4)\Delta_{a\ell}
		\right\}\Dt_{a\ell}.
	\end{align*}
	We proceed to bound each term in the sequel on the event $\cE_\pi$, given  in \eqref{def_event_pi}, intersecting with
	\begin{align}\label{def_event_N_G_op}
		\cE_{\op} &:= \bigcap_{\ell \in [L]} \left\{
		\left\|\sum_{i\in \wh G_\ell}   N_i N_i^\T  \right\|_\op \lesssim d + n_\ell
		\right\} \bigcap \left\{
		\left\|\sum_{i=1}^n   N_i N_i^\T  \right\|_\op\lesssim n + d
		\right\},\\\label{def_event_op_diff}
		\cE_{\op}' &:= \left\{	\left\|  {1 \over n} \sum_{i=1}^n    N_i N_i^\T - \bI_d \right\|_\op  \lesssim \sqrt{d+\log n\over n} \right\},\\\label{def_event_N_ell_2}
		\cE_{\ell_2} &:=	\bigcap_{\ell \in [L]} \left\{\left\| {1\over n_\ell}\sum_{i\in \wh G_\ell} N_i  \right\|_2 \lesssim \sqrt{d + \log n\over n_\ell}  \right\}.
	\end{align}  
	According to \cref{lem_N_op_norm,Gau_l2_bd,lem_op_diff}, the intersection event holds with probability $1-e^{-n\pmin}-n^{-2}-n^{-1}\ge 1-2n^{-1}$ by $n\pmin \ge C\log n$.  The event $\cE_{\ell_2}$ ensures that
	\begin{equation}\label{rate_mu_tilde} 
		\left\| \wt \mu_{\ell} - \mu_{\ell}^*\right\|_{\sw^*}= {1\over n_\ell}\left\|\sum_{i \in \wh{G}_\ell}  N_i\right\|_2 \lesssim \sqrt{d + \log n\over n_\ell} ,\qquad \forall ~ \ell \in [L],
	\end{equation}
	which, together with \eqref{eq_Dt_max_Dt_inf} and  $C d \log n \le n \pmin$, further implies 
	\begin{equation}\label{eq_order_mu_tilde}
		\|\wt \mu_{\ell}\|_{\sw^*} \le\left\| \wt \mu_{\ell} - \mu_{\ell}^*\right\|_{\sw^*}+ \|\mu_{\ell}^*\|_{\sw^*} \le 2\sqrt{\Dtmax}.
	\end{equation} 
	Moreover, using the decomposition in \eqref{decomp_sw_hat} and \eqref{rate_mu_tilde},  we find that on the event $\cE_{\ell_2} \cap \cE_{\op}'$, 
	\begin{align}\label{bd_sw_td}\nonumber
		d(\wt\sw,\sw^*) &\le  \left\|  {1 \over n} \sum_{i=1}^n    N_i N_i^\T - \bI_d \right\|_\op  + \left\| \sum_{\ell = 1}^L {n_\ell \over n}\Bigl( {1\over n_\ell}\sum_{i \in \wh{G}_\ell}  N_i\Bigr) \Bigl({1\over n_\ell}\sum_{i \in \wh{G}_\ell}  N_i\Bigr)^\T  \right\|_\op\\
		& \lesssim \sqrt{d + \log n\over n} + {L(d+\log n)\over n},
	\end{align} 
	which, by Weyl's inequality and $C d \log n \le n \pmin \le n/L$, implies that 
	\begin{equation}\label{bd_eig_sw_td}
		1/2 \le 	\lambda_j\bigl(\sw^{*-1/2}\wt\sw \sw^{*-1/2}\bigr)\le 2,\qquad \forall ~ 1\le j\le d.
	\end{equation} 
	
	\paragraph{Bounding $\rI$.} By \cref{def_event_N,bd_tail_prob_N1c},  we have
	\[
	\EE\left[ 
	{1\over 1 + \exp(N_i^\T \sw^{*1/2} J^*(\be_a - \be_\ell) + (c/2)\Delta_{a\ell})} 
	\right] \le \exp(-c'\Dt_{a\ell})
	\]
	which in conjunction with $\Dtmin \ge C\log (L)$  implies
	\begin{align*}
		\EE[\rI] &\le \sum_{\ell = 1}^L \sum_{a\in[L]\setminus\{\ell\}}\sum_{i\in \wh G_a}  \exp(-c'\Dt_{a\ell} + \log(\Dt_{a\ell})) \le  n  \exp(-c'' \Dtmin).
	\end{align*}
	By Markov's inequality, we conclude with probability at least $1-\exp(-(c''/2)\Dtmin)$, 
	\begin{align}\label{bd_I_phi}
		\rI  \le  \EE[\rI] \exp((c''/2)\Dtmin) \le n  \exp(-(c''/2) \Dtmin).
	\end{align}

	\paragraph{Bounding $\rII$.} 
	\cref{lem_sigma_diff_td} states that with probability at least $1- 4n^{-1}$, 
	\[
	\max_{\ell \in [L]}\max_{a\in[L]\setminus\{\ell\}}\max_{i \in \wh G_a}  {1\over \Delta_{a\ell} }\left| N_i^\T\sw^{*1/2} (\wt J - J^*)(\be_a - \be_\ell) \right| \le  c/4,
	\]
	whence $\rII = 0$.

	\paragraph{Bounding $\rIII$.}  
	Note that
	\begin{align*}
		\rIII &\le \sum_{\ell = 1}^L \sum_{a\in[L]\setminus\{\ell\}}\sum_{i\in \wh G_a}  {16 \over   c^2\Delta_{a\ell} }  \left[N_i^\T   \sw^{*1/2} (J(\wh\theta^{(t)})- \wt J)(\be_a - \be_\ell)\right]^2\\
		&\le {16 \over   c^2}\sum_{\ell = 1}^L \sum_{a\in[L]\setminus\{\ell\}}  \left\|\sum_{i\in \wh G_a}   N_i N_i^\T  \right\|_\op {1\over    \Delta_{a\ell}} \left\| \sw^{*1/2} (J(\wh\theta^{(t)})- \wt J)(\be_a - \be_\ell)\right\|_2^2\\
		&\le {16 \over   c^2}\sum_{\ell = 1}^L \sum_{a\in[L]\setminus\{\ell\}} (d+ n_a)  {1\over    \Delta_{a\ell}} \left\| \sw^{*1/2} (J(\wh\theta^{(t)})- \wt J)(\be_a - \be_\ell)\right\|_2^2 &&\text{by }\cE_{\op}.
	\end{align*}	
	Using $d(\whsw^{(t)},\sw^*)\le 1/2$ implied by $\wh \theta^{(t)}$ satisfying \eqref{cond_init_unknown},  reasoning as \eqref{decomp_J_diff} gives
	\begin{align*}
		& \|\sw^{*1/2}(J(\wh \theta^{(t)}) - \wt J)(\be_{a} - \be_{\ell})\|_2\\
		&\le 	\|\sw^{*1/2}\bwhsw^{(t)-1}\sw^{*1/2}\|_\op \left(  \|\wh \mu_{\ell}^{(t)}-\wt \mu_{\ell}\|_{\sw^*}  + 	 \|\wh \mu_{a}^{(t)}-\wt \mu_{a}\|_{\sw^*} \right)\\
		&\quad +  \|\sw^{*1/2}\bwhsw^{(t)-1}\sw^{*1/2}\|_\op 
		\|(\bwhsw^{(t)}-\wt\sw) \wt J(\be_a - \be_{\ell})\|_{\sw^*} 	\\
		&\le 2  \|\wh \mu_{\ell}^{(t)}-\wt \mu_{\ell}\|_{\sw^*}  + 	2 \|\wh \mu_{a}^{(t)}-\wt \mu_{a}\|_{\sw^*} + 2\|(\bwhsw^{(t)}-\wt\sw)\wt J(\be_a - \be_{\ell})\|_{\sw^*}.
	\end{align*}
	Note that  $(\wh \pi^{(t)}, \wh M^{(t)}, \whsw^{(t)})=(\wh \bpi(\wh \theta^{(t-1)}), \wh \bM(\wh \theta^{(t-1)}), \bwhsw(\wh \theta^{(t-1)}))$.
	Invoking \cref{lem_phi_params} with $\theta = \wh \theta^{(t-1)})$    gives  
	\begin{align*}
		\rIII  &\lesssim    \left( 
		{dL^2\log(q)\over n  \pminsq \Dtmin^2} +{L \phi(\wh \theta^{(t-1)}) \over n  \pminsq \Dtmin^2}    +  \exp(-2c'\Dtmin)
		\right)\phi(\wh \theta^{(t-1)})\hspace{-2cm}\\
		&\quad +  
		\left\{ \left({\log n\over \Dtmin}+1\right)\left( { dL^2\log(q) \over n }+ L \exp(-2c\Dtmin)\right)+   {L  \phi(\wh \theta^{(t-1)})\over n}\right\} \phi(\wh \theta^{(t-1)})\\
		&\le  {1\over 2}  \phi(\wh \theta^{(t-1)}) 
	\end{align*}
	where the last step is due to $\Dtmin \ge C \log(1/\pmin + \log n)$, condition \eqref{cond_init_phi} and the initialization requirement in \eqref{cond_n_phi}.
	Collecting the bounds for $\rI$, $\rII$ and $\rIII$ yields that: with probability at least  $1- 4n^{-1}-\exp(-c\Dtmin)$,
	\begin{equation}\label{contract_phi}
		\phi(\wh\theta^{(t)}) ~ \le ~  n  \exp(-c \Dtmin) + {1\over 2}\phi(\wh \theta^{(t-1)}).
	\end{equation}
	Since this together with $\Dtmin \ge C\log (1/\pmin)$ further ensures $\phi(\wh\theta^{(t)})$ satisfying \eqref{cond_init_phi}, by induction and noting that the high probability bound in \eqref{contract_phi} holds uniformly for all $\wh \theta^{(t-1)}$, the claim is proved. 
	
	Finally, we note that when  $\wh \theta^{(t)}$  meets the requirement in (a), the above proof implies that with probability  $1- 4n^{-1}-\exp(-c\Dtmin)$,
	\[
		\phi(\wh \theta^{(t)}) \le n\exp(-c\Dtmin) + { 64 nL\over c^2} \left(2c_\mu^2 +  c_{\sw}^2\right)
	\]
	so that \eqref{cond_init_phi} holds  as 
	\[
		L \left(
		1 + {1\over \pminsq \Dtmin^2} 
		\right){\phi(\wh \theta^{(t)})\over n} \le   \left(
		1 + {1\over \pminsq \Dtmin^2} 
		\right) \left( L\exp(-c\Dtmin) + { 64 L^2\over c^2} \left(2c_\mu^2 +  c_{\sw}^2\right)\right) \le c'
	\]
	by $\Dtmin \ge C\log(1/\pmin)$ and (a).
	On the other hand, when $\wh \theta^{(t)}$ meets the requirement in (b), then \cref{eq_phi_pi,eq_phi_mu,eq_phi_sw}  and (b) ensure  that $\wh \theta^{(t)}$ also satisfies \eqref{cond_init_unknown}. The proof is complete.  
\end{proof}

\subsection{Proof of \cref{thm_conv_phi_known}}\label{app_sec_proof_thm_conv_phi_known}
\begin{proof}
	The proof follows from that of \cref{thm_conv_phi} and the only difference is in bounding $\rIII$ where the term $\|(\bwhsw^{(t)}-\wt\sw)\wt J(\be_a - \be_{\ell})\|_{\sw^*} = 0$ in this case. 
	
	Regarding the initialization requirement, for $\wh \theta^{(0)} = (\wh \pi^{(0)},\wh M^{(0)},\sw^*)$, if $\phi(\wh \theta^{(0)})/n \le c \pmin \Dtmin(1+\pmin\Dtmin)$, then \eqref{eq_phi_pi} and \eqref{eq_phi_mu} ensure that $\wh \theta^{(0)}$ satisfies \eqref{cond_init_unknown} with $c_{\sw} = 0$. 
	On the other hand, if $\wh \theta^{(0)} = (\wh \pi^{(0)},\wh M^{(0)},\sw^*)$ satisfies \eqref{cond_init_unknown}, then the proof reveals 
	\[
	{\phi(\wh \theta^{(0)})\over n} \le {C L \over \Dtmin}d(\wh M^{(0)}, M^*)^2\le C  L c_{\mu}^2
	\] 
	so that   $c_\mu \le  \min  \{(\sqrt{2}-1)/2,   c \pmin \Dtmin/L\}$ 
	suffices.
\end{proof}

\subsection{Technical lemmas used in the proof of \cref{thm_conv_phi}}\label{app_sec_tech_lemmas_thm_conv_phi}

We first state two useful facts. 
\begin{fact}\label{fact_gamma_in_notin}
	The following equality holds for any $\theta$: 
	\begin{align*}
		\sum_{\ell=1}^{L}\sum_{i \in \wh{G}_\ell} \left(
		1-\gamma_{\ell}(X_i; \theta)
		\right) &=\sum_{\ell=1}^L \sum_{i\notin \wh G_\ell} \gamma_\ell(X_i; \theta),\\
		\sum_{\ell=1}^{L}\sum_{i \in \wh{G}_\ell} \left(
		1-\gamma_{\ell}(X_i; \theta) 
		\right) N_i&= \sum_{\ell=1}^L  \sum_{i\notin \wh G_\ell} \gamma_\ell(X_i; \theta) N_i.
	\end{align*}
\end{fact}
\begin{proof}
	By definition,
	\begin{align*}
		\sum_{\ell=1}^{L}\sum_{i \in \wh{G}_\ell} \left(
		1-\gamma_{\ell}(X_i; \theta)
		\right)  &= \sum_{\ell=1}^{L}\sum_{i \in \wh{G}_\ell} \sum_{k \in [L]\setminus\{\ell\}}  \gamma_k(X_i; \theta) \\
		&= \sum_{\ell=1}^{L}\left( \sum_{k=1}^L \sum_{i \in \wh{G}_\ell}   \gamma_k(X_i; \theta) - \sum_{i\in \wh G_\ell}  \gamma_\ell(X_i; \theta)\right)\\
		&= \sum_{k=1}^L \sum_{\ell=1}^{L}  \sum_{i \in \wh{G}_\ell}   \gamma_k(X_i; \theta) -\sum_{k=1}^{L}  \sum_{i\in \wh G_k}  \gamma_k(X_i; \theta)\\
		&=  \sum_{k=1}^L\sum_{\ell\in [L]\setminus\{k\}} \sum_{i \in \wh{G}_\ell}   \gamma_k(X_i; \theta)  \\
		&= \sum_{k=1}^L \sum_{i\in \wh G_k} \gamma_k(X_i; \theta).
	\end{align*}
	Similar argument can be used to prove the second statement. 
\end{proof}

There are a few useful inequalities about $\phi(\theta)$ which follow directly by definitions of $\phi(\theta)$  in \eqref{def_phi_theta} and $\Delta_{\ell k}$ in \eqref{def_Delta} as well as  \cref{fact_gamma_in_notin}.
\begin{fact}
	For any $\theta$, we has the following inequalities.
	\begin{align}\label{bd_phi_gamma_ell}
		\sum_{\ell = 1}^L \sum_{i\in \wh G_\ell} (1-\gamma_{\ell}(X_i;\theta)) = \sum_{\ell=1}^L \sum_{i\notin \wh G_a} \gamma_\ell(X_i; \theta)&~\le~ {\phi(\theta)\over \Dtmin},\\\label{bd_phi_gamma_mu_diff}
		\sum_{\ell=1}^L \sum_{a\in[L]\setminus\{\ell\}}\sum_{i\in \wh G_a} \gamma_\ell(X_i; \theta)\sqrt{\Dt_{a\ell}}&~\le~ {\phi(\theta)\over \sqrt{\Dtmin}}.
	\end{align} 
\end{fact}

\begin{lemma}\label{lem_sigma_diff_td} 
	Grant  $\Dtmin \gtrsim 1$, $n\ge dL^2$ and  $ C d\log n\le n \pmin $  for some large constant $C>0$.
	The following holds with probability at least $1- 4n^{-1}$,
	\[
	\max_{\ell \in [L]}\max_{a\in[L]\setminus\{\ell\}}\max_{i \in \wh G_a}  ~ {1\over \Delta_{a\ell} }\left| N_i^\T \sw^{*1/2}(\wt J-J^*)(\be_a - \be_{\ell})\right| \le  c/4.
	\]
\end{lemma}
\begin{proof}
	We work on the event $\cE_{\ell_2}\cap \cE_2\cap \cE_\pi\cap \cE_{\op}' $ and   first condition on $Y_1,\ldots, Y_n$.
	Fix any $a \ne \ell$ and $i\in \wh G_a$. 
	We use the leave-one-out technique by decomposing
	\[
	\left| N_i^\T \sw^{*1/2}(\wt J-J^*)(\be_a - \be_{\ell})\right| \le \left| N_i^\T
	\sw^{*1/2}(\wt J-\wt J^{(i)})(\be_a - \be_{\ell}) \right| + \left| N_i^\T \sw^{*1/2}(\wt J^{(i)}-J^*)(\be_a - \be_{\ell})\right|
	\]
	where 
	$ \wt J^{(i)} :=  \wt \sw^{(i)-1} \wt M^{(i)}$ with $\wt \sw^{(i)}$ and $\wt M^{(i)}$ being the same as $\wt\sw$ and $\wt M$ except leaving out the sample $X_i$. Specifically, 
	\begin{equation}\label{def_mu_td_LOO}
		\wt \mu_a^{(i)} = {1\over n_a-1} \sum_{j\in \wh G_a \setminus \{i\}} X_j, \qquad \wt \mu_b^{(i)} = \wt \mu_b, \quad \forall ~ b \in [L]\setminus\{a\},
	\end{equation}
	as well as 
	\begin{equation}\label{def_sw_td_LOO}
		\wt \sw^{(i)} = {1\over n} \left[\sum_{b\in[L]\setminus\{a\}} \sum_{i \in \wh G_b} (X_i -  \wt \mu_b)(X_i -  \wt \mu_b)^\T +  \sum_{j \in \wh G_a\setminus\{i\}} (X_j-  \wt \mu_a^{(i)})(X_j-  \wt \mu_a^{(i)})^\T\right].
	\end{equation}
	Note that  $\cE_{\pi}$ implies $n_a \ge n \pi_a^*/2 \ge C\log n$.
	By construction, the independence between $N_i$ and $\wt J^{(i)}$ yields that for all $t\ge 0$,
	\begin{equation}\label{bd_init_N_sigma_diff_1}
		\PP\left\{
		\left| N_i^\T \sw^{*1/2}(\wt J^{(i)}-J^*)(\be_a - \be_{\ell})\right| \ge  t \left\| \sw^{*1/2}(\wt J^{(i)}-J^*)(\be_a - \be_{\ell})\right\|_2 
		\right\} \le 2\exp \left(
		- t^2/2
		\right).
	\end{equation}
	On the other hand, by $\cE_2$ in \eqref{def_event_N_2}, we have 
	\begin{align}\label{bd_init_N_sigma_diff_2}\nonumber
		\left|N_i^\T \sw^{*1/2}(\wt J-\wt J^{(i)})(\be_a - \be_{\ell})\right| &\le \left\|\sw^{*1/2}(\wt J-\wt J^{(i)})(\be_a - \be_{\ell})\right\|_2 \max_{i\in [n]}\|N_i\|_2\\
		& \le 4\left\|\sw^{*1/2}(\wt J-\wt J^{(i)})(\be_a - \be_{\ell})\right\|_2\sqrt{d + \log n}.
	\end{align}
	We proceed to control $\|\sw^{*1/2}(\wt J-J^*)(\be_a - \be_{\ell})\|_2$ and 
	$\|\sw^{*1/2}(\wt J-\wt J^{(i)})(\be_a - \be_{\ell})\|_2$, which by triangle inequality also gives a bound for $\|\sw^{*1/2}(\wt J^{(i)}-J^*)(\be_a - \be_{\ell})\|_2$. 
	
	Recall $d(\wt \sw, \sw^*)\le 1/2$  from \eqref{bd_eig_sw_td}. Reasoning as \eqref{decomp_J_diff} and using \cref{rate_mu_tilde,bd_sw_td} together with $n\ge dL^2$   give
	\begin{align}\label{rate_sigma_td}\nonumber
		 \|\sw^{*1/2}(\wt J-J^*)(\be_a - \be_{\ell})\|_2  
		&\le  2\|\wt \mu_a - \mu_a^*\|_{\sw^*}+2\|\wt \mu_\ell - \mu_\ell^*\|_{\sw^*} + 2 d(\wt\sw, \sw^*) \sqrt{\Dt_{a\ell}}\\
		&\lesssim \sqrt{d + \log n\over n_a \wedge n_\ell} + \sqrt{ d+ \log n\over n}\sqrt{\Dt_{a\ell}}.
	\end{align}
	In conjunction with $n \pmin \Dtmin \ge Cd\log n$ and $\cE_\pi$, one also has 
	\begin{equation}\label{bd_sigma_td_ell_2}
		\|\sw^{*1/2}\wt J (\be_a - \be_{\ell})\|_2 \le 	\|\sw^{*1/2}(\wt J -J^*)(\be_a - \be_{\ell})\|_2 + \| \sw^{*1/2} J^*(\be_a - \be_{\ell})\|_2 \le 2\sqrt{\Dt_{a\ell}}.
	\end{equation}

	Regarding $\|\sw^{*1/2}(\wt J-\wt J^{(i)})(\be_a - \be_{\ell})\|_2$, by the definition in \eqref{def_sw_td_LOO}, one has 
	\begin{align*}
		d(\wt \sw, \wt\sw^{(i)}) &= {n_a \over n(n_a-1)}\left\|
		\sw^{*-1/2} (X_i - \wt \mu_a)(X_i - \wt \mu_a)^\T \sw^{*-1/2}
		\right\|_\op \hspace{-1cm}\\
		&\le{2n_a \over n(n_a-1)}\left(
		\|N_i\|_2^2 + \|\wt \mu_a-\mu_a^*\|_{\sw^*}^2
		\right)\\
		&\lesssim  {d\log n\over n  }&&\text{by \eqref{rate_mu_tilde} and $\cE_2$}
	\end{align*}
	whence 
	\begin{equation}\label{bd_eig_sw_td_LOO}
		1/2 \le 	\lambda_j\Bigl(\sw^{*-1/2}\wt\sw^{(i)} \sw^{*-1/2}\Bigr)\le 2,\qquad \forall ~ 1\le j\le d.
	\end{equation}
	Meanwhile, by \eqref{def_mu_td_LOO}, we obtain
	\begin{align}\nonumber
		\|\wt \mu_a - \wt \mu_a^{(i)}-\wt \mu_\ell + \wt \mu_\ell^{(i)}\|_{\sw^*} & = {1\over n_a-1}\|X_i - \wt \mu_a \|_{\sw^*}\\\nonumber
		&\le {1\over n_a -1}\left(\left\|
		N_i
		\right\|_2 +\| \wt \mu_a - \mu_a^*\|_{\sw^*}\right)\\
		& \le {2\over n_a-1}\sqrt{d+\log n}\label{bd_diff_mu_LOO} 
	\end{align} 
	where the last step uses $\cE_2$ and \eqref{rate_mu_tilde}.
	We thus have, by \eqref{bd_eig_sw_td_LOO} and  reasoning as \eqref{decomp_J_diff},
	\begin{align}\nonumber
		 \|\sw^{*1/2}(\wt J-\wt J^{(i)})(\be_a - \be_{\ell})\|_2
		 & \le 2\|\wt \mu_a - \wt \mu_a^{(i)}\|_{\sw^*}+2   \left\|
		(\wt \sw^{(i)} - \wt\sw ) \wt J^{(i)}(\be_a - \be_{\ell})
		\right\|_{\sw^*}\\
		&\le {8\over n_a}\sqrt{d+\log n} + 2\left\|
		(\wt \sw^{(i)} - \wt\sw ) \wt J^{(i)}(\be_a - \be_{\ell})
		\right\|_{\sw^*}\label{decomp_sig_diff_LOO}
	\end{align}
	where the last step uses \cref{bd_diff_mu_LOO}. Also observe that 
	\begin{align}\nonumber
		&\left\|
		(\wt \sw^{(i)} - \wt\sw )  \wt J^{(i)}(\be_a - \be_{\ell})
		\right\|_{\sw^*} \\\nonumber
		&= {n_a\over n(n_a-1)}\left\|
		(X_i-\wt \mu_a ) (X_i-\wt \mu_a ) ^\T \wt J^{(i)}(\be_a - \be_{\ell})
		\right\|_{\sw^*}\\\nonumber
		&\le {2\over n}\left\| N_i N_i^\T  \sw^{*1/2} \wt J^{(i)}(\be_a - \be_{\ell})
		\right\|_2 + {2\over n}\left\|
		N_i (\mu_a^*-\wt \mu_a ) ^\T \wt J^{(i)}(\be_a - \be_{\ell})
		\right\|_2 \\\nonumber
		&\quad +{2\over n} \left\|
		(\mu_a^*-\wt \mu_a )  N_i ^\T \sw^{*1/2} \wt J^{(i)}(\be_a - \be_{\ell})
		\right\|_{\sw^*}\\\nonumber
		&\quad  +{2\over n}\left\|
		(\mu_a^*-\wt \mu_a ) (\mu_a^*-\wt \mu_a )^\T\wt J^{(i)}(\be_a - \be_{\ell})
		\right\|_{\sw^*}\hspace{-1cm}\\
		&\lesssim {\sqrt{ d+\log n} \over n}   |N_i^\T  \sw^{*1/2} \wt J^{(i)}(\be_a - \be_{\ell})|  +  {d \log n\over n\sqrt{n_a}}  \|\sw^{*1/2}\wt J^{(i)}(\be_a - \be_{\ell})\|_2.\label{decomp_Sig_sig_diff_LOO}
	\end{align}
	The last step is due to \eqref{rate_mu_tilde} and $\cE_2$.
	To bound the remaining term  $ N_i^\T    \sw^{*1/2} \wt J^{(i)}(\be_a - \be_{\ell})$,  note that  $\wt J^{(i)}$ is independent of $X_i$ hence of $N_i$. We thus have
	\begin{align*}
		\PP\left(
		\left|N_i^\T  \sw^{*1/2} \wt J^{(i)}(\be_a - \be_{\ell})\right| \ge 2\sqrt{2\log n} \left \|\wt \sw^{*1/2} \wt J^{(i)}(\be_a - \be_{\ell})\right\|_2
		\right) \le {2/n^4},
	\end{align*}
	which, by taking the union bounds, holds for all $i\in [N]$ and $a\ne \ell$, with probability $1-n^{-2}$.
	Combining with \eqref{decomp_sig_diff_LOO} and \eqref{decomp_Sig_sig_diff_LOO}
	yields 
	\begin{align*}
		\|\sw^{*1/2} (\wt J -\wt J^{(i)})(\be_a - \be_{\ell})\|_2 & \lesssim {\sqrt{d+\log n}\over n_a} +\left(  {\sqrt{d} \log n  \over n}  +  {d \log n\over n\sqrt{n_a}} \right)   \|\sw^{*1/2} (\wt J -\wt J^{(i)})(\be_a - \be_{\ell})\|_2 \\
		&\quad + \left(  {\sqrt{d} \log n  \over n}  +  {d \log n\over n\sqrt{n_a}} \right)  \|\sw^{*1/2} \wt J(\be_a - \be_{\ell})\|_2.
	\end{align*}
	By $n \ge d \log n$ and \eqref{bd_sigma_td_ell_2}, we conclude 
	\begin{equation}\label{rate_sigma_td_LOO}
		\|\sw^{*1/2} (\wt J -\wt J^{(i)})(\be_a - \be_{\ell})\|_2\lesssim {\sqrt{d+\log n}\over n_a} +\left(  {\sqrt{d} \log n  \over n}  +  {d \log n\over n\sqrt{n_a}} \right)  \sqrt{\Dt_{a\ell}}.
	\end{equation}
	In conjunction with \eqref{rate_sigma_td}, we also have 
	\begin{equation}\label{rate_sigma_td_LOO_true}
 	\|\sw^{*1/2} (\wt J^{(i)} - J^*)(\be_a - \be_{\ell})\|_2 \lesssim  \sqrt{d+\log n\over n_a \wedge n_\ell} + \sqrt{ d+\log n\over n}\sqrt{\Dt_{a\ell}}.
\end{equation}
	
	Finally, plugging \eqref{rate_sigma_td_LOO} and \eqref{rate_sigma_td_LOO_true} into \cref{bd_init_N_sigma_diff_1,bd_init_N_sigma_diff_2}, choosing $t = C\sqrt{\log n}$, taking the union bounds over $i\in [n]$, $a,\ell \in [L]$ and  invoking $\cE_\pi$ give 
	\begin{align}\nonumber
		&{1\over \Dt_{a\ell}}\left| N_i^\T \sw^{*1/2} (\wt J - J^*)(\be_a - \be_{\ell})\right|\\\nonumber  
		&~ \lesssim  ~  
		{d+\log n\over n \pi_a^* \Dt_{a\ell}} +{\sqrt{d(d+\log n)} \log n  \over n \sqrt{\Dt_{a\ell}}}  +  {d \log n\over n}\sqrt{d+\log n\over n\pi_a^*\Dt_{a\ell}}   \\
		&\quad + 
		{\sqrt{\log n}\over \Dt_{a\ell}}\left( {\sqrt{d+\log n}\over n \pi_a^*  } +  
		\sqrt{d+\log n\over n (\pi_a^* \wedge \pi_{\ell}^*) } + \sqrt{ d+\log n\over n }\sqrt{\Dt_{a\ell}} \right)  \label{rate_N_J_td_diff}
	\end{align}
	with probability at least $1-4n^{-1}$.   The proof is complete. 
\end{proof}

\begin{lemma}\label{lem_phi_params}
	Assume    $n\pmin   \ge Cd\log n$, $n\ge dL^2$ and $\Dtmin \ge C\log(1/\pmin)$. 
	On the event $\cE_\pi\cap\cE_{\op}\cap \cE_{\ell_2}\cap \cE_{\op}'$, the following holds for any $\theta$ satisfying  $\phi(\theta)/n \le \Dtmin \pmin (1 + \Dtmin \pmin)$. For all $\ell \in [L]$:
	\begin{align}\label{eq_phi_pi}
		\sum_{\ell=1}^L \left| \wh\bpi_{\ell}(\theta) - \wt \pi_{\ell}\right| &~ \le  ~ {2\over\Dtmin}{ \phi(\theta)\over n},\\\label{eq_phi_mu}
		\|\wh \bmu_{\ell}(\theta)-\wt \mu_{\ell}\|_{\sw^*} &~ \lesssim ~ 
		{1\over \pi_{\ell}^*\sqrt{\Dtmin}} \left(\sqrt{\phi(\theta) \over n} + {\phi(\theta) \over n}\right)
	\end{align}
	and 
	\begin{equation}\label{eq_phi_sw}
		d(\bwhsw(\theta),\wt\Sigma) ~ \lesssim ~\sqrt{\Dtmax \over \Dtmin}\sqrt{\phi(\theta)\over n}+ {\Dtmax \over \Dtmin} {\phi(\theta)\over n}.
	\end{equation}
	Furthermore, if additionally $\theta$ satisfies \eqref{cond_init_unknown} and  $dL\log(q) \le n\pmin \Dtmin$, then  \eqref{eq_phi_mu}  can be improved to: with probability at least $1-n^{-1}$,  for all $\ell \in [L]$, 
	\begin{align}\label{eq_phi_mu_better}
		\|\wh \bmu_{\ell}(\theta)-\wt \mu_{\ell}\|_{\sw^*} &~ \lesssim ~ 	{1\over \pi_{\ell}^*\sqrt{\Dtmin}}\left(  \sqrt{\phi(\theta)\over n  }\sqrt{dL\log(q)\over n} +{\phi(\theta) \over n} \right)  +  \sqrt{\phi(\theta)\over n }\exp(-c'\Dtmin).
	\end{align}
	In this case, with the same probability, we also have 
	\begin{align} \label{eq_phi_sw_better}
		  d(\bwhsw(\theta),\wt\Sigma)  ~ \lesssim  ~  
		   {\phi(\theta)\over n} +  {dL\log(q)\over n} + \exp(-c'\Dtmin) 
	\end{align}
	and 
	\begin{align} \label{eq_phi_sig} \nonumber
	 & {1\over \sqrt{\Dt_{a\ell}}}\|(\bwhsw(\theta)-\wt\sw)\wt J(\be_a - \be_{\ell})\|_{\sw^*}  \\
	 &\qquad\qquad \lesssim ~ \sqrt{{\log n\over \Dtmin}+1}  \left(\sqrt{ dL\log(q)\over n  }+ \exp(-c'\Dtmin)\right)  \sqrt{\phi(\theta)\over n}  +   {\phi(\theta)\over n}. 
	\end{align}
\end{lemma}

\begin{proof}
	To prove \eqref{eq_phi_pi}, it suffices to note
	\begin{align*}\nonumber
		\sum_{\ell=1}^L \left| \wh\bpi_{\ell}(\theta) - \wt \pi_{\ell}\right| & = \sum_{\ell=1}^L\left|
		{1\over n}\sum_{i=1}^n \gamma_{\ell}(X_i;\theta) - {1\over n}\sum_{i\in \wh G_\ell} 1
		\right|\\\nonumber
		&\le  	{1\over n} \sum_{\ell=1}^L 
		\sum_{i\in \wh G_\ell} (1-\gamma_{\ell}(X_i;\theta)) +   {1\over n}  \sum_{\ell=1}^L\sum_{i\notin G_\ell} \gamma_{\ell}(X_i;\theta)\\
		&\le {2 \phi(\theta)\over n\Dtmin}
	\end{align*} 
	where the last inequality uses  \cref{bd_phi_gamma_ell}.
	By $8\phi(\theta)/n \le   \pmin \Dtmin$ and $\cE_\pi$, we thus have 
	\begin{equation}\label{eq_order_pi_phi}
		1/4 \le \wh \bpi_{\ell}(\theta) / \pi_{\ell}^* \le 5/4
	\end{equation}
	
	To show \eqref{eq_phi_mu}, start with the decomposition
	\begin{align}\label{decomp_mu_hat_tilde}\nonumber
		\wh  \bmu_{\ell}(\theta) - \wt \mu_{\ell} &= {1\over n \wh \bpi_{\ell}(\theta)} \sum_{i=1}^n \gamma_{\ell}(X_i;\theta)(X_i-\mu_{\ell}^*) - {1\over n \wt \pi_{\ell}}\sum_{i \in \wh{G}_\ell} (X_i-\mu_{\ell}^*)\\\nonumber
		&= {\wt \pi_{\ell} - \wh\bpi_{\ell}(\theta) \over   \wh\bpi_{\ell}(\theta)} {1\over n_\ell}\sum_{i\in \wh G_\ell} N_i - {1\over n \wh \bpi_{\ell}(\theta)} \sum_{i\in \wh G_\ell}(1-\gamma_{\ell}(X_i;\theta))N_i\\
		&\qquad  +  {1\over n \wh \bpi_{\ell}(\theta)} \sum_{j \notin \wh G_\ell} \gamma_{\ell}(X_j;\theta)(X_j-\mu_{\ell}^*).
	\end{align}
	First, using \eqref{eq_phi_pi}, $\cE_{\ell_2}\cap \cE_{\pi}$ and \eqref{eq_order_pi_phi} gives 
	\begin{align*}
		{|\wt \pi_{\ell} - \wh\bpi_{\ell}(\theta)| \over   \wh\bpi_{\ell}(\theta)} \left\| {1\over n_\ell}\sum_{i\in \wh G_\ell} N_i  \right\|_{\sw^*} \lesssim   { \phi(\theta)\over n\Dtmin  \wh\bpi_{\ell}(\theta)} \sqrt{d+\log n\over n_\ell}\lesssim   {\phi(\theta)\over n\Dtmin   \pi_{\ell}^*} \sqrt{d+\log n\over n \pi_{\ell}^*}.
	\end{align*}
	Second, an application of the Cauchy-Schwarz inequality yields 
	\begin{align}\nonumber
		&{1\over n \wh \bpi_{\ell}(\theta)} \left\|\sum_{i\in \wh G_\ell}(1-\gamma_{\ell}(X_i;\theta))N_i\right\|_{\sw^*}\\\nonumber
		&\le {1\over n \wh \bpi_{\ell}(\theta)}\sqrt{\sum_{i\in \wh G_\ell}(1-\gamma_{\ell}(X_i;\theta))}\sqrt{\| \sum_{i\in \wh G_\ell}(1-\gamma_{\ell}(X_i;\theta))N_iN_i^\T\|_\op} \hspace{-2.5cm}\\\nonumber
		&\le  {1\over n \wh \bpi_{\ell}(\theta)}\sqrt{\phi(\theta)\over \Dtmin}\sqrt{\| \sum_{i\in \wh G_\ell}N_iN_i^\T\|_\op}&&\text{by \eqref{bd_phi_gamma_ell}}\\
		&\lesssim  \sqrt{\phi(\theta)\over n  \pi_{\ell}^*\Dtmin}\sqrt{{d\over n \pi_{\ell}^*}+1} &&\text{by }\cE_{\op}\cap \cE_\pi.\label{bd_gamma_N_phi}
	\end{align} 
	Regarding the remaining term,  
	\begin{align}\nonumber
		&{1\over n \wh \bpi_{\ell}(\theta)} \left\|\sum_{j \notin \wh G_\ell} \gamma_{\ell}(X_j;\theta)(X_j-\mu_{\ell}^*)\right\|_{\sw^*}\\ \nonumber
		&\le {1\over n \wh \bpi_{\ell}(\theta)} \left\|\sum_{j \notin \wh G_\ell} \gamma_{\ell}(X_j;\theta)N_j^*\right\|_2+{1\over n \wh \bpi_{\ell}(\theta)}  \sum_{a\in [L]\setminus\{\ell\}} \sum_{j\in  \wh G_a} \gamma_{\ell}(X_j;\theta)\left\|\mu_a^*-\mu_{\ell}^*\right\|_{\sw^*}\\\ \nonumber
		&\le 	 {1\over n \wh \bpi_{\ell}(\theta)} \sqrt{\sum_{j \notin \wh G_\ell} \gamma_{\ell}(X_j;\theta)}\sqrt{\|\sum_{j \notin \wh G_\ell}\gamma_{\ell}(X_j;\theta) N_jN_j^\T\|_\op} + {1\over n \wh \bpi_{\ell}(\theta)}  {\phi(\theta)\over \sqrt{\Dtmin} } &&\text{by \eqref{bd_phi_gamma_mu_diff}}\\
		&\lesssim  {1\over   \pi_{\ell}^*} \sqrt{\phi(\theta)\over n\Dtmin}\sqrt{n+d\over n} + {\phi(\theta)\over n  \pi_{\ell}^* \sqrt{\Dtmin} }.\label{bd_gamma_X_phi}
	\end{align}
	Combining the last three displays together with $n\pmin\Dtmin  \ge C d\log n$  gives \eqref{eq_phi_mu}, which, in conjunction with $\phi(\theta)/n \le c \pminsq \Dtmin^2$, \eqref{eq_order_mu_tilde} and \eqref{eq_Dt_max_Dt_inf}, further ensures 
	\begin{equation}\label{eq_order_mu_phi}
		\| \wh \bmu_{\ell}(\theta)\|_{\sw^*}  \le \| \wh \bmu_{\ell}(\theta) -\wt \mu_{\ell}\|_{\sw^*}  + \|\wt \mu_{\ell} - \mu_{\ell}^*\|_{\sw^*} +\|\mu_{\ell}^*\|_{\sw^*} \le  3 \sqrt{\Dtmax}.
	\end{equation} 

	To prove \eqref{eq_phi_sw}, observe from \eqref{iter_sw_hat_alter} that 
	$$\bwhsw(\theta)  - \wt\sw = \wh \bM(\theta)\diag(\wh\bpi(\theta))\wh \bM(\theta)^\T - \wt M \diag(\wt \pi) \wt M^\T.$$
	Adding and subtracting terms gives 
	\begin{align*}
		d(\bwhsw(\theta) ,\wt \sw) &\le \left\|\sum_{\ell = 1}^L \wh \bpi_{\ell}(\theta) \sw^{*-1/2}(\wh \bmu_{\ell}(\theta)-\wt \mu_{\ell}) (\wh \bmu_{\ell}(\theta)+\wt \mu_{\ell})^\T \sw^{*-1/2}
		\right\|_\op\\
		&\quad  + \left\|\sum_{\ell = 1}^L (\wh \bpi_{\ell}(\theta) -\wt \pi_{\ell}) \sw^{*-1/2}\wt \mu_{\ell}  \wt \mu_{\ell}^\T \sw^{*-1/2}
		\right\|_\op.
	\end{align*}
	By $\cE_\pi$ and \cref{eq_phi_mu,eq_order_mu_phi,eq_order_mu_tilde}, the first term on the right-hand-side is bounded by  
	\begin{align}\label{eq_phi_mu_sum}
		\sum_{\ell = 1}^L \wh \bpi_{\ell}(\theta) \left\|\wh \bmu_{\ell}(\theta)-\wt \mu_{\ell}\right\|_{\sw^*} \max_{\ell \in[L]} \|\wh \bmu_{\ell}(\theta)+\wt \mu_{\ell}\|_{\sw^*} 
		&\lesssim  \sqrt{\Dtmax \over \Dtmin}\left(\sqrt{\phi (\theta)\over n}+  {\phi(\theta)\over n}\right).
	\end{align}  
	On the other hand, since \eqref{eq_phi_pi} and \eqref{eq_order_mu_phi} imply
	\[
	\left\|\sum_{\ell = 1}^L (\wh \bpi_{\ell}(\theta) -\wt \pi_{\ell}) \sw^{*-1/2}\wt \mu_{\ell}  \wt \mu_{\ell}^\T \sw^{*-1/2}
	\right\|_\op \lesssim {\Dtmax \over \Dtmin}{\phi(\theta)\over n},
	\]
	the proof of \eqref{eq_phi_sw} is complete.\\

	To prove the improved bounds, we start with \eqref{eq_phi_mu_better} and refine the bounds of 
	$\| \sum_{i\in \wh G_\ell}(1-\gamma_{\ell}(X_i;\theta))N_iN_i^\T\|_\op$ and $\|\sum_{j \notin \wh G_\ell}\gamma_{\ell}(X_j;\theta) N_jN_j^\T\|_\op$ 
	in \eqref{bd_gamma_N_phi} and \eqref{bd_gamma_X_phi}, respectively. We use the arguments similar as the proof of \cref{lem_I_fN_samp} and recall $S(\omega)$ in \eqref{def_set_S} for $\omega = \omega(\theta) \in \cB_d(\omega^*)$. The latter is ensured as $\theta$ satisfies \eqref{cond_init_unknown}. 
	By repeating the similar arguments therein, 
	\begin{align*}
		&\Bigl\|\sum_{\ell = 1}^L \sum_{i\in \wh G_\ell}(1-\gamma_{\ell}(X_i;\theta))N_iN_i^\T\Bigr\|_\op \\
		&\le ~ 
		\Bigl \|\sum_{\ell = 1}^L \sum_{i\in \wh G_\ell \cap S(\omega)} N_iN_i^\T\Bigr\|_\op +  \Bigl\| \sum_{\ell = 1}^L\sum_{i\in \wh G_\ell \cap [S(\omega)]^c}(1-\gamma_{\ell}(X_i;\theta))N_iN_i^\T\Bigr\|_\op\\
		&\le  ~ 
		\Bigl \| \sum_{i\in S(\omega)} N_iN_i^\T\Bigr\|_\op  +  \Bigl\| \sum_{\ell = 1}^L\sum_{i\in \wh G_\ell} N_iN_i^\T\Bigr\|_\op ~  {1\over \pmin} \exp(-c'\Dtmin)   &&\text{by \eqref{bd_I_ell_ell}}.
	\end{align*}
	Invoking \cref{lem_N_op_norm}, $d\le n$, \eqref{bd_S_card} and $\cE_\pi$ yields that: with probability  $1-n^{-dL}-e^{-n}$, 
	\begin{align}\label{bd_gamma_NN_op}
		\left\| \sum_{\ell = 1}^L\sum_{i\in \wh G_\ell}(1-\gamma_{\ell}(X_i;\theta))N_iN_i^\T\right\|_\op 
		&\lesssim   dL\log(q) + {nL \over \pmin}\exp(-c\Dtmin).
	\end{align}
	By \cref{fact_gamma_in_notin}, the same bound also holds for $\|\sum_{\ell = 1}^L\sum_{j \notin \wh G_\ell}\gamma_{\ell}(X_j;\theta) N_jN_j^\T\|_\op$. Plugging the above two bounds into \eqref{bd_gamma_N_phi} and \eqref{bd_gamma_X_phi} and using  $\Dtmin \ge C\log(1/\pmin)$  gives 
	\begin{align*}
		&{1\over n \wh \bpi_{\ell}(\theta)} \left\|\sum_{i\in \wh G_\ell}(1-\gamma_{\ell}(X_i;\theta))N_i\right\|_{\sw^*}  +{1\over n \wh \bpi_{\ell}(\theta)} \left\|\sum_{j \notin \wh G_\ell} \gamma_{\ell}(X_j;\theta)(X_j-\mu_{\ell}^*)\right\|_{\sw^*}\\
		&\lesssim  \sqrt{\phi(\theta)\over n  \pi_{\ell}^*\Dtmin}\sqrt{dL\log(q)\over n\pi_{\ell}^*} +  \sqrt{\phi(\theta)\over n }\exp(-c'\Dtmin) +{\phi(\theta)\over n  \pi_{\ell}^* \sqrt{\Dtmin} },
	\end{align*} 
	which implies the improved rate in \eqref{eq_phi_mu_better}.   
	
	To prove \eqref{eq_phi_sw_better}, we have the decomposition
	\begin{align}\label{decomp_Sig_diff_LOO}
		\bwhsw(\theta) - \wt \sw &= \rI + \rII + \rIII 
	\end{align}
	where 
	\begin{align*}
		\rI &= {1\over n}\sum_{\ell = 1}^L  \sum_{j \notin \wh G_\ell}   \gamma_{\ell}(X_j;\theta)(X_j-\mu_{\ell}^*)(X_j-\mu_{\ell}^*)^\T\\  
		\rII &= -  {1\over n}\sum_{\ell = 1}^L \sum_{i \in \wh G_\ell} \left(1- \gamma_{\ell}(X_i;\theta)\right)\sw^{*1/2}N_iN_i^\T \sw^{*1/2}\\
		\rIII&=  \sum_{\ell = 1}^L \wh \bpi_{\ell}(\theta)(\wh \bmu_{\ell}(\theta)-\mu_{\ell}^*)(\wh \bmu_{\ell}(\theta)-\mu_{\ell}^*)^\T - \sum_{\ell = 1}^L \wt \pi_{\ell} (\wt \mu_{\ell} -\mu_{\ell}^*)(\wt \mu_{\ell} -\mu_{\ell}^*)^\T.
	\end{align*}
	The term $\|\sw^{*-1/2}   \rII  \sw^{*-1/2}\|_\op$ is bounded in  \eqref{bd_gamma_NN_op} while, by Cauchy-Schwarz inequality, 
	\begin{align*}
		\|\sw^{*-1/2}   ~ \rI ~   \sw^{*-1/2}\|_\op &\le \left\| {2\over n}\sum_{\ell = 1}^L  \sum_{j \notin \wh G_\ell}   \gamma_{\ell}(X_j;\theta) N_j N_j^\T \right\|_\op + {2\over n}\sum_{\ell = 1}^L  \sum_{a \in [L]\setminus\{\ell\}} \sum_{j \in \wh G_a}   \gamma_{\ell}(X_j;\theta) \Dt_{a\ell}\\
		&= 2\|\sw^{*-1/2}   \rII  \sw^{*-1/2}\|_\op + {2\phi(\theta)\over n} 
	\end{align*}
	using  \cref{fact_gamma_in_notin} and \eqref{def_phi_theta}.  
	Regarding $\rIII$, by adding and subtracting terms, 
	\begin{align*}
		\left\| \sw^{*-1/2} ~ \rIII ~    \sw^{*-1/2} \right\|_\op &\le  \left\|\sum_{\ell = 1}^L \wh \bpi_{\ell}(\theta)\sw^{*-1/2} (\wh \bmu_{\ell}(\theta)-\wt\mu_{\ell})(\wh \bmu_{\ell}(\theta)-\wt\mu_{\ell})^\T\sw^{*-1/2} \right\|_\op   \\
		&\quad  + 2\left\|\sum_{\ell = 1}^L \wh \bpi_{\ell}(\theta)\sw^{*-1/2}(\wh \bmu_{\ell}(\theta)-\wt\mu_{\ell})(\wt\mu_{\ell}-\mu_{\ell}^*)^\T\sw^{*-1/2}\right\|_\op  \\ 
		&\quad +  \left\|\sum_{\ell = 1}^L (\wh \bpi_{\ell}(\theta)-\wt \pi_{\ell}) \sw^{*-1/2} (\wt \mu_{\ell} -\mu_{\ell}^*)(\wt \mu_{\ell} -\mu_{\ell}^*)^\T\sw^{*-1/2}\right\|_\op\\
		&\le \sum_{\ell = 1}^L \wh \bpi_{\ell}(\theta)  \left\| \wh \bmu_{\ell}(\theta)-\wt\mu_{\ell}\right\|_{\sw^*}\max_{\ell \in[L]}\left(
		\| \wh \bmu_{\ell}(\theta)-\wt\mu_{\ell}\|_{\sw^*} + 2\|\wt \mu_{\ell} -\mu_{\ell}^*\|_{\sw^*}
		\right)\\
		&\quad +  \sum_{\ell = 1}^L\left|  \wh \bpi_{\ell}(\theta)-\wt \pi_{\ell}\right| \max_{\ell \in[L]}  \|\wt \mu_{\ell} -\mu_{\ell}^*\|_{\sw^*}^2.
	\end{align*}
	By invoking the improved bound in \cref{eq_phi_mu_better} and using \cref{rate_mu_tilde,eq_phi_pi} together with the event $\cE_{\ell_2}\cap \cE_\pi$, after some algebra, we obtain that
	\begin{align}\label{bd_sw_diff_III}\nonumber
		&\left\| \sw^{*-1/2} ~ \rIII ~    \sw^{*-1/2} \right\|_\op\\ \nonumber
		&~ \lesssim \left( {dL\log(q)\over n \pmin \Dtmin} + \sqrt{d+\log n\over n\pmin \Dtmin} + \exp(-2c'\Dtmin)  + {\phi(\theta)\over n  \pmin\Dtmin} \right){\phi(\theta)\over n}   \\
		&\quad 	+  \left(\sqrt{dL\log(q)\over n} +\exp(-c'\Dtmin) \right) \sqrt{d+\log n\over n\pmin}  \sqrt{\phi(\theta)\over n}. 
	\end{align} 
	Combining with \eqref{bd_gamma_NN_op} gives 
	\begin{align*}
		d(\bwhsw(\theta),\wt\Sigma) &~ \lesssim ~\left( {dL\log(q)\over n \pmin \Dtmin}   + {\phi(\theta)\over n  \pmin\Dtmin} + 1\right){\phi(\theta)\over n}   \\
		&\qquad 	+   \sqrt{dL\log(q)\over n}  \sqrt{d+\log n\over n\pmin}  \sqrt{\phi(\theta)\over n} +  {dL\log(q)\over n} + \exp(-c'\Dtmin)  
	\end{align*}
	which,  under  $\phi(\theta) + dL\log(q) \le n\pmin \Dtmin$ and $d\log n\le n\pmin$,  simplifies to 
	\eqref{eq_phi_sw_better}.
	
	Finally, we prove \eqref{eq_phi_sig} by using the decomposition \eqref{decomp_Sig_diff_LOO} and \cref{fact_gamma_in_notin}
	\begin{align}\label{decomp_init_Sigma_sig}\nonumber
		\|(\bwhsw(\theta)-\wt\sw)\wt J(\be_a - \be_{\ell})\|_{\sw^*} &\le {1\over n}\| 
		(\rI + \rII + \rIII )\wt J(\be_a - \be_{\ell})\|_{\sw^*} \\\nonumber
		&\le {1\over n}\left\|
		\sum_{k = 1}^L  \sum_{i \notin \wh G_b} \gamma_k(X_i;\theta)(X_i-\mu_{k}^*)(X_i-\mu_{k}^*)^\T\wt J(\be_a - \be_{\ell})\right\|_{\sw^*}\\\nonumber
		&\quad +  {1\over n}\left\|
		\sum_{k = 1}^L \sum_{i  \in \wh G_k}(1- \gamma_k(X_i;\theta))  N_i N_i^\T  \sw^{*1/2}\wt J(\be_a - \be_{\ell})\right\|_2\\
		&\quad +  {1\over n}\left\| \sw^{*-1/2} ~ \rIII ~    \sw^{*-1/2} \right\|_\op \|\sw^{*1/2}\wt J(\be_a - \be_{\ell})\|_2.
	\end{align}
	Note that the first term is bounded by $R_1 + R_2 + R_3 + R_4$ with 
	\begin{align*}
		R_1  &= {1\over n}\left\|
		\sum_{k = 1}^L  \sum_{i \notin \wh G_k} \gamma_k(X_i;\theta) N_i N_i^\T  \sw^{*1/2}\wt J(\be_a - \be_{\ell})\right\|_2,\\
		R_2 &=   {1\over n}\left\|
		\sum_{k = 1}^L\sum_{b \in [L]\setminus\{k\}} \sum_{i \in \wh G_b} \gamma_k(X_i;\theta)N_i (\mu_b^*-\mu_{k}^*)^\T  \wt J(\be_a - \be_{\ell})\right\|_2,\\
		R_3 &=   {1\over n}\left\|
		\sum_{k = 1}^L  \sum_{b \in [L]\setminus\{k\}} \sum_{i \in \wh G_b}  \gamma_k(X_i;\theta)(\mu_b^*-\mu_{k}^*)N_i^\T \sw^{*1/2}\wt J(\be_a - \be_{\ell})\right\|_{\sw^*},\\	
		R_4 &= {1\over n} 
		\sum_{k = 1}^L\sum_{b \in [L]\setminus\{k\}} \sum_{i \in \wh G_b} \gamma_k(X_i;\theta) \Delta_{bk}  \|\sw^{*1/2}\wt J(\be_a - \be_{\ell})\|_2.
	\end{align*}
	By Cauchy-Schwarz inequality,  
	\begin{align*}
		R_1 &\le {1\over n}\left\|
		\sum_{k = 1}^L \sum_{i  \notin \wh G_k} \gamma_k(X_i;\theta)  N_i N_i^\T   \right\|_\op^{1/2}\left(\sum_{k = 1}^L \sum_{i  \notin \wh G_k} \gamma_k(X_i;\theta) (N_i^\T \sw^{*1/2}\wt J(\be_a - \be_{\ell}))^2\right)^{1/2}\\
		&\lesssim  \left(\sqrt{ dL\log(q) \over n}+ \exp(-c\Dtmin)\right)  \sqrt{\phi(\theta)\over n\Dtmin} \max_{i \in [n]} \left|N_i^\T  \sw^{*1/2}\wt J(\be_a - \be_{\ell})\right|   
	\end{align*}
	with probability at least $1-n^{-dL}-e^{-n}-n^{-2}$. The second step is due to \eqref{bd_gamma_NN_op} \& \eqref{bd_phi_gamma_ell}. Furthermore, using $N_i^\T  \sw^{*1/2}J^*(\be_a - \be_{\ell}) \sim \cN(0, \Dt_{a\ell})$ and \eqref{rate_N_J_td_diff} gives 
	\begin{align*}
		\max_{i \in [n]}  \left|N_i^\T \sw^{*1/2}\wt J(\be_a - \be_{\ell})\right|   & \le 	\max_{i \in [n]} \left(  |N_i^\T  \sw^{*1/2}J^*(\be_a - \be_{\ell})| +  |N_i^\T\sw^{*1/2}(\wt J-J^*)(\be_a - \be_{\ell})|\right)\\
		& \lesssim   \sqrt{\Dt_{a\ell} +\log n}
	\end{align*} 
	with probability at least $1-4n^{-1}$.
	By \cref{fact_gamma_in_notin}, the same bound of $R_1$ also holds for the second term in  \eqref{decomp_init_Sigma_sig}. 
	
	For $R_2$, by similar arguments as well as \eqref{bd_sigma_td_ell_2}, one has 
	\begin{align*}
		R_2 &\le     {1\over n}\Bigl\|
		\sum_{k = 1}^L\sum_{i \notin \wh G_k} \gamma_k(X_i;\theta)N_i  N_i^\T \Bigr\|_\op^{1/2} \Bigl( 
		\sum_{k = 1}^L\sum_{b \in [L]\setminus\{k\}} \sum_{i \in \wh G_b} \gamma_k(X_i;\theta)  \Dt_{bk}\Bigr)^{1/2} \| \sw^{*1/2}\wt J(\be_a - \be_{\ell})\|_2\\
		&\lesssim 
		\left(\sqrt{ dL\log(q) \over n}+ \exp(-c\Dtmin)\right)  \sqrt{\phi(\theta)\over n}\sqrt{\Dt_{a\ell}}
	\end{align*}
	with probability at least $1-n^{-dL}-e^{-n}$. It is easy to see that the same bound  also holds for $R_3$. Regarding $R_4$,  observe that 
	\begin{align*}
		R_4   \overset{\eqref{def_phi_theta}}{=} {\phi(\theta)\over n} \| \sw^{*1/2}\wt J(\be_a - \be_{\ell})\|_2 \overset{\eqref{bd_sigma_td_ell_2}}{\le}  2\sqrt{\Dt_{a\ell}}{\phi(\theta)\over n}.
	\end{align*} 
	Finally, invoking \eqref{bd_sw_diff_III} bounds the term 
	$
	 n^{-1}\| \sw^{*-1/2} ~ \rIII ~    \sw^{*-1/2} \|_\op \| \sw^{*1/2}\wt J(\be_a - \be_{\ell})\|_2
	$ so that  collecting the previous bounds yields 
	\begin{align*}
		{1\over \sqrt{\Dt_{a\ell}}} 	\|(\bwhsw(\theta)-\wt\sw)\wt J(\be_a - \be_{\ell})\|_{\sw^*}  &\lesssim 
		\sqrt{ {\log n\over \Dtmin}+1}  \left(\sqrt{ dL\log(q) \over n}+ \exp(-c\Dtmin)\right)  \sqrt{\phi(\theta)\over n}   \\
		&\quad +  \left( {dL\log(q)\over n \pmin \Dtmin}   + {\phi(\theta)\over n  \pmin\Dtmin} + 1\right){\phi(\theta)\over n}  
	\end{align*}
	thereby completing the proof of \eqref{eq_phi_sig}. 
\end{proof}

\subsection{Proof of \cref{thm_param_phi}}\label{app_sec_proof_thm_param_phi}

\begin{proof}
	From \cref{thm_conv_phi}, by induction and \eqref{cond_init_phi}, we know that for all $t\ge 1$,
	\begin{align}\label{bd_phi_induction}\nonumber 
		 \phi(\wh\theta^{(t)}) &\le n \exp(-c\Dtmin) \left(1 + 1/2 + \cdots + 1/2^{t-1}\right) + 2^{-t} \phi(\wh\theta^{(0)})\\
		 &\le 2n \exp(-c\Dtmin) + 2^{-t}.
	\end{align}
	Since $\wh \theta$ satisfies \eqref{cond_init_unknown}, invoking \eqref{eq_phi_mu_better}  gives that with probability $1-n^{-1}-\exp(-c\Dtmin)$, for all $t\ge 1$,
	\begin{align*} 
	\|\wh \mu_{\ell}^{(t+1)}-\wt \mu_{\ell}\|_{\sw^*}  &= 	\|\wh \bmu_{\ell}(\wh\theta^{(t)})-\wt \mu_{\ell}\|_{\sw^*}\\ 
		&  \lesssim 	{1\over \pi_{\ell}^*\sqrt{\Dtmin}}\left(  \sqrt{\phi(\wh\theta^{(t)})\over n  }\sqrt{dL\log(q)\over n} +{\phi(\wh\theta^{(t)}) \over n} \right)  +  \sqrt{\phi(\wh\theta^{(t)})\over n }\exp(-c'\Dtmin)\\ 
		&\lesssim \exp(-c\Dtmin) + {2^{-t}\over \pmin}.
	\end{align*}
	Together with \eqref{rate_mu_tilde}, we conclude that with the same probability,
	\begin{align*}
		\|\wh \mu_{\ell}^{(t+1)} - \mu_{\ell}^*\|_{\sw^*} &\le 	\|\wh \mu_{\ell}^{(t+1)} - \mu_{\ell}^*\|_{\sw^*} + 	\|\wt \mu_{\ell}-\mu_{\ell}^*\|_{\sw^*} \\
		& \lesssim \sqrt{d+\log n\over n\pi_{\ell}^*} + \exp(-c\Dtmin),\qquad \text{ for $t \ge C \log n$.}
	\end{align*}
	 Similarly, using \eqref{eq_phi_pi} and $\cE_\pi$ yields: for $t\ge C\log n$,
	\begin{align*}
			d(\wh \pi^{(t+1)},\pi^*) &\le d(\wh \bpi(\wh \theta^{(t)}),\wt \pi) + d(\wt \pi,  \pi^*) \lesssim {\phi(\wh \theta^{(t)})\over n\Dtmin \pmin} + \sqrt{\log n\over n\pmin} \lesssim \exp(-c'\Dtmin) + \sqrt{\log n\over n\pmin}.
	\end{align*}
	Finally,  regarding $\bwhsw^{(t+1)}$, using \eqref{bd_sw_td} and \eqref{eq_phi_sw_better} as well as $dL^2\log(q)\log n \le n$ yields 
	\begin{align*}  
		d(\bwhsw(\wh \theta^{(t)}),\sw^*)   ~ &\lesssim  ~   \sqrt{ d+\log n\over n} + 
		{\phi(\wh \theta^{(t)})\over n} +  {dL\log(q)\over n} + \exp(-c'\Dtmin) \\
		&\lesssim ~  \sqrt{ d+\log n\over n} + \exp(-c'\Dtmin) 
	\end{align*}
	for all $t\ge C\log n$. This completes the proof. 
\end{proof}

\subsection{Proof of \cref{thm_clustering}}\label{app_sec_proof_thm_clustering}
 \begin{proof}
    We first prove for $\Dtmax \ge \log n$.
 	For any $\theta$ and any $\ell \in [L]$ with $i\in \wh G_\ell$, we know 
 	$$
 	\left\{g(X_i;\theta) \ne \ell\right\} \subseteq \left\{ \gamma_\ell(X_i;\theta) < {1\over 2}\right\} = \left\{ \sum_{k \in [L]\setminus\{\ell\}}\gamma_k(X_i;\theta) \ge {1\over 2}\right\}.
 	$$
 	It follows that 
 	\begin{align}\nonumber
 		\ell(\theta) = {1\over n} \sum_{\ell = 1}^L \sum_{i \in \wh{G}_\ell}  1\{ g(X_i;\theta) \ne \ell\} &\le {2\over n}\sum_{\ell = 1}^L \sum_{i \in \wh{G}_\ell} \sum_{k \in [L]\setminus\{\ell\}}\gamma_k(X_i;\theta) \\
 		&= {2\over n}\sum_{k=1}^L \sum_{i\notin \wh G_k} \gamma_k(X_i;\theta) ~ \overset{\eqref{def_phi_theta}}{\le} {2\phi(\theta)\over n\Dtmin}.\label{eq_ell_phi}
 	\end{align}
 	Using \eqref{bd_phi_induction} gives that with probability $1-n^{-1}-\exp(-c\Dtmin)$,  for all $t\ge 1$,
 	\begin{align*}
 		\ell(\wh\theta^{(t)} )  \le {2\over n \Dtmin} \left(2n \exp(-c\Dtmin) + 2^{-t}\right).
 	\end{align*}
 	Since $\ell(\wh \theta^{(t)})$ only takes values from $\{0,1/n,\ldots, 1\}$, the claim thus follows. 

    For $\Dtmax \le \log n$, by repeating the same arguments in the proof of \cref{thm_conv_phi} except adding and subtracting $J^*$ directly instead of $\wt J^*$, one can obtain that, with probability at least $1-2n^{-1}-\exp(-c\Dtmin)$, for all $t\ge 1$,
    \begin{align*}
        \phi(\wh \theta^{(t)}) \le n\exp(-c\Dtmin) + CnL\left\{
            {1\over \Dtmin} \left[d(\wh M^{(t)},  M^*)\right]^2 + \left[d(\whsw^{(t)},\sw^*)\right]^2
        \right\}.
    \end{align*}
    The claim then follows by invoking \cref{thm_EM_samp} and \eqref{eq_ell_phi}.
 \end{proof}

 \section{Proof of \cref{thm_lower_bound}: the minimax lower bounds}\label{app_sec_proof_lower_bound}
 	
 	We separate the proof into $d \ge L-1$ and $d< L-1$. 
 
 	\begin{proof}[Proof of $d\ge L-1$]
 	It suffices to prove 
 	\begin{align}\label{lb_mu}
 		&\inf_{\wh \theta} \sup_{\theta \in \Theta(\alpha, \delta)} \PP_{\theta}\left\{ 
 		d(\wh M, M) \ge c_1 \sqrt{d \over n \alpha} 
 		\right\} \ge (1+c_0)/2,\\\label{lb_sw}
 		&\inf_{\wh \theta} \sup_{\theta \in \Theta(\alpha, \delta)} \PP_{\theta}\left\{ 
 		d(\wh \sw, \sw) \ge c_1\sqrt{d \over n}
 		\right\} \ge (1+c_0)/2
 	\end{align}
 	which jointly imply the claim.   We first note that when $d< 16$, the lower bounds of $c/\sqrt{n\alpha}$ and $c/\sqrt{n}$ can be easily proved by applying Le Cam's two-point argument; see, \citet[Theorem 2.2 \& Eq (2.9)]{Tsybakov09}. We thus focus on $d\ge 16$ in the following.
 	
 	We first prove \eqref{lb_mu} by constructing a set of  hypotheses. Let $\theta^{(0)}  =  (\pi, \mu_1, \mu_2, \ldots, \mu_L, \bI_d)$ with $\alpha = \pi_1 = \min_{\ell\in [L]} \pi_\ell$  and 
 	$$
 	\|\mu_k - \mu_{\ell}\|_2 = 	\sqrt{2\delta},\quad \forall ~ k\ne \ell.
 	$$ 
 	Note that the existence of such $\mu_\ell$'s is guaranteed by $d\ge L-1$.
 	For all $0\le i\le N$, let 
 	$\theta^{(i)} = (\pi, \mu_1 + \eps \beta^{(i)}, \mu_2, \ldots, \mu_L, \bI_d)$ with  $\beta^{(0)}=0_d$ and $\beta^{(i)} \in \{0,1\}^d$ for $i\ge 1$, which  according to the Varshamov-Gilbert bound in \citet[Lemma 2.9]{Tsybakov09}, satisfy  
 	\begin{equation}\label{hd_beta}
 		\|\beta^{(i)}-\beta^{(j)}\|_1 \ge d/8,\qquad   \forall~  0\le i<j\le N
 	\end{equation}
 	and $\log N \ge (d/8)\log 2$. For some constant $c\in (0,1/4)$ to be determined later, by   
 	choosing 
 	\begin{equation}\label{choice_eps}
 		\eps = \sqrt{ c\log 2 \over 8n \alpha} \le  (\sqrt{2}-1)\sqrt{\delta \over d},
 	\end{equation}
 	it is easy to see that $\theta^{(i)} \in \Theta(\alpha, \delta, A)$ for all $0\le i\le N$. Indeed, we note that  	for any $k, \ell \ge 2$ with $k\ne \ell$,
 	$
 	\|\mu_k - \mu_{\ell}\|_2^2 = 2\delta
 	$
 	by construction,  and 
 	\begin{align}\label{lb_mu_1_diff}
 		&\|\mu_1^{(i)} - \mu_{\ell} \|_2 \ge \sqrt{2\delta} - \eps \|\beta^{(i)}\|_2 \ge\sqrt{2\delta} - \sqrt{c d \log 2\over 8n\alpha} \ge \sqrt{\delta},\\\label{ub_mu_1_diff}
 		&\|\mu_1^{(i)} - \mu_{\ell} \|_2 \le \sqrt{2\delta} + \eps \|\beta^{(i)}\|_2 \le\sqrt{2\delta} +\sqrt{c d \log 2\over 8n\alpha} \le 2 \sqrt{\delta}.
 	\end{align} 
 	
 	We proceed to invoke \citet[Theorem 2.5]{Tsybakov09} to prove  \eqref{lb_mu}, which   requires to verify  
 	\begin{enumerate}
 		\item[(i)] $\KL(\PP_{\theta^{(i)}}, \PP_{\theta^{(0)}}) \le  c' \log N$ for some $0<c'<1/8$;
 		\item[(ii)] $d(M^{(i)}, M^{(j)}) = \|\mu_1^{(i)} - \mu_1^{(j)}\|_2 \ge 2c_1 \sqrt{d/n\alpha}$.
 	\end{enumerate}
 	Here $\PP_{\theta^{(i)}}$ is the distribution of $n$ i.i.d. samples from \eqref{model} under the parametrization $\theta^{(i)}$.
 	For (i), we find that 
 	\begin{align*}
 		\KL(\PP_{\theta^{(i)}}, \PP_{\theta^{(0)}}) &= n ~ 	\KL\left(
 		\sum_{\ell=1}^L \pi_\ell \cN_d(\mu_{\ell}^{(i)}, \bI_d), \sum_{\ell=1}^L \pi_\ell \cN_d(\mu_{\ell}^{(0)}, \bI_d)
 		\right) \\
 		&\le n \pi_1 ~ 	\KL\left(
 		\cN_d(\mu_1^{(i)}, \bI_d),  \cN_d(\mu_1^{(0)}, \bI_d)
 		\right)\\
 		&= {n\alpha \over 2}\eps^2 \|\beta^{(i)}\|_2^2.
 	\end{align*}
 	By $\|\beta^{(i)}\|_2^2\le d$, $\log N \ge (d/8)\log 2$ and  \eqref{choice_eps} with $c'=c/2$,
 	we obtain (i). Since (ii) follows as
 	\[
 	\|\mu_1^{(i)} - \mu_1^{(j)}\|_2 = \eps	\|\beta^{(i)}-\beta^{(j)}\|_2 = \eps	\|\beta^{(i)}-\beta^{(j)}\|_1^{1/2} \overset{\eqref{hd_beta}}{\ge}   \sqrt{ {c\log 2 \over 64}{d\over n\alpha}},
 	\]
 	invoking \citet[Theorem 2.5]{Tsybakov09}  proves \eqref{lb_mu}   for  
 	$
 	c_1 = \sqrt{c\log 2} / 16
 	$ 
 	and
 	$$
 	{\sqrt{N}\over 1+\sqrt{N}}\left(1-c-\sqrt{c \over \log (N)}\right) \ge  (1+c_0)/2
 	$$
 	by using  $d\ge 16$ and  choosing $c$ sufficiently small.\\
 	
 	To establish \eqref{lb_sw}, it suffices to prove
 	\begin{equation}\label{lb_sw_target}
 		\inf_{\wh \theta} \sup_{\theta \in \Theta(\alpha, \delta, A), \sw \preceq {3\over 2}} \PP_{\theta}\left\{ 
 		\|\wh  \sw - \sw\|_\op \ge {3c_1 \over 2}\sqrt{d \over n}
 		\right\} \ge (1+c_0)/2.
 	\end{equation}
 	For simplicity, we assume $d$ is even as the same argument can be proved for odd $d$, with $d/2$ replaced by $(d+1)/2$. Consider the following hypotheses:
 	$\theta^{(i)} = (\pi, \mu_1,\ldots, \mu_L, \sw^{(i)})$ with $\pi_{\min} \ge \alpha$, $\|\mu_k-\mu_{\ell}\|_2^2 = 3\delta/2$ for all $k\ne \ell$,  and 
 	\[
 	\sw^{(i)} =  \bI_d +  {\tau\over \sqrt{nd}} \sum_{m=1}^{d/2} \beta_m^{(i)}  B(m)
 	\]
 	with $\beta_m^{(i)} \in \{0,1\}$ for $i = 0,1,\ldots, N$, $1\leq m \leq d/2$ and 
 	some $\tau>0$ to be chosen later.
 	Here the matrix $B(m) = (B_{ij}(m))_{d\times d}$ is defined as: for $i,j\in [d]$, 
 	\[
 	B_{ij}(m) = 1\{i = m \text{ and } m+1\le j\le d, \text{ or } j =m \text{ and } m+1 \le i\le d\}.
 	\]
 	Then by using the inequality $\|A\|_\op \le \max_{\ell\in [d]} \|A_{\ell\cdot}\|_1$ for any symmetric matrix $A$ and Weyl's inequality, we obtain
 	\begin{equation}\label{bd_op_sw_j}
 		\lambda_1(\sw^{(i)}) \le 1 + {\tau \over \sqrt {nd}} \max_{\ell\in [d]}  \sum_{m=1}^{d/2} \beta_m^{(i)}\|B(m)_{\ell\cdot}\|_1 \le 1 + \tau \sqrt{d\over n} \le 3/2
 	\end{equation}
 	provided that $\tau \le 1/2$, and similarly, we can also deduce that $\lambda_p(\sw^{(i)}) \ge 1/2$. We thus obtain 
 	$\min_{k\ne \ell} \|\mu_k - \mu_{\ell}\|_{\sw^{(i)}}^2 \ge \min_{k\ne \ell} \|\mu_k - \mu_{\ell}\|_2^2 / \lambda_1(\sw^{(i)}) \ge \delta$ as well as 
 	$\max_{k,\ell} \|\mu_k - \mu_{\ell}\|_{\sw^{(i)}}^2 \le \max_{k,\ell}  \|\mu_k - \mu_{\ell}\|_2^2 / \lambda_p(\sw^{(i)}) \le 3\delta$.
 	Therefore, $\theta^{(i)} \in \Theta(\alpha, \delta,A)$ for all $0\le i\le N$. 
 	
 	By using the Varshamov-Gilbert bound again, one has $\|\beta^{(i)}-\beta^{(j)}\|_1 \ge d / 16$ for any $i\ne j$ and $\log N \ge (d/16)\log 2$. This  also gives 
 	\begin{align}\label{lb_sw_ij_diff}\nonumber
 		\|	\sw^{(i)}  - \sw^{(j)}\|_\op^2 &= \sup_{v \ne 0} {\|(\sw^{(i)}  - \sw^{(j)}) v\|_2^2 \over \|v\|_2^2}\\\nonumber
 		&\ge   {\tau^2 \over  nd} \sum_{m=1}^{d/2}|\beta_m^{(i)} -\beta_m^{(j)}|  {  \|B(m) v\|_2^2 \over \|v\|_2^2} &&\text{by taking }v_i = 1\{d/2 \le i\le d\}\\\nonumber
 		&=   {\tau^2 \over  nd} \|\beta^{(i)} -\beta^{(j)}\|_1  {d\over 2}\\
 		&\ge  {\tau^2 d\over 32n}.
 	\end{align} 
 	On the other hand, for any $1\le i\le N$, we have 
 	\begin{align*}
 		\KL(\PP_{\theta^{(i)}}, \PP_{\theta^{(0)}}) &\le n \sum_{\ell=1}^L \pi_\ell  ~ 	\KL\left( \cN_d(\mu_{\ell}, \sw^{(i)}), \cN_d(\mu_{\ell}, \bI_d)
 		\right)\\
 		&  = {n \over 2} \left(
 		\tr(\sw^{(i)}) -d - \log{\det(\sw^{(i)})}
 		\right)\\
 		&= {n\over 2}\left[ 
 		\tr(\sw^{(i)} - \bI_d) - \log \det(\bI_d + \sw^{(i)} - \bI_d)
 		\right].
 	\end{align*}
 	Recall from \eqref{bd_op_sw_j} that the eigenvalues of $ \sw^{(i)} - \bI_d$ are bounded within $[-1/2, 1/2]$. Taylor expansion gives that 
 	\[
 	\log \det(\bI_d + \sw^{(i)} - \bI_d) \ge \tr(
 	\sw^{(i)} - \bI_d)  - 2 \| \sw^{(i)} - \bI_d\|_F^2
 	\]
 	so that, by choosing $\tau^2 \le c \log 2/16$,
 	\[
 	\KL(\PP_{\theta^{(i)}}, \PP_{\theta^{(0)}}) \le n \| \sw^{(i)} - \bI_d\|_F^2  \le n { \tau^2\over  nd} \sum_{m=1}^{d/2} \beta_m^{(i)} 2d  \le   \tau^2 d\le c  \log N.
 	\]
 	In view of \eqref{lb_sw_ij_diff}, choosing $c$ sufficiently small, invoking \citet[Theorem 2.5]{Tsybakov09}  proves \eqref{lb_sw_target}  for  
 	$
 	c_1 =  \tau / (12\sqrt{2})
 	$ 
 	and $d\ge 16$. This completes the proof. 
 \end{proof}

 \begin{proof}[Proof of $d<L-1$]
 	The proof follows the same arguments except that: for proving \eqref{lb_mu}, due to $d < L-1$, one can only construct $\mu_1,\ldots, \mu_L$ such that  
 	\[ 
 	\sqrt{2\delta} \le \min_{k\ne \ell} \|\mu_k - \mu_{\ell}\|_2\le \max_{k,\ell} \|\mu_k - \mu_{\ell}\|_2 \le  L^{1/d}\sqrt{2\delta}\le (\sqrt{A}-1) \sqrt{\delta}.
 	\]  
 	The existence of such $\mu_{\ell}$'s is guaranteed by  the standard bounds of the  packing number of Euclidean balls.\footnote{For any $\epsilon\in (0,1]$, the $\epsilon$-packing number of $\cB_2(\RR^d; 1)$, the $\ell_2$-unit ball of $\RR^d$ lies between $(1/\epsilon)^d$ and $(3/\epsilon)^d$. Recall that a set $S \subseteq \cB_2(\RR^d; 1)$ is an $\epsilon$-packing  of $\cB_2(\RR^d; 1)$ if for any $x\ne x' \in S$, $\|x-x'\|_2 > \epsilon$. The $\epsilon$-packing number of $\cB_2(\RR^d; 1)$ is the largest cardinality of its any $\epsilon$-packing.} 
 	It is easy to see that   \eqref{lb_mu_1_diff} still holds whereas   \eqref{ub_mu_1_diff} is replaced by 
 	\[
 	\|\mu_1^{(i)} - \mu_{\ell} \|_2 \le (\sqrt{A}-1)\sqrt{\delta} + \eps \|\beta^{(i)}\|_2  \le  \sqrt{A\delta},
 	\]
 	so that $\theta^{(i)}\in \Theta_{\alpha,\delta, A}$ for all $0\le i\le N$.
 	For proving \eqref{lb_sw},   with the choice of $\mu_{\ell}$'s such that 
 	\[
 	\sqrt{3\delta/2} \le \min_{k\ne \ell} \|\mu_k - \mu_{\ell}\|_2\le \max_{k,\ell} \|\mu_k - \mu_{\ell}\|_2 \le   \sqrt{A\delta/2},
 	\]
 	one has that for any $k\ne \ell$,
 	\[
 	\|\mu_k - \mu_{\ell}\|_{\sw^{(i)}}^2 \le \|\mu_k - \mu_{\ell}\|^2/\lambda_p(\sw^{(i)}) \le A\delta 
 	\]
 	and 
 	\[
 	\|\mu_k - \mu_{\ell}\|_{\sw^{(i)}}^2 \ge \|\mu_k - \mu_{\ell}\|^2/\lambda_1(\sw^{(i)}) \ge \delta
 	\]
 	so that the constructed hypotheses belong to the specified parameter space. The rest of the proof uses the same argument as proving \cref{thm_lower_bound}.
 \end{proof}

\section{Auxiliary Lemma}\label{app_sec_aux}

    The following lemmas are proved in \cite{bing2023optimal}. 
    \begin{lemma}\label{lem_concen_pis}
        Assume $\pmin \ge 2\log n/n$. Then the following holds with probability at least $1-2n^{-1}$ for all $k\in [L]$,
        	\[ 
        	|\wt\pi_k - \pi_k^*| < \sqrt{16\pi_k (1-\pi_k^*)\log n\over n}.
        	\]
        	Furthermore, if $\pmin \ge C\log n/ n$ for some sufficiently large constant $C$, then 
        	\[
        	\PP\left\{
        	c\pi_k^* \le \wt \pi_k \le c'\pi_k^*
        	\right\} \ge 1-2n^{-1}.
        \]
    \end{lemma}
The following lemma states upper bounds of the quadratic form of a sub-Gaussian random vector \citep{hsu2012tail}.
\begin{lemma}\label{subG_quad_bd}
    Let  $ \xi \in \mathbb{R}^d$  be a subGaussian random vector with parameter $\gamma_{\xi}$. Then, for all symmetric positive semi-definite matrices $H$, and all $t \geq 0$,
\[
\mathbb{P} \left\{  \xi ^\top H \xi \geq \gamma_{\xi}^2 \left( \sqrt{\mathrm{tr}(H)} + \sqrt{2t \|H\|_{\mathrm{op}}} \right)^2 \right\} \leq e^{-t}.
\]
\end{lemma}
As an application of \cref{subG_quad_bd}, we have the following result.  
\begin{lemma}\label{Gau_l2_bd}
    Let $N_1, \dots, N_n$ be i.i.d. from $\cN_d(0,\bI_d)$. Then  for any $\delta \in (0,1)$,  
    \[
    \PP\left\{ \left\|{1\over n}\sum_{i=1}^n N_i \right\|_2 \ge \sqrt{ d\over n} + \sqrt{2\log(1/ \delta)\over n}\right\} \leq \delta.
    \] 
\end{lemma}
\begin{proof}
    The result follows from \cref{subG_quad_bd} by taking $\xi = n^{-1}\sum_{i=1}^n N_i $ with $\gamma_\xi = {1/\sqrt{n}}$.  
\end{proof}


The following lemma provides an upper bound on the operator norm of $\sum_{i\in S} N_iN_i^\T$ for any $S\in \{0,1\}^m$ with any $m\le n$. 
\begin{lemma}\label{lem_N_op_norm}
	For any $m\in \NN$ with $m\le n$, with probability at least $1-e^{-m}$, the following holds for all $S\subseteq [n]$  with $|S|\le m$
	\[
	\left\|\sum_{i\in S} N_iN_i^\T \right\|_\op \le ~ 2 
	d + 12m \log\left(en\over m\right). 
	\]
\end{lemma}
\begin{proof}
	Fix any $S$ and write $N_S$ as the $|S|\times d$ matrix containing $\{N_i\} _{i\in S}$. Let $\cN_{1/2}$ be the epsilon net of $\Sp^{|S|}$. Then
	\[
		 \left\|\sum_{i\in S} N_iN_i^\T \right\|_\op = \left\|N_SN_S^\T\right\|_\op = \sup_{v \in \Sp^{|S|}}  v^\T N_S N_S^\T v\le 2\max_{v\in \cN_{1/2}}  v^\T N_S N_S^\T v.
	\]
	Since $N_S^\T v \sim \cN_d(0,\bI_d)$, applying \cref{subG_quad_bd} with $\xi = N_S^\T v$ and $H = \bI_d$ gives 
	\[
		\PP\left\{
				v^\T N_S N_S^\T v  \ge \left(
				\sqrt{d} + \sqrt{2 t}
				\right)^2
		\right\} \le e^{-t},\qquad \forall ~ t\ge 0.
	\]
	Note that $|\cN_{1/2}|\le 5^{|S|}$ and there are at most ${n\choose m}\le (en/m)^m$ all possible $S \in \{0,1\}^m$. Taking the union bounds over   $v\in \cN_{1/2}$ and $S$ and choosing $t = 3m \log(en/m)$ gies the result. 
\end{proof}

The following lemmas state the well-known vector-valued and matrix-valued Bernstein inequalities. See, for instance, \citet[Theorem 3.1, Corollary 3.1 and Corollary 4.1]{Minsker2017}. 
\begin{lemma}[Vector-valued Bernstein inequality]\label{lem_bernstein_vector}
	Let $X_1,\ldots, X_n$ be independent random vectors with zero mean and $\max_{i\in [n]}\|X_i\|_2 \le U$ almost surely. Denote $\sigma^2 := \sum_{i=1}^n \EE[\|X_i\|_2^2]$. Then for all $t \ge {1\over 6}(U + \sqrt{U^2 + 36\sigma^2})$, 
	\[
	\PP\left(
	\left\|\sum_{i=1}^n  X_i\right\|_2 > t
	\right) \le 28\exp\left(
	- {t^2 / 2\over \sigma^2 + Ut/3}
	\right).
	\]
\end{lemma}

\begin{lemma}[Matrix-valued Bernstein inequality]\label{lem_bernstein_mat}
	Let $X_1,\ldots, X_n \in \RR^{d\times d}$ be independent, symmetric random matrices with zero mean and $\max_{i\in [n]}\|X_i\|_\op \le U$ almost surely. Denote $\sigma^2 := \|\sum_{i=1}^n \EE[X_i^2]\|_\op$. Then for all $t \ge {1\over 6}(U + \sqrt{U^2 + 36\sigma^2})$, 
	\[
	\PP\left(
	\left\|\sum_{i=1}^n  X_i\right\|_\op > t
	\right) \le 14   \exp\left(
	- {t^2 / 2\over \sigma^2 + Ut/3} + \log(d)
	\right).
	\]
\end{lemma}

Another useful concentration inequality of the operator norm of the random matrices with i.i.d. sub-Gaussian rows is stated in the following lemma \cite[Lemma 16]{bing2020prediction}. This is an immediate result of \citet[Remark 5.40]{vershynin_2012}.

\begin{lemma}\label{lem_op_diff} 
	Let $G$ be $n$ by $d$ matrix whose rows are i.i.d. $\gamma$ sub-Gaussian  random vectors with covariance matrix $\Sigma_G$. Then, for every $t\ge 0$, with probability at least  $1-2e^{-ct^2}$,
	\[
	\left\|	{1\over n}G^\T G - \Sigma_G\right\|_{\op}\le \max\left\{\delta, \delta^2\right\} \left\|\Sigma_G\right\|_{\op},
	\]
	with $\delta = C\sqrt{d/n}+ t/\sqrt n$ where $c,C$ are positive constants depending on $\gamma$.
\end{lemma}

\end{document}